\documentclass[11pt,reqno]{amsart}
\usepackage{mathrsfs}
\usepackage{physics}
\usepackage{url}
\usepackage{mathtools}
\usepackage{latexsym,epsfig,amssymb,amsmath,amsthm,color,url,bm}
\usepackage[inline,shortlabels]{enumitem}
\usepackage{hyperref}
\usepackage[foot]{amsaddr}
\usepackage{amsmath,amsbsy}
\usepackage{mwe}
\RequirePackage[numbers]{natbib}
\usepackage{mathptmx}
\usepackage[text={16cm,24cm}]{geometry}
\usepackage{xcolor}

\allowdisplaybreaks 
 \numberwithin{equation}{section}
\newtheorem{theorem}{Theorem}[section]

\newtheorem{lemma}{Lemma}
\newtheorem{corollary}{Corollary}

\newcommand\Item[1][]{%
  \ifx\relax#1\relax  \item \else \item[#1] \fi
  \abovedisplayskip=0pt\abovedisplayshortskip=0pt~\vspace*{-\baselineskip}}

\theoremstyle{definition}

\theoremstyle{definition}
\newtheorem{remark}{Remark}

\newcommand{\sr}{\textcolor{red}}

\DeclareMathOperator{\Prob}{\mathbf{P}}
\DeclareMathOperator{\E}{\mathbf{E}}

\title[]{Percolation games on rooted, edge-weighted random trees}
\date{}
\author{Sayar Karmakar, Moumanti Podder, Souvik Roy, Soumyarup Sadhukhan}
\address{Sayar Karmakar, University of Florida, 230 Newell Drive, Gainesville, Florida 32605, USA.}
\address{Moumanti Podder, Indian Institute of Science Education and Research (IISER) Pune, Dr.\ Homi Bhabha Road, Pashan, Pune 411008, Maharashtra, India.}
\address{Souvik Roy, Indian Statistical Institute, 203 Barrackpore Trunk Road, Kolkata 700108, West Bengal, India.}
\address{Soumyarup Sadhukhan, Indian Institute of Technology, Kalyanpur, Kanpur, Uttar Pradesh 208016, India.}
\email{sayarkarmakar@ufl.edu}
\email{moumanti@iiserpune.ac.in}
\email{souvik.2004@gmail.com}
\email{soumyarup.sadhukhan@gmail.com}

\begin{document}
\bibliographystyle{plainnat}

\begin{abstract}
 Consider a rooted Galton-Watson tree $T$, to each of whose edges we assign, independently, a weight that equals $+1$ with probability $p_{1}$, $0$ with probability $p_{0}$ and $-1$ with probability $p_{-1}=1-p_{1}-p_{0}$. We play a game on a realization of this rooted, edge-weighted Galton-Watson tree, involving two players and a token. The token is allowed to be moved from where it is currently located, say a vertex $u$ of $T$, to any child $v$ of $u$. The players begin with initial capitals that amount to $i$ and $j$ units respectively, and a player wins if 
\begin{enumerate*}
\item either she is the first to amass a capital worth $\kappa$ units, where $\kappa$ is a prespecified positive integer,
\item or she is able to move the token to a leaf vertex, from where her opponent cannot move it any farther,
\item or her opponent's capital is the first to dwindle to $0$.
\end{enumerate*}
The game may continue indefinitely, with neither player being able to clinch a victory, in which case we say that the game results in a draw. This paper is concerned with analyzing the probabilities of the three possible outcomes such a game may culminate in, namely a win for the first player (i.e.\ the player who plays the first round of the game), a loss for the first player, and a draw for both players, as well as with finding conditions under which the expected duration of this game is finite. Of particular interest to us is the exploration of criteria that guarantee the probability of draw in such a game to be $0$. The theory we develop in this paper for the analysis of this game is further supported by observations obtained via computer simulations, and these observations provide a deeper insight into how the above-mentioned probabilities behave as the underlying parameters and / or offspring distributions are allowed to vary. We include in this paper conjectures pertaining to the behaviour of the probability of draw in our game (including a phase transition phenomenon, in which the probability of draw goes from being $0$ to being strictly positive) as the parameter-pair $(p_{0},p_{1})$ is varied suitably while keeping the underlying offspring distribution of $T$ fixed. \end{abstract}

\maketitle

\section{Introduction}\label{sec:intro}
To each edge of any realization of a \emph{rooted Galton-Watson tree} $T$, we assign, independently, a weight that equals $+1$ with probability $p_{1}$, $0$ with probability $p_{0}$, and $-1$ with probability $p_{-1}=(1-p_{1}-p_{0})$. A \emph{percolation game} is then played on this \emph{rooted, edge-weighted tree}, involving two players and a token. The players alternate in making \emph{moves}, where a \emph{move} permits the token to be displaced from where it is currently located, say a vertex $u$ of $T$, and to be situated on any child $v$ of $u$ (if at least one such child exists). When the token is moved from $u$ to $v$, the player who performs this move is rewarded with an amount equal to the edge-weight assigned to the edge connecting $u$ with $v$. The players begin with \emph{initial capitals} worth $i$ and $j$ units respectively, with $i, j \in \mathbb{N}$, and a player wins if 
\begin{enumerate*}
\item either she is the first to amass a capital worth $\kappa$ units, where $\kappa$ is a pre-specified positive integer,
\item or she is able to move the token to a leaf vertex, from where her opponent cannot move it any farther,
\item or her opponent is the first to have her capital dwindle to $0$.
\end{enumerate*}
The game may continue indefinitely, with neither player being able to clinch a victory, in which case we say that the game results in a draw. This paper is concerned with studying the probabilities of the three possible outcomes (i.e.\ win for the first player, loss for the first player, and draw for both players) of this game, necessary and / or sufficient criteria that guarantee that the probability of draw equals $0$, as well as finding conditions under which the expected duration of this game is finite.

The primary motivation for this work can be attributed to \cite{holroyd2021galton}, in which \emph{normal}, \emph{mis\`{e}re} and \emph{escape games} have been studied on rooted Galton-Watson trees. Each of these games is a \emph{two-player combinatorial game} involving two players and a token, just as it is the case for the percolation games described in the previous paragraph. The permitted moves in each of these games are the same as those for the percolation games studied in this paper, i.e.\ the token may be moved from the vertex where it is currently situated, to any one of its children. The normal game is lost by whichever of the two players is unable to move the token any farther (i.e.\ her opponent has succeeded in moving the token to a leaf vertex), whereas the mis\`{e}re game is won by whichever of the two players is the first to become stagnated at a leaf vertex. The escape game is played between players titled \emph{Stopper} and \emph{Escaper}, and Stopper wins if \emph{either} of the two players becomes confined to a leaf vertex.

A natural extension of the idea of normal games presented above can be achieved by simply labeling each edge of the rooted Galton-Watson tree $T$, independently, a \emph{trap} with some probability $p_{-1}$, and \emph{open} or \emph{safe} with the remaining probability $(1-p_{-1})$. The moves permitted remain the same as before, and a player now loses if her opponent either moves the token to a leaf vertex (thus preventing her from being able to move the token any farther, thereafter), or if she is forced to move the token \emph{along} an edge that has been labeled a trap (which happens only if her opponent moves the token to some vertex $u$ of $T$ such that, for \emph{every} child $v$ of $u$, the edge between $u$ and $v$ has been labeled a trap). This is the same as pruning from $T$ every edge that has been labeled a trap (thus turning a vertex $u$ into a leaf vertex whenever the edge between $u$ and $v$ is a trap for \emph{every} child $v$ of $u$), and \emph{then} playing the normal game on this pruned tree. A further generalization is attained if we now allow each edge of $T$ to be labeled, independently, a \emph{trap} with probability $p_{-1}$, a \emph{target} with probability $p_{1}$, and \emph{safe} or \emph{open} with the remaining probability $p_{0}=(1-p_{-1}-p_{1})$. A player now may win in one of three possible ways:
\begin{enumerate*}
\item either by moving the token along an edge labeled a target,
\item or forcing her opponent to move the token along an edge labeled a trap,
\item or preventing her opponent from moving any farther, by placing the token at a leaf vertex of $T$.
\end{enumerate*}
This game, too, can be simplified to a normal game by modifying the underlying tree $T$ as follows: we remove every edge of $T$ that has been labeled a trap, and for every edge $(u,v)$ (between a vertex $u$ and its child $v$) that has been labeled a target, we remove the entire subtree of $T$ induced on $v$ and its descendants. 

As a consequence of the discussion in the previous paragraph, one may wonder how a true generalization / extension of the normal games studied in \cite{holroyd2021galton} can be achieved -- one that cannot be reduced to the relatively simple set-up of normal games no matter how we tweak the underlying tree $T$. To this end, we ask: what happens if we allow the game rules to be lenient enough that a \emph{single} traversal of an edge that is not safe / open does not \emph{necessarily} cause the game to come to an end? In other words, what if a player must be forced to make the mistake of traversing trap-edges a ``sufficient" number of times, or her opponent must succeed in moving the token along target-edges a ``sufficient" number of times, in order for her to be declared the loser? This is a rather natural and spontaneous extension, and we formalize this by replacing the label of \emph{trap} with the edge-weight $-1$, the label of \emph{target} with the edge-weight $+1$, and the label of \emph{safe} or \emph{open} with the edge-weight $0$. The game now continues until one of the players has squandered away her entire capital, or one of them has amassed a capital worth $\kappa$ units, where $\kappa \in \mathbb{N}$ is prespecified (or if one of them becomes stranded at a leaf vertex). This is how the investigation that began with \cite{holroyd2021galton} has paved the way for the work accomplished in this paper.

We now summarize, informally and briefly, some of our main contributions in this paper. To the best of our knowledge, the set-up we consider in this paper has not been addressed in the existing literature, and, as mentioned above, this is one of the most natural ways one may seek to generalize the set-up studied in \cite{holroyd2021galton}. Theorem~\ref{thm:1} gives a complete characterization of the probabilities of the various outcomes that our game may culminate in -- this is done by describing these probabilities as coordinates of \emph{specific fixed points} of \emph{specific multivariable functions} (that are defined in terms of $p_{-1}$, $p_{1}$, $p_{0}$ and the offspring distribution underlying the rooted Galton-Watson tree on which the game is being played). While the representation of the probabilities of the various outcomes of a two-player combinatorial game, played on a random board / random rooted tree, in terms of the fixed point(s) of a function is not surprising in itself (for instance, such results can be found in \cite{holroyd2021galton}), the novelty lies in the fact that in this paper, the far more generalized set-up forces us to consider the fixed points of a multivariable function (in fact, we visualize this function, $h$, as defined in \eqref{fixed_point_eq}, as mapping from $\mathcal{S}$ to $\mathcal{S}$, where $\mathcal{S}$ is the space of all $(\kappa-1)\times (\kappa-1)$ matrices in which each entry is in the interval $[0,1]$). Determining whether such a function has a unique fixed point in $\mathcal{S}$ (which happens if and only if the probability of draw equals $0$ for the game that begins with initial capital amounts $i$ and $j$ units, for each $i,j \in \{1,2,\ldots,\kappa-1\}$) or not turns out to be a far harder task than when we consider single-variable functions. If we let $w_{i,j}$ (respectively, $\ell_{i,j}$) denote the probability of the event that the first player wins (respectively, loses) the game that begins with an initial capital of $i$ units for the first player and an initial capital of $j$ units for the second player, then the statement of Theorem~\ref{thm:1} as well as the recurrence relations arising from our game rules indicate that all $w_{i,j}$ and $\ell_{i,j}$, for $i,j \in \{1,2,\ldots,\kappa-1\}$, are very closely related to each other. As a further testimony to this, we see that Theorem~\ref{thm:2} connects all $d_{i,j}$s, where $d_{i,j}$ is the probability of the event that the game that begins with the first player possessing an initial capital of $i$ units and the second player possessing an initial capital of $j$ units, results in a draw. For instance, we show that when each of $p_{-1}$, $p_{1}$ and $p_{0}$ is strictly positive, either $d_{i,j}=0$ for all $i,j \in \{1,2,\ldots,\kappa-1\}$, or $d_{i,j}>0$ for all $i,j \in \{1,2,\ldots,\kappa-1\}$. It is not possible to observe such a fascinating phenomenon in simpler set-ups such as that considered in \cite{holroyd2021galton}, and this opens up the possibility of exploring phase transition phenomena pertaining to the probabilities of draw from a different perspective. Theorem~\ref{thm:3} presents sufficient conditions ensuring that the expected duration of our game is finite. Theorem~\ref{thm:kappa=2} addresses the case of $\kappa=2$, and presents not only very neat results, but also a complete picture describing the phase transition phenomenon pertaining to the probability of draw (i.e.\ how the probability of draw, $d_{1,1}$, goes from being $0$ to being strictly positive as the parameters $p_{-1}$ and $p_{1}$ are allowed to vary), under various commonly studied offspring distributions. Such neat results are hard to obtain when $\kappa \geqslant 3$ is considered, due to the lack of existing theory (to the best of our knowledge) on fixed points of multivariable functions. Despite this difficulty, Theorem~\ref{thm:kappa=3} and Theorem~\ref{thm:kappa=3_special} provide sufficient criteria under which the probabilities of draw (i.e.\ $d_{i,j}$, for $i,j \in \{1,2\}$) equal $0$ when the target capital is $\kappa=3$ (while Theorem~\ref{thm:kappa=3} does this for arbitrary values of $p_{-1}$, $p_{1}$ and $p_{0}$ and for general offspring distributions, Theorem~\ref{thm:kappa=3_special} does this when $p_{-1}$, $p_{0}$ and $p_{1}$ obey certain constraints, and assumptions are made on the offspring distribution under consideration).     

\subsection{Organization of the rest of the paper} We conclude \S\ref{sec:intro} with a description of how the rest of the paper has been organized. We motivate the study of percolation games on rooted, edge-weighted Galton-Watson trees in \S\ref{sec:further_motivations}. Although the games and the set-up (or `random board') on which they are being played have already been informally introduced above, we provide a formal introduction to these in \S\ref{sec:twogames}, with \S\ref{subsec:deterministic_game} dedicated to describing the game on a \emph{deterministic} rooted tree with \emph{fixed} edge-weights, and \S\ref{subsec:random_game} dedicated to describing the game when it is played on a \emph{rooted Galton-Watson tree} with \emph{random} edge-weights assigned, independently, to all the edges. The main results of this paper are stated in \S\ref{subsec:main_results}, and since an application of many of these results requires verifying whether a (somewhat complicated) multivariable function has a unique fixed point in a given domain or not -- a feat best achieved via analyzing this function using a computer -- this is followed by \S\ref{sec:simulations}, where we present our findings via numerical simulations in the form of tables. In particular, in \S\ref{sec:simulations}, we numerically estimate the probabilities of draw, $d_{i,j}$ (of the game in which the first player begins with an initial capital of amount $i$ units, and the second player begins with an initial capital of $j$ units), by making use of Theorem~\ref{thm:1}, for various values of $p_{0}$ and $p_{1}$ and for various underlying offspring distributions, and try to glean a pattern from the observed values regarding the occurrence of a phase transition phenomenon (where each $d_{i,j}$ goes from being $0$ to being strictly positive). We also present tables showing various instances (for various values of $p_{0}$ and $p_{1}$ and for a couple of offspring distributions) where Theorem~\ref{thm:kappa=3} is applicable, and some instances where it is not, but nonetheless, these instances reveal fascinating patterns regarding when one may expect each $d_{i,j}$ to be strictly positive. Finally, the detailed proofs of our results are presented in \S\ref{sec:formal_proofs}, with \S\ref{subsec:notations_definitions} summarizing the notations used in the rest of the paper as well as some fundamental observations / results needed for the analysis of our game, \S\ref{subsec:recurrence} explaining the recurrence relations that arise from our game rules, \S\ref{subsec:thm:1_proof} outlining the proof of Theorem~\ref{thm:1}, \S\ref{sec:proof_thm_2} outlining the proof of Theorem~\ref{thm:2}, \S\ref{sec:proof_thm_3} containing the proof of Theorem~\ref{thm:3}, \S\ref{sec:kappa=2_proofs} containing the proof of Theorem~\ref{thm:kappa=2}, and \S\ref{sec:kappa=3_proofs} outlining the proofs of Theorems~\ref{thm:kappa=3} and \ref{thm:kappa=3_special}.   

\section{Literature Review and Motivations for studying the weighted-tree-percolation games}\label{sec:further_motivations}
The weighted-tree-percolation game can be interpreted as a simplistic model for \emph{oligopolistic competitions}. Imagine two companies, denoted $C_{1}$ and $C_{2}$, locked in an oligopolistic competition, in a market that does not permit the entry of a third company. It is assumed that a company is forced to concede defeat if it goes bankrupt, i.e.\ the short term capital loss that it sustains equals or exceeds the initial capital amount it entered the competition with, and a company is assumed to emerge decisively as the winner, eliminating its opponent from the race (i.e.\ by becoming a monopoly), if it becomes the first among the two to accumulate a net capital that is beyond a certain pre-prescribed margin. We assume that this competition unfolds on a rooted tree, $t$, each of whose edges has been assigned a weight: this weight amounts to a profit or reward when it is positive, to a loss or penalty when it is negative, and it serves as neither profitable nor damaging when it equals $0$. Traversing an edge by a company adds to its current capital the weight associated with that edge. Since in an oligopolistic market, the decision made by each company is bound to impact the other, we can imagine that a company traversing an edge $(u,v)$ of $t$ is equivalent to an \emph{action} or \emph{move} performed (by that company), and the subtree $t_{v}$ of $t$, induced on the subset of vertices of $t$ comprising $v$ and its descendants, along with the weights assigned to the edges of $t_{v}$, is the \emph{state} of the game after the above-mentioned action has been completed. Following the notion of sequential games, we assume that each company waits for its opponent to finish making their move, before deciding on a subsequent action of their own. The game may continue indefinitely, without either company being able to establish itself as the decisive superior, in which case we say that the game has resulted in a draw.

A second theoretical motivation for studying these games lies in how they incorporate an \emph{adversarial element} into the otherwise random process of \emph{percolation}. Percolation is a topic of immense interest spanning physics, chemistry and mathematics, and extensive research has been conducted in this area (see \cite{broadbent1957percolation} that introduced the notion of percolation, as well as \cite{grimmett1999percolation}, \cite{stauffer2018introduction}, \cite{kesten1982percolation}, \cite{shante1971introduction} and \cite{toom2001contours}, to name just a few). Two types of percolation are usually considered, namely \emph{site percolation} and \emph{bond percolation}. In the former, each \emph{vertex} of a given graph is assigned, independently, a label that reads \emph{closed} with some probability $p$, and \emph{open} with the remaining probability $1-p$. In the latter, such a labeling is applied, instead, to the \emph{edges} of a given graph (often, the edges of such a graph are \emph{directed}, in which case we refer to the process as \emph{oriented bond percolation}). A collection of vertices of the given graph is specified to be the \emph{source}, and another collection, disjoint from the source, is specified to be the \emph{sink}. The primary question asked when studying bond percolation is: for what values of $p$ (and of any underlying parameter(s) that influence the structure of the graph under consideration) does there exist, with positive probability, a path consisting of open edges leading \emph{from} the source \emph{to} the sink? This class of questions includes asking, when studying bond percolation on a rooted infinite tree with each edge directed \emph{away} from the root (and the root being the source): for what values of $p$ does there exist, with positive probability, an infinite path that starts at the root and consists only of open edges? 

The understanding of the phenomenon of percolation (and its variants) on trees and tree-like graphs is rather different from that on infinite lattice graphs. This paragraph enumerates only a handful of the long list of studies done on percolation, and its variants, on (both deterministic and random) trees. In \cite{lyons1990random}, the notion of the \emph{branching number} of a rooted tree (in some sense, the ``average" number of branches sent forth by each vertex of the tree) is defined, and it is shown to be fundamental in the study of probabilistic processes, including bond percolation, on rooted trees. In \cite{lyons1992random}, an exact correspondence between random walks and percolation on rooted trees is proved, and this correspondence is utilized to determine whether certain random walks on random trees are transient or recurrent. In \cite{pemantle1992random}, the branching number is shown to determine the rate of first-passage percolation on trees. In \cite{lyons1989ising}, a general percolation process on trees, termed \emph{quasi-Bernoulli percolation}, is introduced, and it is utilized in studying \emph{spin percolation} in the \emph{ferromagnetic Ising model}, in the absence of any external magnetic field, on trees, and in studying \emph{Bernoulli percolation} on \emph{tree-like graphs}. In \cite{haggstrom1997infinite}, \emph{dependent bond percolation} on the \emph{homogeneous tree} of order $n$ (i.e.\ the connected, cycle-free infinite graph in which each vertex has degree $n+1$) is studied under the assumption of \emph{automorphism invariance}. In \cite{baur2016percolation}, \emph{Bernoulli bond percolation} on a \emph{random recursive tree} of size $n$, with the percolation parameter $p(n)$ converging to $1$ as $n$ approaches $\infty$, is studied. \emph{Dynamical percolation} is introduced and studied in \cite{olle1997dynamical}, \cite{peres1998number}, \cite{khoshnevisan2008dynamical} and \cite{steif2009survey}. Two other variants of percolation, namely \emph{invasion percolation} and \emph{accessible percolation}, of which the latter is inspired by evolutionary biology, have been studied in \cite{angel2008invasion} and \cite{nowak2013accessibility} respectively.

\sloppy We include here a brief discussion of some of the two-player combinatorial games (such as the weighted-tree-percolation game itself), studied in the literature, that have been inspired by percolation. The notion of \emph{site percolation games} was introduced by \cite{holroyd2019percolation}, as follows: each vertex $(x,y) \in \mathbb{Z}^{2}$ is marked, independent of all else, a \emph{trap} with probability $p$, a \emph{target} with probability $q$, and \emph{open} with the remaining probability $1-p-q$. Two players take turns to make \emph{moves}, where a \emph{move} involves relocating a token from where it is currently located, say $(x,y)$, to one of $(x+1,y)$ and $(x,y+1)$. A player wins if she is either able to move the token to a vertex labeled a target, or force her opponent to move the token to a vertex labeled a trap. A variant of this game is studied in \cite{bhasin2022class}, with the token now permitted to be moved from where it is currently located, say $(x,y)$, to any one of $(x,y+1)$, $(x+1,y+1)$ and $(x+2,y+1)$. In each of \cite{holroyd2019percolation} and \cite{bhasin2022class}, it is shown, via the technique of \emph{weight functions} or \emph{potential functions}, that the probability of draw in the game under consideration is $0$ whenever $p+q > 0$. In \cite{bhasin2022ergodicity}, a somewhat more general version of the game studied in \cite{holroyd2019percolation} is investigated, with the token now permitted to be moved from $(x,y)$ to any of $(x,y+1)$ and $(x+1,y+1)$ when $x$ is even, and to any of $(x+1,y+1)$ and $(x+2,y+1)$ when $x$ is odd. In \cite{bhasin2024bond}, \emph{bond percolation games} on the $2$-dimensional square lattice have been investigated. Each edge of the lattice, which is either between $(x,y)$ and $(x+1,y)$ or between $(x,y)$ and $(x,y+1)$, for some $(x,y) \in \mathbb{Z}^{2}$, is labeled, independently, a trap with probability $p$, a target with probability $q$, and open with probability $(1-p-q)$, and the moves permitted are the same as those in \cite{holroyd2019percolation}. The primary question attempted to answer is: for what values of $(p,q)$ is the probability of draw in this game equal to $0$? Finally, in \cite{basu2016trapping}, another two-player combinatorial game is studied on percolation clusters of Euclidean lattices, in which a move involves relocating the token from where it is currently located, to a neighbour of that vertex, along an edge that has not been previously visited, and whichever player is unable to move any farther, loses.

The extensively studied \emph{Maker-Breaker positional games} (see, for instance, \cite{stojakovic2005positional}, \cite{muller2014threshold} and \cite{nenadov2016threshold}, to name just a few) paved the way for the introduction of the \emph{Maker-Breaker percolation games} in \cite{day2021maker} and \cite{day2021makerescaping}. In \cite{day2021maker}, given $b, m \in \mathbb{N}$, and an $M \times N$ rectangular grid-graph (with $M$ vertices in each row and $N$ vertices in each column), the player called Maker, in each of her turns, chooses $m$ of the as-yet-unclaimed edges of this board and marks them safe, while her opponent, called Breaker, in each of her turns, chooses $b$ of the as-yet-unclaimed edges of this board and deletes them. Maker clinches a victory if she is able to mark as safe all the edges constituting a path connecting the left margin of this board to its right. In \cite{day2021makerescaping}, given an infinite connected graph $\Lambda$ and a specific vertex $v_{0}$ in it, once again, Maker, in each of her turns, chooses $m$ of the as-yet-unclaimed edges of $\Lambda$ and marks them safe, while Breaker, in each of her turns, chooses $b$ of the as-yet-unclaimed edges of $\Lambda$ and deletes them. Breaker wins if during any round of this game, the component of $\Lambda$ containing $v_{0}$ becomes finite. In \cite{dvovrak2021maker}, the Maker-Breaker percolation game with $m=b=1$ (termed the \emph{unbiased} version) is studied on the random premise that is obtained after the usual bond percolation process with parameter $p$ has been performed on the edges of the $2$-dimensional square lattice $\mathbb{Z}^{2}$. The same Maker-Breaker percolation game as the one investigated in \cite{day2021maker} is addressed in \cite{wallwork2022maker}, but on a triangular grid-graph that consists of $N$ rows, with $M$ and $M-1$ vertices per row alternating upward. This is a relatively new area in which many fascinating questions remain open, and delving deeper into bond percolation games, as well as their generalizations such as the weighted-tree-percolation games, may very well establish a connection via which this area can be explored.

\section{Model Description}\label{sec:twogames}
\subsection{The bond percolation game on a deterministic edge-weighted tree}\label{subsec:deterministic_game}
Let $t$ denote a rooted (finite or infinite) tree in which each edge has been assigned a \emph{weight} (which takes an integer value). We let $\phi(t)$ denote the root of $t$, and we let $V(t)$ denote its set of vertices. Given vertices $u, v \in V(t)$ with $u$ the parent of $v$, we imagine the edge between $u$ and $v$ to be directed \emph{from} $u$ \emph{towards} $v$, and henceforth, denote this directed edge by $(u,v)$. We let $\omega_{t}(u,v)$ denote the weight assigned to $(u,v)$. In this paper, we shall only be concerned with $\omega_t(u,v) \in \{1,0,-1\}$, for all directed edges $(u,v)$ of $t$.

Our object of interest in this paper is the \emph{bond percolation game on a rooted edge-weighted tree}, henceforth abbreviated as the \emph{weighted-tree-percolation game}, played, as the name suggests, on a rooted edge-weighted tree. Such a game is indicated by $\mathcal{G}(i,j,\kappa,t)[v^*]$, where 
\begin{enumerate}
\item $i,j \in \mathbb{N}_{0}$ and $\kappa \in \mathbb{N}$ are the underlying parameters, with 
\begin{equation}
(i,j) \in \mathcal{R}, \text{ where } \mathcal{R} = \{0,1,\ldots,\kappa\}^{2} \setminus \{(0,0), (\kappa,\kappa)\},\label{i,j_domain}
\end{equation} 
\item $t$ is the rooted edge-weighted tree on which the game is played, 
\item and $v^* \in V(t)$ is a vertex of the tree $t$, referred to as the \emph{initial vertex}.
\end{enumerate}
Here and elsewhere, we use $\mathbb{N}$ to indicate the set of all positive integers, and $\mathbb{N}_{0}$ to indicate the set of all non-negative integers. The game involves 
\begin{enumerate*}
\item two players who are henceforth referred to as $P_{1}$ and $P_{2}$, 
\item and a single token that is placed at $v^*$ just before the game begins.
\end{enumerate*}
At the start of the game, $P_{1}$ owns a capital worth $i$ units of money, whereas $P_{2}$ owns a capital worth $j$ units of money. They take turns to make \emph{moves}, with $P_{1}$ assumed to make the first move, where a \emph{move} (by either player) involves relocating the token from where it is currently situated, say a vertex $u \in V(t)$, to any child $v$ of $u$ (if such a child exists). Each move comes with a \emph{reward}: whichever player moves the token from a vertex $u$ to a child $v$ of $u$ is awarded an amount equal to the edge-weight $\omega_t(u,v)$. A player wins if
\begin{enumerate}
	\item either she is the first of the two to amass a total capital worth $\kappa$ units of money, 
	\item or her opponent is the first to have her capital dwindle to $0$,
	\item or she is able to move the token to a leaf vertex (from where her opponent is unable to move the token any farther).
\end{enumerate} 
The game continues for as long as neither of the two players emerges victorious, and this may happen indefinitely, leading to a draw. The \emph{duration} of the game either indicates the number of moves that occur until the game ends (with one of the two players winning and the other losing), or else it equals $\infty$ (which happens when the game results in a draw).

It is worthwhile to note here that, were $P_{2}$ to make the first move (or, in other words, play the first round of the game), the corresponding game can simply be referred to as $\mathcal{G}(j,i,\kappa,t)[v^*]$. It is immediate that if either of $i$ and $j$ were to equal $\kappa$ or $0$, the fate of the game is decided even before the game begins, i.e.\ if $i=\kappa$ (which implies, from \eqref{i,j_domain}, that $j \in \{0,1,\ldots,\kappa-1\}$), then $P_{1}$ wins right away, whereas if $i=0$ (which implies, from \eqref{i,j_domain}, that $j \in \{1,2,\ldots,\kappa\}$), then $P_{1}$ loses right away; likewise, if $j=\kappa$, $P_{2}$ wins right away, and if $j=0$, $P_{2}$ loses right away.

We assume that, if the game is destined to end in a finite number of rounds, the player who wins tries to wrap up the game as quickly as possible, whereas the player who loses tries to prolong its duration as much as possible.\footnote{It should be noted here that the rooted edge-weighted tree $t$ is revealed in its entirety to both the players before the game begins (this becomes particularly relevant when the game is played on a suitably defined \emph{random} rooted edge-weighted tree), which makes it a \emph{perfect information game}.} A formal game theoretic description of this game can be accomplished in a manner analogous to that for the \emph{site percolation game} studied (on an infinite lattice graph whose set of vertices is $\mathbb{Z}^{2}$) in \cite{bhasin2022class} (see \S 5 of \cite{bhasin2022class}).

\subsection{The bond percolation game on a random edge-weighted tree}\label{subsec:random_game}
We begin by recalling the well-known, extensively studied \emph{rooted Galton-Watson (GW) branching process} that is defined as follows (see \cite{watson1875probability} which introduced this branching process as a model for analyzing the extinction of family lines, and \cite{athreya2012branching} for a detailed exposition on the properties and important variants of this branching process). Given a root-vertex $\phi$ and a probability distribution $\chi$ that is supported on $\mathbb{N}_{0}$ (or a subset thereof), henceforth referred to as the \emph{offspring distribution}, let $\phi$ have $X_{0}$ children with $X_{0}$ distributed as $\chi$. If $X_{0}=0$, the process stops. Conditioned on $X_{0}=m$ for some $m \in \mathbb{N}$, we name these $m$ children $u_{1}, u_{2}, \ldots, u_{m}$, and we let $u_{i}$ have $X_{i}$ children, with $X_{1}, X_{2}, \ldots, X_{m}$ i.i.d.\ $\chi$. We continue thus, generating a rooted (random) tree that is indicated by $T_{\chi}$ (we drop the subscript $\chi$ whenever the offspring distribution $\chi$ is fixed \emph{a priori}). It is a well-known fact that $T_{\chi}$ exhibits a positive probability of being infinite if and only if the expectation of $\chi$ is strictly greater than $1$. 

As in \S\ref{subsec:deterministic_game}, we let $V(T)$ denote the set of vertices of $T$, and given $u, v \in V(T)$ with $u$ the parent of $v$, we indicate by $(u,v)$ the edge between $u$ and $v$, directed from $u$ towards $v$. We also let $G$ denote the probability generating function corresponding to $\chi$, i.e.\
\begin{equation}
G(x) = \sum_{m=0}^{\infty}x^{m}\chi(m), \text{ for } x \in [0,1],\label{pgf}
\end{equation}
where $\chi(m)$ is the probability of the event that $\phi$ has $m$ children, for each $m \in \mathbb{N}_{0}$.

It is important to mention here that we make the following (very weak) assumption regarding $\chi$: that $\chi(0) < 1$, which is equivalent to the assumption that there exists at least one $m \in \mathbb{N}$ with $\chi(m) > 0$ (the complement of this criterion, i.e.\ $\chi(0)=1$, results in a rooted tree with the root as the only vertex). Note that this assumption makes $G$ a strictly increasing function on $[0,1]$.

The definition of the weighted-tree-percolation game, as presented in \S\ref{subsec:deterministic_game}, readily extends to the premise of \emph{random} rooted edge-weighted trees. Here, we slightly abuse the notation introduced above and let $T$ (or $T_{\chi}$, when we need to emphasize on the offspring distribution $\chi$) indicate the rooted GW tree in which each directed edge $(u,v)$ has been assigned, \emph{independently}, a \emph{random} edge-weight $\omega_T(u,v)$ distributed as follows:
\begin{equation}
\omega_T(u,v) = i \text{ with probability } p_{i}, \text{ for } i \in \{-1,0,1\},\label{edge_weight_prob}
\end{equation}
where $p_{i} \in [0,1]$ for each $i \in \{-1,0,1\}$ and $p_{-1}+p_{0}+p_{1}=1$. The corresponding weighted-tree-percolation game played on $T$ is indicated by $\mathcal{G}(i,j,\kappa,T)[v^*]$ -- as before, $i$ and $j$ denote, respectively, the initial capitals that $P_{1}$ and $P_{2}$ begin playing with (and $P_{1}$ plays the first round), $\kappa$ is the amount that, if amassed by a player, results in her victory, and $v^{*}$ is the initial vertex (in $V(T)$) at which the token is placed just before the game begins. We emphasize here that the realization of $T$ (including all its edge-weights) is revealed in its entirety to both $P_{1}$ and $P_{2}$ before the game begins, and that the assumption stated at the very end of \S\ref{subsec:deterministic_game} is imposed in this case as well.

\section{Main results of this paper}\label{subsec:main_results}
The primary goal of this paper is to understand the probabilities of the various possible outcomes of the weighted-tree-percolation game described in \S\ref{subsec:random_game}. These outcomes are:
\begin{enumerate*}
\item a win for $P_{1}$ (assumed to play the first round of the game),
\item a loss for $P_{1}$ (equivalently, a win for $P_{2}$), and
\item a draw for both the players.
\end{enumerate*}
The randomness of the underlying premise on which the game is being played, which gives rise to the probabilities mentioned above, is \emph{two-fold}:
\begin{enumerate*}
\item because of the measure induced by the GW branching process (as described at the very beginning of \S\ref{subsec:random_game}),
\item and because of the measure induced by the i.i.d.\ edge-weights assigned to the edges of the tree generated by the GW branching process, according to \eqref{edge_weight_prob}.
\end{enumerate*}

Having fixed $\kappa$, the offspring distribution $\chi$ of $T=T_{\chi}$, and the edge-weight probabilities $p_{-1}$, $p_{0}$ and $p_{1}$ given by \eqref{edge_weight_prob}, let $w_{i,j}$, for each $i, j \in \{1,2,\ldots,\kappa-1\}$, indicate the probability of the event that $P_{1}$ wins the game $\mathcal{G}(i,j,\kappa,T)[\phi]$, where $\phi$ is the root of $T$. Likewise, let $\ell_{i,j}$ denote the probability that $P_{1}$ loses $\mathcal{G}(i,j,\kappa,T)[\phi]$, and let $d_{i,j}$ denote the probability that $\mathcal{G}(i,j,\kappa,T)[\phi]$ culminates in a draw, so that $d_{i,j} = 1-\ell_{i,j}-w_{i,j}$. Although it is difficult to express $w_{i,j}$ and $\ell_{i,j}$ explicitly as functions of $i$, $j$, $\kappa$, $\chi$, $p_{-1}$, $p_{0}$ and $p_{1}$, it is possible to express them, loosely speaking, as specific fixed points of suitably defined functions. This is precisely what has been accomplished in Theorem~\ref{thm:1} below.

Given a matrix $X$, we let $X_{\alpha,\beta}$ indicate the entry in its $\alpha$-th row and $\beta$-th column, and by $X^{T}$ we mean the transpose of $X$. Let $L$ and $W$ be two $(\kappa-1)\times(\kappa-1)$ square matrices with $L_{i,j}=\ell_{i,j}$ and $W_{i,j}=w_{i,j}$, for each $i, j \in \{1,2,\ldots,\kappa-1\}$. Let $P$ be a $(\kappa-1) \times (\kappa-1)$ matrix with
\[
 P_{i,j} = 
  \begin{cases} 
   p_{-1} & \text{if } j=i-1, \\
   p_{0} & \text{if } j=i,\\
   p_{1} & \text{if } j=i+1,\\
   0 & \text{otherwise,}
  \end{cases}
\]
for all $i,j \in \{1,2,\ldots,\kappa-1\}$. In particular, $P_{1,1}=p_{0}$, $P_{1,2}=p_{1}$ and $P_{1,j}=0$ for all $j \in \{3,\ldots,\kappa-1\}$, whereas $P_{\kappa-1,\kappa-2}=p_{-1}$, $P_{\kappa-1,\kappa-1}=p_{0}$ and $P_{\kappa-1,j}=0$ for all $j \in \{1,\ldots,\kappa-3\}$. Let $\mathbf{e}_{1}$ denote the $(\kappa-1) \times 1$ column vector in which the element in the $1$st row equals $1$ and every other row equals $0$, let $\mathbf{1}$ denote the $(\kappa-1) \times 1$ column vector in which each entry equals $1$, and let $J = \mathbf{1}\mathbf{1}^{T}$ denote the $(\kappa-1)\times(\kappa-1)$ matrix in which every entry equals $1$. Given any two $m \times n$ matrices $A$ and $B$, where $m, n \in \mathbb{N}$, we say that $A \preceq B$ if $A_{i,j} \leqslant B_{i,j}$ for all $i \in \{1,2,\ldots,m\}$ and all $j \in \{1,2,\ldots,n\}$. Finally, let $f$ be the function, defined on the space $\mathcal{S}$ of all $(\kappa-1)\times(\kappa-1)$ matrices whose entries all belong to the set $[0,1]$, such that for any $A \in \mathcal{S}$,
\begin{equation}
f(A) = B \text{ where } B_{i,j} = G(A_{i,j}), \text{ for all } i,j \in \{1,2,\ldots,\kappa-1\}\nonumber
\end{equation}
with $G$ as defined in \eqref{pgf}. Armed with the above notations and definitions, we are now ready to state Theorem~\ref{thm:1}:
\begin{theorem}\label{thm:1}
Each of the matrices $L$ and $(J - W)$ is a fixed point of the matrix equation $Y = h(X)$, where the function $h$, defined on $\mathcal{S}$, is given by
\begin{equation}
h(X) = f\left[p_{-1}\mathbf{e}_{1}\mathbf{1}^{T} + P\left\{J - f\left(p_{-1}\mathbf{1}\mathbf{e}_{1}^{T}+(J-X)P^{T}\right)\right\}\right].\label{fixed_point_eq}
\end{equation}
Moreover, if $\hat{X}$ is any other fixed point, belonging to the set $\mathcal{S}$, of the function $h$, then
\begin{equation}
L \preceq \hat{X} \preceq J-W.\label{smallest_largest_fixed_point}
\end{equation}
Consequently, $d_{i,j}=0$ for each $i,j \in \{1,2,\ldots,\kappa-1\}$ if and only if the function $h$ has a unique fixed point in $\mathcal{S}$. If we define $g:\mathcal{S} \rightarrow \mathcal{S}$ as 
\begin{equation}
g(X) = f\left[p_{-1}\mathbf{e}_{1}\mathbf{1}^{T}+P\left(J-X^{T}\right)\right],\label{g_composition}
\end{equation}
then $g$ has a unique fixed point in $\mathcal{S}$ whenever $f$ has a unique fixed point in $\mathcal{S}$, and this fixed point of $g$ equals each of $L$ and $J-W$.
\end{theorem}

While Theorem~\ref{thm:1} concerns itself with the probabilities of \emph{all} three possible outcomes, Theorem~\ref{thm:2} is all about the probabilities of draw, $d_{i,j}$, where $i, j \in \{1,2,\ldots,\kappa-1\}$:
\begin{theorem}\label{thm:2}
As long as the pgf $G$ (defined in \eqref{pgf}) corresponding to the offspring distribution $\chi$ is strictly increasing, the following are true: 
\begin{enumerate}
\item \label{draw_1} When each of $p_{1}$, $p_{0}$ and $p_{-1}$ is strictly positive, either $d_{i,j}=0$ for \emph{all} $i,j \in \{1,2,\ldots,\kappa-1\}$, or else $d_{i,j} > 0$ for \emph{all} $i,j \in \{1,2,\ldots,\kappa-1\}$. 
\item \label{draw_2} When $p_{-1}=0$ while each of $p_{1}$ and $p_{0}=(1-p_{1})$ is each strictly positive, we have $d_{i,j}=0$ if and only if $d_{j,i}=0$, for each $i, j \in \{1,2,\ldots,\kappa-1\}$. Whenever each of $i$, $j$, $i+1$ and $j+1$ belongs to $\{1,2,\ldots,\kappa-1\}$, we have $d_{i,j}=0$ only if $d_{i+1,j}=d_{i,j+1}=0$. When $i=\kappa-1$ while $j \in \{1,2,\ldots,\kappa-2\}$, we have $d_{\kappa-1,j}=0$ only if $d_{\kappa-1,j+1}=0$, and when $j=\kappa-1$ while $i \in \{1,2,\ldots,\kappa-2\}$, we have $d_{i,\kappa-1}=0$ only if $d_{i+1,\kappa-1}=0$. 
\item \label{draw_3} Likewise, when $p_{1}=0$ while each of $p_{-1}$ and $p_{0}=(1-p_{-1})$ is strictly positive, we have $d_{i,j}=0$ if and only if $d_{j,i}=0$, for each $i, j \in \{1,2,\ldots,\kappa-1\}$. Whenever each of $i$, $j$, $i-1$ and $j-1$ belongs to $\{1,2,\ldots,\kappa-1\}$, we have $d_{i,j}=0$ only if $d_{i-1,j}=d_{i,j-1}=0$. When $i=1$ while $j \in \{2,\ldots,\kappa-1\}$, we have $d_{1,j}=0$ only if $d_{1,j-1}=0$, and when $j=1$ while $i \in \{2,\ldots,\kappa-1\}$, we have $d_{i,1}=0$ only if $d_{i-1,1}=0$.
\item \label{draw_4} When $p_{0}=0$ while each of $p_{1}$ and $p_{-1}=(1-p_{1})$ is strictly positive, we have $d_{i_{1},j_{1}}=0$ if and only if $d_{i_{2},j_{2}}=0$ whenever each of $i_{1}$, $i_{2}$, $j_{1}$, $j_{2}$ belongs to $\{1,2,\ldots,\kappa-1\}$ and $(i_{1}-i_{2})+(j_{1}-j_{2})$ is even.
\end{enumerate}
\end{theorem}
An interesting conclusion that can be drawn from \eqref{draw_1} of Theorem~\ref{thm:2}, for fixed values of $p_{-1}$, $p_{0}$ and $p_{1}$ in $(0,1)$, and offspring distribution $\chi$ of $T_{\chi}$ that has its underlying parameter (or parameter-tuple) $\theta$ coming from the parameter space $\Theta$, is as follows: if for some $i, j \in \{1,2,\ldots,\kappa-1\}$, there exists a subset $S \subset \Theta$ such that $d_{i,j}=0$ if and only if $\theta \in S$, then for \emph{all} $i', j' \in \{1,2,\ldots,\kappa-1\}$, we have $d_{i',j'}=0$ if and only if $\theta \in S$. In other words, the same \emph{phase transition phenomenon} is exhibited by $d_{i,j}$ for all $i,j \in \{1,2,\ldots,\kappa-1\}$ as $\theta$ is allowed to vary over all of $\Theta$. 

Our next main result concerns itself with the average duration of the game $\mathcal{G}(i,j,\kappa,\phi)[T]$ for each $(i,j) \in \{1,2,\ldots,\kappa-1\}^{2}$. An immediate necessary condition for the expected duration of $\mathcal{G}(i,j,\kappa,\phi)[T]$ to be finite is, of course, that $d_{i,j}=0$ (since $d_{i,j} > 0$ implies that the duration of $\mathcal{G}(i,j,\kappa,\phi)[T]$ is $\infty$ with a strictly positive probability).
\begin{theorem}\label{thm:3}
Assume that each of $p_{-1}$, $p_{0}$ and $p_{1}$ is strictly positive. Let $\mathcal{T}_{i,j}$ denote the (random) duration of the game $\mathcal{G}(i,j,\kappa,T)[\phi]$, where we recall that $T$ is the rooted, edge-weight GW tree with root $\phi$. For each $i,j \in \{1,2,\ldots,\kappa-1\}$, we set
\begin{align}
{}&\alpha_{i,j} = p_{-1}w_{j,i-1}+p_{0}w_{j,i}+p_{1}w_{j,i+1},\label{alpha}\\
{}&\beta_{i,j} = p_{-1}\left(1-\ell_{j,i-1}\right)+p_{0}\left(1-\ell_{j,i}\right)+p_{1}\left(1-\ell_{j,i+1}\right).\label{beta}
\end{align}
For each $s,t \in \{1,2,\ldots,\kappa-1\}$, we define 
\begin{equation}\label{intermediate_11}
 C_{(s,t)}^{(i,j)} = 
  \begin{cases} 
   G'(\beta_{s,t})G'(\alpha_{t,i})p_{i-s}p_{j-t} & \text{if } s \in \{i-1,i,i+1\} \text{ and } t \in \{j-1,j,j+1\}, \\
   0 & \text{otherwise. }
  \end{cases}
\end{equation}
If $d_{i',j'}=0$ for \emph{all} $i',j' \in \{1,2,\ldots,\kappa-1\}$ (which yields $\alpha_{i,j}=\beta_{i,j}$ for all $i, j \in \{1,2,\ldots,\kappa-1\}$), and
\begin{equation}
\sum_{s,t \in \{1,2,\ldots,\kappa-1\}}C_{(s,t)}^{(i',j')} < 1 \text{ for each } i',j' \in \{1,2,\ldots,\kappa-1\},\label{thm:3:eq}
\end{equation}
then the expected duration of the game, $\E[\mathcal{T}_{i,j}]$, is finite for every $i,j \in \{1,2,\ldots,\kappa-1\}$.
\end{theorem}
\begin{remark}\label{rem:thm_3}
Note, from Theorem~\ref{thm:2}, that when $p_{-1}$, $p_{0}$ and $p_{1}$ are all strictly positive, it suffices to consider $d_{i,j}=0$ for \emph{any one} pair $(i,j) \in \{1,2,\ldots,\kappa-1\}^{2}$ in the statement of Theorem~\ref{thm:3}.

It is crucial to note here that, even when our goal is to investigate the finiteness (or lack thereof) of the expected duration of $\mathcal{G}(i,j,\kappa,T)[\phi]$ for a \emph{specific} pair $(i,j)$, with $i,j \in \{1,2,\ldots,\kappa-1\}$ (defined in \eqref{i,j_domain}), we need the inequality in \eqref{thm:3:eq} to hold for \emph{all} pairs $(i',j')$ with $i',j' \in \{1,2,\ldots,\kappa-1\}$. 
\end{remark}

As evident from the statement of Theorem~\ref{thm:1}, for values of $\kappa$ higher than $2$, we are faced with the task of finding (or at the very least, drawing conclusions about the uniqueness, or lack thereof, of) all fixed points, belonging to $\mathcal{S}$, of the multivariable function $h$ defined in \eqref{fixed_point_eq}. This is usually not tractable using rigorous analytical methods, and one needs to resort to numerical methods for which seeking the aid of a computer programme seems the wisest course of action. However, for $\kappa=2$, for specific offspring distributions (that are fairly commonly considered for GW branching processes), we are able to obtain precise phase transition results (see Theorem~\ref{thm:kappa=2}), whereas for $\kappa=3$, for specific regimes of values of $(p_{-1},p_{0},p_{1})$ and for specific underlying offspring distributions, we are either able to provide sufficient conditions for $d_{i,j}$ to be $0$ for every $i,j \in \{1,2\}$, or sufficient conditions for $d_{i,j}$ to be strictly positive for each $i,j \in \{1,2\}$ (see Theorem~\ref{thm:kappa=3}).
\begin{theorem}\label{thm:kappa=2}
\sloppy Let $\kappa=2$, and recall that $\chi$ is the offspring distribution of the underlying rooted GW tree.
\begin{enumerate}
\item when $\chi$ is Binomial$(d,\pi)$ distribution, for any $0 < \pi \leqslant 1$ and any $d \in \mathbb{N}$ with $d \geqslant 2$, the probability of draw, $d_{1,1}$, is $0$ if and only if $\left(1-p_{-1}-p_{1}\right) \pi (1-\pi p_{1})^{d-1} \leqslant (d+1)^{d-1}d^{-d}$;
\item when $\chi$ is Poisson$(\lambda)$, for any $\lambda > 0$, the probability of draw, $d_{1,1}$, is $0$ if and only if $\left(1-p_{-1}-p_{1}\right) \lambda e^{-\lambda p_{1}} \leqslant e$;
\item when $\chi$ is Negative Binomial$(r,\pi)$, for $r \in \mathbb{N}$ and $0 < \pi \leqslant 1$, the probability of draw, $d_{1,1}$, is $0$ if and only if $(r-1)^{r+1}(1-\pi)\left(1-p_{-1}-p_{1}\right)\pi^{r} \leqslant (p_{1}+\pi-p_{1}\pi)^{r+1}r^{r}$;
\item when $\chi(0) = 1-\pi$ and $\chi(d) = \pi$ for some $0 < \pi < 1$ and some $d \in \mathbb{N}$ with $d \geqslant 2$, the probability of draw, $d_{1,1}$, is $0$ if and only if $\left(1-p_{-1}-p_{1}\right) \pi \left\{\pi(1-p_{1}) + p_{-1}(1-\pi)\right\}^{d-1} \leqslant \frac{(d+1)^{d-1}}{d^{d}}$.
\end{enumerate}
\end{theorem}
Suppose we consider the special case of $p_{-1}=0$ when $\kappa=3$. From Theorem~\ref{thm:2}, we know that $d_{1,1}=0$ implies that $d_{1,2}=d_{2,1}=0$, and each of these, in turn, implies that $d_{2,2}=0$. This, in turn, tells us that when $d_{2,2}>0$, we have each of $d_{1,1}$, $d_{1,2}$ and $d_{2,1}$ strictly positive. Since $p_{-1}=0$, the only way for a player to win the game $\mathcal{G}(2,2,3,T)[\phi]$ (i.e.\ with $i=2$, $j=2$ and $\kappa=3$) is to move the token along an edge with edge-weight $+1$, which means that this is equivalent to the game $\mathcal{G}(1,1,2,T)[\phi]$ (i.e.\ $i=j=1$ and $\kappa=2$) when $p_{-1}=0$. From Theorem~\ref{thm:kappa=2}, we then conclude that:
\begin{enumerate}
\item when $\chi$ is Binomial$(d,\pi)$ distribution, for any $0 < \pi \leqslant 1$ and any $d \in \mathbb{N}$ with $d \geqslant 2$, and $p_{-1}=0$, each $d_{i,j}$, for $i, j \in \{1,2\}$, is strictly positive whenever $\left(1-p_{1}\right) \pi (1-\pi p_{1})^{d-1} > (d+1)^{d-1}d^{-d}$;
\item when $\chi$ is Poisson$(\lambda)$, for any $\lambda > 0$, and $p_{-1}=0$, each $d_{i,j}$, for $i,j \in \{1,2\}$, is strictly positive whenever $\left(1-p_{1}\right) \lambda e^{-\lambda p_{1}} > e$;
\item when $\chi$ is Negative Binomial$(r,\pi)$, for $r \in \mathbb{N}$ and $0 < \pi \leqslant 1$, each $d_{i,j}$, for $i,j \in \{1,2\}$, is strictly positive whenever $(r-1)^{r+1}(1-\pi)\left(1-p_{1}\right)\pi^{r} > (p_{1}+\pi-p_{1}\pi)^{r+1}r^{r}$;
\item when $\chi(0) = 1-\pi$ and $\chi(d) = \pi$ for some $0 < \pi < 1$ and some $d \in \mathbb{N}$ with $d \geqslant 2$, each $d_{i,j}$, for $i,j \in \{1,2\}$, is strictly positive whenever $\left(1-p_{1}\right) \pi^{d}(1-p_{1})^{d-1} > \frac{(d+1)^{d-1}}{d^{d}}$.
\end{enumerate}
Similar conclusions can be drawn when we consider $p_{1}=0$ and $\kappa=3$.

In order to state Theorem~\ref{thm:kappa=3}, we introduce some notations: for $r \in \mathbb{N}$ with $r \geqslant 2$ and for any function $f: [0,1]^{r} \rightarrow [0,1]$ whose partial derivatives exist, we set $\partial_{i}f(\mathbf{x})=\frac{\partial}{\partial x_{i}}f(\mathbf{x})$ for each $i \in \{1,2,\ldots,r\}$, where $\mathbf{x}=(x_{1},x_{2},\ldots, x_{r})$.  
\begin{theorem}\label{thm:kappa=3}
Let $\kappa=3$. Recall, from \eqref{i,j_domain}, that we are now concerned with $d_{i,j}$ for $i, j \in \{1,2\}$. Let $A_{1,1}=A_{1,2}=G\left(p_{-1}\right)$, let $B_{2,1}=B_{2,2}=G\left(1-p_{1}\right)$, and let 
\begin{align}
A_{2,1}={}&G\left[p_{-1}\left\{1-G\left[1-\left(1-p_{-1}\right)G\left(p_{-1}\right)\right]\right\}\right],\nonumber\\
A_{2,2}={}&G\left[p_{-1}\left\{1-G\left\{1-G\left(p_{-1}\right)\right\}\right\}\right]+G\left[\left(1-p_{1}\right)\left\{1-G\left(1-p_{1}\right)\right\}\right]\nonumber\\& \hspace{9.7 
mm}-G\left[p_{-1}\left[1-G\left(1-p_{1}\right)-G\left[1-G\left(p_{-1}\right)\right]+G\left[\left(1-p_{1}\right)\left\{1-G\left(p_{-1}\right)\right\}\right]\right]\right],\nonumber\\
B_{1,1}={}&G\left[1-\left(1-p_{-1}\right)G\left(p_{-1}\right)\right]-G\left[\left(1-p_{-1}\right)\left\{1-G\left(p_{-1}\right)\right\}+p_{-1}\right]+G[\left(1-p_{-1}\right)\left\{1-G\left(p_{-1}\right)\right\}\nonumber\\&+p_{-1}-p_{1}G\left\{1-G\left(1-p_{1}\right)\right\}+p_{1}G\left[p_{-1}\left\{1-G\left(1-p_{1}\right)\right\}\right]],\nonumber\\
B_{1,2}={}&G\left[1-p_{1}G\left[\left(1-p_{1}\right)\left\{1-G\left(1-p_{1}\right)\right\}\right]\right].\nonumber
\end{align}
Define the functions $f_{1}: [0,1]^{2} \rightarrow [0,1]$ and $f_{2}: [0,1]^{2} \rightarrow [0,1]$ as $f_{1}(x_{1},x_{2})=G\left(1-p_{1}x_{2}-p_{0}x_{1}\right)$ and $f_{2}(x_{1},x_{2})=G\left(p_{0}+p_{-1}-p_{0}x_{2}-p_{-1}x_{1}\right)$. For each $(i,j) \in \{1,2\}^{2}$, we set
\begin{align}
E_{i,j}={}&\Big|p_{i-1}\partial_{j}f_{1}(A_{i,1},A_{i,2})G'\left\{1-p_{1}f_{1}\left(B_{2,1},B_{2,2}\right)-p_{0}f_{1}\left(B_{1,1},B_{1,2}\right)\right\}\nonumber\\&+p_{i-1}\partial_{j}f_{2}(A_{i,1},A_{i,2})G'\left\{1-p_{1}f_{2}\left(B_{2,1},B_{2,2}\right)-p_{0}f_{2}\left(B_{1,1},B_{1,2}\right)\right\}\nonumber\\&+p_{i-2}\partial_{j}f_{1}(A_{i,1},A_{i,2})G'\left\{p_{0}+p_{-1}-p_{0}f_{1}\left(B_{2,1},B_{2,2}\right)-p_{-1}f_{1}\left(B_{1,1},B_{1,2}\right)\right\}\nonumber\\&+p_{i-2}\partial_{j}f_{2}(A_{i,1},A_{i,2})G'\left\{p_{0}+p_{-1}-p_{0}f_{2}\left(B_{2,1},B_{2,2}\right)-p_{-1}f_{2}\left(B_{1,1},B_{1,2}\right)\right\}\Big|.\label{E_{i,j}_defn}
\end{align}
As long as $E_{i,j} < 1$ for each $(i,j) \in \{1,2\}^{2}$, the function $h$, defined in \eqref{fixed_point_eq}, has a unique fixed point in $\mathcal{S}$, implying that $d_{i,j}=0$ for each $(i,j) \in \{1,2\}^{2}$. 
\end{theorem}

Improved results (in the sense that the sufficient conditions obtained to guarantee that the probability of draw is $0$ are easier and more concise than what we achieve in Theorem~\ref{thm:kappa=3}) can be deduced when we consider certain special cases under $\kappa=3$, as stated in Theorem~\ref{thm:kappa=3_special}. Here, for any function $a: [0,1] \rightarrow [0,1]$, we let $a^{(n)}$ denote, for any $n \in \mathbb{N}$, the $n$-fold composition of $a$ with itself, and for any two functions $a, b :[0,1] \rightarrow [0,1]$, we let $a \circ b$ denote the composition of these two functions, i.e.\ $a \circ b(x)=a(b(x))$ for each $x \in [0,1]$.

\begin{theorem}\label{thm:kappa=3_special}
As mentioned, we consider $\kappa=3$ once again.
\begin{enumerate}
\item \label{kappa=3_part_1} When $\chi$ is the Dirac measure at $2$, i.e.\ $T_{\chi}$ is the rooted $2$-regular tree (in which each vertex has precisely $2$ children), and the ratio of $p_{-1}$ to $p_{0}$ to $p_{1}$ is of the form $p_{-1} : p_{0} : p_{1} = 1 : \alpha : \alpha^{2}$, i.e.\ we have
\begin{equation}
p_{1}=\frac{\alpha^{2}}{\alpha^{2}+\alpha+1},\quad p_{0}=\frac{\alpha}{\alpha^{2}+\alpha+1} \quad \text{and} \quad p_{-1}=\frac{1}{\alpha^{2}+\alpha+1},\label{special_form}
\end{equation}
the probability $d_{i,j}$ of draw is $0$ for each $i, j \in \{1,2\}$ whenever $\alpha \in [0,0.242915)\cup(2.57162,\infty)$. 
\item \label{kappa=3_part_2} When $p_{0}=0$, so that $p_{1}=(1-p_{-1})$, we have $d_{1,1}=0$, which is equivalent to $d_{2,2}=0$, if and only if the function $b^{(2)} \circ a^{(2)}$ has a unique fixed point in $[0,1]$, where $a, b: [0,1] \rightarrow [0,1]$ are defined as
\begin{equation}
a(x)=G\left(p_{-1}-p_{-1} x\right) \quad \text{and} \quad b(x)=G\left(1-p_{1}x\right) \quad \text{for each } x \in [0,1].\nonumber
\end{equation}
Likewise, $d_{1,2}=0$, which is equivalent to $d_{2,1}=0$, if and only if $b \circ a^{(2)} \circ b$ has a unique fixed point in $[0,1]$. In particular, whenever we have 
\begin{equation}
p_{-1}\left(1-p_{-1}\right) G'\left[1-\left(1-p_{-1}\right)G\left(p_{-1}\right)\right]G'\left[p_{-1}\left\{1-G\left(p_{-1}\right)\right\}\right] < 1,\nonumber
\end{equation}
we conclude that $d_{1,1}$ (equivalently, $d_{2,2}$) as well as $d_{1,2}$ (equivalently, $d_{2,1}$) equals $0$.
\end{enumerate}
\end{theorem}

{The following Corollary is obtained from part (\ref{kappa=3_part_2}) of Theorem \ref{thm:kappa=3_special}.
\begin{corollary}\label{cor_1}
   When $\kappa=3$, $\chi$ is either the Dirac measure at 2 or Poisson(2), and $p_0=0$, we have  $d_{i,j} = 0$ for each $(i,j)\in \{1,2\}^2$.
\end{corollary}}

When it comes to answering whether the probability of a weighted-tree-percolation game resulting in a draw equals $0$ or not, we have already explained, in the paragraph preceding Theorem~\ref{thm:kappa=2}, why the conclusions drawn in the main results (in particular, Theorem~\ref{thm:1}) of this paper are all in terms of the uniqueness, or the lack thereof, of the fixed points of certain functions in specific domains. It would be desirable to come up with more easily verifiable criteria (such as those stated in Theorem~\ref{thm:kappa=2} for $\kappa=2$) -- perhaps inequalities in terms of the parameters $p_{-1}$, $p_{0}$, $p_{1}$, and the parameter(s) underlying the offspring distribution of the rooted GW tree on which the game is being played -- such that the probability of draw equals $0$ when these criteria hold (and it is strictly positive when these criteria fail). However, for values of $\kappa$ higher than $2$ and in the most general scenario (where no additional assumptions are made about $p_{-1}$, $p_{0}$, $p_{1}$), even simple and commonly studied choices of the offspring distribution (such as when we assume that the underlying tree is rooted regular, or that the offspring distribution is binomial or Poisson) lead to extremely complicated forms for the function $h$ defined in \eqref{fixed_point_eq}. Other than finding the domain of values of $p_{-1}$, $p_{0}$, $p_{1}$ and the parameter(s) underlying the chosen offspring distribution for which $h$ is guaranteed to be a contraction -- this, in turn, would imply, by the well-known Banach Fixed Point Theorem, that $h$ has a unique fixed point in $\mathcal{S}$, and therefore, $d_{i,j}=0$ for all $(i,j) \in \{1,2,\ldots,\kappa-1\}^{2}$ -- there are no results from the literature on multivariable calculus that we can use to come up with easily and manually verifiable necessary and sufficient condition(s) for $h$ to have a unique fixed point in $\mathcal{S}$. 

In the most general scenario, even the domain in which an application of the Banach Fixed Point Theorem is possible is rather limited by the very crude bounds on the gradient $\grad h_{i,j}(X)$, where $h_{i,j}$ is the $(i,j)$-th coordinate of the function $h$ defined in \eqref{fixed_point_eq}, for $X \in \mathcal{S}$. This situation is improved to some extent when we consider $\kappa=3$, since we are able to provide a lower bound on each of $\ell_{i,j}$ and $w_{i,j}$, for each $(i,j) \in \{1,2\}^{2}$ (these bounds are where the quantities $A_{i,j}$ and $B_{i,j}$, for $(i,j) \in \{1,2\}^{2}$, come into play). Finding such lower bounds becomes an unwieldy and complicated task when $\kappa$ is increased.

\section{Simulations, and subsequent discussions, pertaining to the results of \S\ref{subsec:main_results}}\label{sec:simulations}

As mentioned earlier, when it comes to deducing easily, quickly and manually verifiable conditions that guarantee that the probability of draw equals $0$, the main hurdle to surmount lies in the need to investigate whether the function $h$, defined in \eqref{fixed_point_eq}, has a unique fixed point in $\mathcal{S}$ or not. The existing theory for analyzing multivariable functions does not provide us with a general means for checking whether $h$ (or, for that matter, the much simpler function $g$ defined in \eqref{g_composition}) has a unique fixed point in the set $\mathcal{S}$ or not. Even in highly simplified set-ups, such as when we replace $T_{\chi}$ by $T_{m}$, the rooted $m$-regular tree in which each vertex, including the root, has precisely $m$ children, for $m \geqslant 2$, concluding whether $h$ (or $g$) has a unique fixed point in $\mathcal{S}$ or not, depending on the values of the parameters $p_{-1}$, $p_{0}$ and $p_{1}$, becomes extremely difficult without the aid of a computer. This is the reason why we need to turn to simulations.   

Throughout the rest of \S\ref{sec:simulations}, we consider $\kappa=3$, so that $\mathcal{S}$ here comprises all $2 \times 2$ matrices, each of whose coordinates belongs to the interval $[0,1]$. Using the {\tt FixedPoint} package of R, the fixed points of the function $h$, belonging to the set $\mathcal{S}$, have been estimated numerically 
\begin{enumerate}
\item for various values of the parameters $p_{-1}$, $p_{0}$ and $p_{1}$ (notice that $p_{-1}=1-p_{0}-p_{1}$, so that $p_{-1}$ becomes fixed automatically once we fix the values of $p_{0}$ and $p_{1}$),
\item and for various offspring distributions of the underlying rooted GW tree (including some special cases where the offspring distribution is a Dirac measure, yielding a rooted regular tree in which every vertex has the same number of children).
\end{enumerate}
In particular, this exercise reveals whether the function $h$ has a unique fixed point in $\mathcal{S}$ or not. This, in turn, allows us, via Theorem~\ref{thm:1} (as well as Theorem~\ref{thm:2}), to conclude whether the probabilities of draw, namely $d_{1,1}$, $d_{1,2}$, $d_{2,1}$ and $d_{2,2}$, are all $0$ or all positive, in each of the cases we have considered. In what follows, 
\begin{enumerate}
\item the notation $\delta_{m}$, for any $m \in \mathbb{N}$, indicates that the offspring distribution is a Dirac measure at $m$, i.e.\ the rooted tree under consideration is $T_{m}$, in which each vertex has precisely $m$ children;
\item the notation Unif$[m]$, for any $m \in \mathbb{N}$, indicates the discrete uniform distribution on the set $\{1,\ldots,m\}$, i.e.\ where a vertex is equally likely to have $i$ children, for all $i \in \{1,\ldots,m\}$;
\item Binomial$(n,p)$, for $n \in \mathbb{N}$ and $p \in [0,1]$, indicates that a vertex of our rooted GW tree has probability ${n \choose i}p^{i}(1-p)^{n-i}$ of having $i$ children, for all $i \in \{0,1,\ldots,n\}$;
\item Poisson$(\lambda)$, for $\lambda > 0$, indicates that a vertex of our rooted GW tree has probability $e^{-\lambda}\lambda^{i}/i!$ of having $i$ children, for all $i \in \mathbb{N}_{0}$.
\end{enumerate}
In each of Tables~\ref{table:Dirac}, \ref{table:unif}, \ref{table:binomial} and \ref{table:Poisson}, the first column tells us which offspring distribution is being studied, while the second and third columns state, respectively, the values of $p_{0}$ and $p_{1}$ being considered. The fourth, fifth, sixth and seventh columns tabulate, respectively, the draw probabilities $d_{1,1}$, $d_{1,2}$, $d_{2,1}$ and $d_{2,2}$, corresponding to, respectively, the games $\mathcal{G}(1,1,3,T)[\phi]$, $\mathcal{G}(1,2,3,T)[\phi]$, $\mathcal{G}(2,1,3,T)[\phi]$ and $\mathcal{G}(2,2,3,T)[\phi]$, where $\phi$ is the root of $T$, the offspring distribution for $T$ is stated in the corresponding row, and $p_{0}$ and $p_{1}$ equal the values stated in the corresponding row. 

\begin{table}[!ht]
    \centering
    \begin{tabular}{|l|l|l|l|l|l|l|}
    \hline
        Distribution & $p_{0}$ & $p_{1}$ & $d_{1,1}$ & $d_{1,2}$ & $d_{2,1}$ & $d_{2,2}$ \\ \hline
        $\delta_{2}$ & $0.8$ & $0.1$ & $3.025 \cdot 10^{-4}$ & $4.54 \cdot 10^{-5}$ & $1.95 \cdot 10^{-5}$ & $8.28 \cdot 10^{-5}$ \\ \hline
        ~ & $0.8$ & $0.15$ & $1.97 \cdot 10^{-9}$ & $1.18 \cdot 10^{-10}$ & $5.36 \cdot 10^{-11}$ & $1.16 \cdot 10^{-10}$ \\ \hline
        ~ & $0.9$ & $0.05$ & $0.985521946$ & $0.953909988$ & $0.827344951$ & $0.865247284$ \\ \hline
        ~ & $0.9$ & $0.075$ & $0.982516861$ & $0.906997229$ & $0.716012057$ & $0.760068679$ \\ \hline
        ~ & $0.95$ & $0.025$ & $0.997569724$ & $0.99239292$ & $0.936814203$ & $0.942461793$ \\ \hline
        ~ & $0.95$ & $0.03$ & $0.998158685$ & $0.991029043$ & $0.924885509$ & $0.928890449$ \\ \hline
        ~ & $0.99$ & $0.001$ & $0.999749449$ & $0.999873892$ & $0.997750077$ & $0.997989003$ \\ \hline
        ~ & $0.99$ & $0.005$ & $0.999921563$ & $0.999762178$ & $0.989558289$ & $0.989720379$ \\ \hline
        $\delta_{5}$ & $0.8$ & $0.1$ & $0.644971286$ & $0.232637754$ & $0.002239967$ & $0.503766966$ \\ \hline
        ~ & $0.8$ & $0.15$ & $0.373571553$ & $0.055377224$ & $0.000135404$ & $0.081651632$ \\ \hline
        ~ & $0.9$ & $0.05$ & $0.999641612$ & $0.998453766$ & $0.768565913$ & $0.771550871$ \\ \hline
        ~ & $0.9$ & $0.075$ & $0.997810684$ & $0.993132385$ & $0.658155123$ & $0.665497296$ \\ \hline
        ~ & $0.95$ & $0.025$ & $0.999992937$ & $0.999941319$ & $0.880855111$ & $0.880984424$ \\ \hline
        ~ & $0.95$ & $0.03$ & $0.99998481$ & $0.999891554$ & $0.858308778$ & $0.858480712$ \\ \hline
        ~ & $0.99$ & $0.001$ & $1$ & $0.999999999$ & $0.995009987$ & $0.99500999$ \\ \hline
        ~ & $0.99$ & $0.005$ & $0.999999999$ & $0.999999977$ & $0.975248642$ & $0.9752487$ \\ \hline
        $\delta_{10}$ & $0.8$ & $0.1$ & $0.457819736$ & $0.252766466$ & $9.20 \cdot 10^{-6}$ & $0.291509514$ \\ \hline
        ~ & $0.9$ & $0.05$ & $0.999885808$ & $0.999738867$ & $0.59744268$ & $0.598300183$ \\ \hline
        ~ & $0.9$ & $0.075$ & $0.997437323$ & $0.995761624$ & $0.44354551$ & $0.452649668$ \\ \hline
        ~ & $0.95$ & $0.025$ & $0.99999982$ & $0.999999219$ & $0.776324172$ & $0.776327537$ \\ \hline
        ~ & $0.95$ & $0.03$ & $0.999999192$ & $0.999996963$ & $0.737404681$ & $0.737414678$ \\ \hline
        ~ & $0.99$ & $0.001$ & $1$ & $1$ & $0.99004488$ & $0.99004488$ \\ \hline
        ~ & $0.99$ & $0.005$ & $1$ & $1$ & $0.95111013$ & $0.95111013$ \\ \hline
    \end{tabular}
\caption{Probabilities of draw for $\mathcal{G}(i,j,\kappa,T)[\phi]$, where $\kappa=3$, $(i,j) \in \{1,2\}^{2}$, $T$ has offspring distribution $\delta_{m}$ for $m \in \{2,5,10\}$, and various values of $p_{0}$ and $p_{1}$ are considered.}
\label{table:Dirac}
\end{table}
Let us note here that in Table~\ref{table:Dirac}, some of the cells in the rows corresponding to $p_{0}=0.99$ equal $1$ (for instance, when we consider the offspring distribution to be $\delta_{10}$, $p_{0}=0.99$ and $p_{1} \in \{0.001,0.005\}$, both $d_{1,1}$ and $d_{1,2}$ equal $1$). We surmise that these values aren't \emph{truly} equal to $1$, but rather, so close to $1$ as to be deemed essentially equal to $1$ by the computer. This, we believe, is a consequence of choosing $p_{0}=0.99$, which ensures that nearly every edge of $T$ (which, when endowed with the offspring distribution $\delta_{10}$, is a rooted $10$-regular tree, with each vertex having precisely $10$ children) bears the edge-weight $0$, thus forcing draw to become an overwhelmingly likely outcome for each of the games $\mathcal{G}(1,1,3,T)[\phi]$ and $\mathcal{G}(1,2,3,T)[\phi]$. On the other hand, when the offspring distribution is $\delta_{10}$, $p_{0}=0.99$ and $p_{1} \in \{0.001,0.005\}$, that $d_{2,1}$ is not \emph{that} close to $1$ so as to be treated as essentially equal to $1$ can perhaps be attributed to the fact that $d_{2,1}$ is the probability of draw for the game $\mathcal{G}(2,1,3,T)[\phi]$. We intuit that, even when $p_{0}$ is as large a value as $0.99$, $w_{2,1}$, the probability of $P_{1}$ winning $\mathcal{G}(2,1,3,T)[\phi]$, is significant enough possibly because $P_{1}$, in addition to enjoying the advantage of playing the first round of the game, begins with a larger initial capital than $P_{2}$. Note, further, that even though $\ell_{1,2} \geqslant \ell_{2,1}$, we suspect that neither of them plays a significant role here (i.e.\ when $p_{0}=0.99$). This is because $\ell_{1,2}$ is the probability of $P_{1}$ losing the game $\mathcal{G}(1,2,3,T)[\phi]$, and although she begins with a lower amount of initial capital than her opponent, the fact that she plays the first round of the game gives her an edge over $P_{2}$ and makes her not so easy to defeat (in other words, $\ell_{1,2}$ is likely negligible, which also makes $\ell_{2,1}$ negligible).

A more complicated, though intuitive, explanation can perhaps be proposed, as follows, for the reason why $d_{1,1}$ is a lot closer to $1$ than $d_{2,2}$ when the offspring distribution is $\delta_{10}$, $p_{0}=0.99$ and $p_{1} \in \{0.001,0.005\}$. Recall that $d_{1,1}$ is the probability of draw for the game $\mathcal{G}(1,1,3,T)[\phi]$ in which each of $P_{1}$ and $P_{2}$ begins with an initial capital worth $1$ unit. Given a vertex $v$ of $T$, with children $v_{1}, v_{2}, \ldots, v_{m}$, the probability of the event that each of the edges $(v,v_{1}), \ldots, (v,v_{m})$ is assigned the weight $-1$ is very low, since $p_{0}=0.99$. Consequently, the probability that the game ever unfolds in a manner such that the token reaches such a vertex $v$ is extremely low as well. This tells us that $P_{1}$ winning $\mathcal{G}(1,1,3,T)[\phi]$ by forcing $P_{2}$ to squander her capital (which only happens if $P_{2}$ is \emph{forced} to traverse an edge with edge-weight $-1$) is a highly unlikely event. The only other way for $P_{1}$ to win $\mathcal{G}(1,1,3,T)[\phi]$ is by amassing a total capital worth $3$ units (since $\kappa=3$). This, of course, is a more difficult goal to accomplish when $P_{1}$ begins with an initial capital of $1$ unit than when she begins with an initial capital of $2$ units. This brings us to $d_{2,2}$, which is the probability of draw for the game $\mathcal{G}(2,2,3,T)[\phi]$. The preceding argument, we hope, provides an intuition as to why the probability of $P_{1}$ winning $\mathcal{G}(2,2,3,T)[\phi]$ by achieving the target total capital worth $3$ units ought to be higher than $P_{1}$ winning $\mathcal{G}(1,1,3,T)[\phi]$ by fulfilling the same criterion. We surmise that this, among other reasons, perhaps serves to justify why $d_{1,1}$ is so close to $1$ as to be considered indistinguishable from $1$ by the computer, while $d_{2,2}$ is not.  
\begin{table}[!ht]
    \centering
    \begin{tabular}{|l|l|l|l|l|l|l|}
    \hline
        Distribution & $p_{0}$ & $p_{1}$ & $d_{1,1}$ & $d_{1,2}$ & $d_{2,1}$ & $d_{2,2}$ \\ \hline
        Unif$[2]$ & $0.8$ & $0.1$ & $1.07 \cdot 10^{-9}$ & $7.3 \cdot 10^{-10}$ & $5.5 \cdot 10^{-10}$ & $6.48 \cdot 10^{-10}$ \\ \hline
        ~ & $0.8$ & $0.15$ & $7.05 \cdot 10^{-10}$ & $2.64 \cdot 10^{-10}$ & $1.97 \cdot 10^{-10}$ & $1.31 \cdot 10^{-10}$ \\ \hline
        ~ & $0.9$ & $0.05$ & $4.4 \cdot 10^{-9}$ & $1.79 \cdot 10^{-9}$ & $1.27 \cdot 10^{-9}$ & $2.4 \cdot 10^{-9}$ \\ \hline
        ~ & $0.9$ & $0.075$ & $2.51 \cdot 10^{-9}$ & $5.25 \cdot 10^{-10}$ & $3.75 \cdot 10^{-10}$ & $4.09 \cdot 10^{-10}$ \\ \hline
        ~ & $0.95$ & $0.025$ & $0.539212545$ & $0.137032613$ & $0.089145207$ & $0.466602762$ \\ \hline
        ~ & $0.95$ & $0.03$ & $0.517066998$ & $0.124672449$ & $0.079709669$ & $0.399299428$ \\ \hline
        ~ & $0.99$ & $0.001$ & $0.95210633$ & $0.97569076$ & $0.963046169$ & $0.989787602$ \\ \hline
        ~ & $0.99$ & $0.005$ & $0.973650568$ & $0.97100973$ & $0.950543282$ & $0.953273448$ \\ \hline
        Unif$[5]$ & $0.8$ & $0.1$ & $1.43 \cdot 10^{-9}$ & $4.53 \cdot 10^{-10}$ & $1.89 \cdot 10^{-10}$ & $2.68 \cdot 10^{-10}$ \\ \hline
        ~ & $0.8$ & $0.15$ & $5.86 \cdot 10^{-10}$ & $9.68 \cdot 10^{-11}$ & $4.05 \cdot 10^{-11}$ & $2.57 \cdot 10^{-11}$ \\ \hline
        ~ & $0.9$ & $0.05$ & $0.476469522$ & $0.134183225$ & $0.041925715$ & $0.281645405$ \\ \hline
        ~ & $0.9$ & $0.075$ & $2.33 \cdot 10^{-7}$ & $1.25 \cdot 10^{-8}$ & $4.91 \cdot 10^{-9}$ & $5.53 \cdot 10^{-9}$ \\ \hline
        ~ & $0.95$ & $0.025$ & $0.93710694$ & $0.917994318$ & $0.746793593$ & $0.771564456$ \\ \hline
        ~ & $0.95$ & $0.03$ & $0.940502744$ & $0.900282827$ & $0.70427371$ & $0.721945807$ \\ \hline
        ~ & $0.99$ & $0.001$ & $0.981846065$ & $0.993579999$ & $0.978510779$ & $0.990858348$ \\ \hline
        ~ & $0.99$ & $0.005$ & $0.989661771$ & $0.989082001$ & $0.954409496$ & $0.95501548$ \\ \hline
        Unif$[10]$ & $0.8$ & $0.1$ & $10^{-9}$ & $2.37 \cdot 10^{-10}$ & $6.06 \cdot 10^{-11}$ & $6.38 \cdot 10^{-11}$ \\ \hline
        ~ & $0.8$ & $0.15$ & $4.02 \cdot 10^{-10}$ & $5.68 \cdot 10^{-11}$ & $1.43 \cdot 10^{-11}$ & $5.49 \cdot 10^{-12}$ \\ \hline
        ~ & $0.9$ & $0.05$ & $0.400761157$ & $0.104575057$ & $0.017271264$ & $0.154663886$ \\ \hline
        ~ & $0.9$ & $0.075$ & $4.57 \cdot 10^{-9}$ & $1.84 \cdot 10^{-10}$ & $4.37 \cdot 10^{-11}$ & $2.73 \cdot 10^{-11}$ \\ \hline
        ~ & $0.95$ & $0.025$ & $0.948416337$ & $0.932479894$ & $0.655127956$ & $0.678855965$ \\ \hline
        ~ & $0.95$ & $0.03$ & $0.945658255$ & $0.910244623$ & $0.586954161$ & $0.612376352$ \\ \hline
        ~ & $0.99$ & $0.001$ & $0.986966425$ & $0.996543296$ & $0.976576626$ & $0.986571106$ \\ \hline
        ~ & $0.99$ & $0.005$ & $0.992371209$ & $0.991874236$ & $0.933023436$ & $0.933551923$ \\ \hline
 \end{tabular}
\caption{Probabilities of draw for $\mathcal{G}(i,j,\kappa,T)[\phi]$, where $\kappa=3$, $(i,j) \in \{1,2\}^{2}$, $T$ has offspring distribution Unif$[m]$ for $m \in \{2,5,10\}$, and various values of $p_{0}$ and $p_{1}$ are considered.}
\label{table:unif}
\end{table}
\begin{table}[!ht]
    \centering
    \begin{tabular}{|l|l|l|l|l|l|l|}
    \hline
        Distribution & $p_{0}$ & $p_{1}$ & $d_{1,1}$ & $d_{1,2}$ & $d_{2,1}$ & $d_{2,2}$ \\ \hline
        Binomial$(20,0.2)$ & $0.8$ & $0.1$ & $3.89 \cdot 10^{-9}$ & $4.04 \cdot 10^{-10}$ & $1.22 \cdot 10^{-10}$ & $2.39 \cdot 10^{-10}$ \\ \hline
        ~ & $0.8$ & $0.15$ & $7.26 \cdot 10^{-10}$ & $3.53 \cdot 10^{-11}$ & $1.11 \cdot 10^{-11}$ & $7.8 \cdot 10^{-12}$ \\ \hline
        ~ & $0.9$ & $0.05$ & $0.875627796$ & $0.824890093$ & $0.528638999$ & $0.616452939$ \\ \hline
        ~ & $0.9$ & $0.075$ & $0.541356592$ & $0.1406333$ & $0.018061161$ & $0.344140159$ \\ \hline
        ~ & $0.95$ & $0.025$ & $0.911054219$ & $0.904927347$ & $0.772964943$ & $0.780557631$ \\ \hline
        ~ & $0.95$ & $0.03$ & $0.912013629$ & $0.899774767$ & $0.745609615$ & $0.749889441$ \\ \hline
        ~ & $0.99$ & $0.001$ & $0.920357757$ & $0.923372715$ & $0.91523159$ & $0.918408287$ \\ \hline
        ~ & $0.99$ & $0.005$ & $0.921939033$ & $0.921746689$ & $0.895744495$ & $0.895944552$ \\ \hline
        Binomial$(20,0.5)$ & $0.8$ & $0.1$ & $0.434900217$ & $0.217984923$ & $0.000168193$ & $0.250856166$ \\ \hline
        ~ & $0.8$ & $0.15$ & $1.24 \cdot 10^{-9}$ & $1.48 \cdot 10^{-13}$ & $1.78 \cdot 10^{-14}$ & $9.51 \cdot 10^{-15}$ \\ \hline
        ~ & $0.9$ & $0.05$ & $0.999329571$ & $0.998398528$ & $0.595549345$ & $0.598721543$ \\ \hline
        ~ & $0.9$ & $0.075$ & $0.99314159$ & $0.988160254$ & $0.427339199$ & $0.446692967$ \\ \hline
        ~ & $0.95$ & $0.025$ & $0.999965279$ & $0.999905569$ & $0.77702589$ & $0.777162214$ \\ \hline
        ~ & $0.95$ & $0.03$ & $0.999947345$ & $0.999838756$ & $0.738238467$ & $0.73842378$ \\ \hline
        ~ & $0.99$ & $0.001$ & $0.999987548$ & $0.999989159$ & $0.990032709$ & $0.990034709$ \\ \hline
        ~ & $0.99$ & $0.005$ & $0.999988377$ & $0.999987723$ & $0.951143286$ & $0.951144104$ \\ \hline
        Binomial$(20,0.9)$ & $0.8$ & $0.1$ & $0.116600749$ & $2.74 \cdot 10^{-9}$ & $1.17 \cdot 10^{-12}$ & $1.63 \cdot 10^{-9}$ \\ \hline
        ~ & $0.8$ & $0.15$ & $9.29 \cdot 10^{-11}$ & $1.23 \cdot 10^{-14}$ & $9.15 \cdot 10^{-24}$ & $2.96 \cdot 10^{-15}$ \\ \hline
        ~ & $0.9$ & $0.05$ & $0.999772079$ & $0.999692862$ & $0.396461797$ & $0.39770508$ \\ \hline
        ~ & $0.9$ & $0.075$ & $0.365969238$ & $0.272115109$ & $5.91 \cdot 10^{-10}$ & $0.234283025$ \\ \hline
        ~ & $0.95$ & $0.025$ & $0.999999961$ & $0.9999999$ & $0.634358359$ & $0.634358947$ \\ \hline
        ~ & $0.95$ & $0.03$ & $0.999999659$ & $0.999999231$ & $0.57843156$ & $0.578435146$ \\ \hline
        ~ & $0.99$ & $0.001$ & $1$ & $1$ & $0.982153072$ & $0.982153072$ \\ \hline
        ~ & $0.99$ & $0.005$ & $1$ & $1$ & $0.913745576$ & $0.913745576$ \\ \hline
 \end{tabular}
\caption{Probabilities of draw for $\mathcal{G}(i,j,\kappa,T)[\phi]$, where $\kappa=3$, $(i,j) \in \{1,2\}^{2}$, $T$ has offspring distribution Binomial$(m,p)$, with $m=20$ and $p \in \{0.2,0.5,0.9\}$, and various values of $p_{0}$ and $p_{1}$ are considered.}\label{table:binomial}
\end{table}
\begin{table}[!ht]
    \centering
    \begin{tabular}{|l|l|l|l|l|l|l|}
    \hline
        Distribution & $p_{0}$ & $p_{1}$ & $d_{1,1}$ & $d_{1,2}$ & $d_{2,1}$ & $d_{2,2}$ \\ \hline
        Poisson$(2)$ & $0.8$ & $0.1$ & $2.17 \cdot 10^{-10}$ & $1.33 \cdot 10^{-10}$ & $1.02 \cdot 10^{-10}$ & $7.26 \cdot 10^{-11}$ \\ \hline
        ~ & $0.8$ & $0.15$ & $1.76 \cdot 10^{-10}$ & $5.88 \cdot 10^{-11}$ & $4.46 \cdot 10^{-11}$ & $1.63 \cdot 10^{-11}$ \\ \hline
        ~ & $0.9$ & $0.05$ & $2.66 \cdot 10^{-10}$ & $1.54 \cdot 10^{-10}$ & $1.28 \cdot 10^{-10}$ & $8.77 \cdot 10^{-11}$ \\ \hline
        ~ & $0.9$ & $0.075$ & $2.71 \cdot 10^{-10}$ & $8.39 \cdot 10^{-11}$ & $6.92 \cdot 10^{-11}$ & $2.46 \cdot 10^{-11}$ \\ \hline
        ~ & $0.95$ & $0.025$ & $3.63 \cdot 10^{-10}$ & $2.13 \cdot 10^{-10}$ & $1.89 \cdot 10^{-10}$ & $1.27 \cdot 10^{-10}$ \\ \hline
        ~ & $0.95$ & $0.03$ & $3.56 \cdot 10^{-10}$ & $1.69 \cdot 10^{-10}$ & $1.5 \cdot 10^{-10}$ & $8.08 \cdot 10^{-11}$ \\ \hline
        ~ & $0.99$ & $0.001$ & $3.53 \cdot 10^{-10}$ & $4.08 \cdot 10^{-10}$ & $4 \cdot 10^{-10}$ & $4.8 \cdot 10^{-10}$ \\ \hline
        ~ & $0.99$ & $0.005$ & $4.22 \cdot 10^{-10}$ & $3.24 \cdot 10^{-10}$ & $3.17 \cdot 10^{-10}$ & $2.51 \cdot 10^{-10}$ \\ \hline
        Poisson$(5)$ & $0.8$ & $0.1$ & $0.385781978$ & $0.156743617$ & $0.000537071$ & $0.177277163$ \\ \hline
        ~ & $0.8$ & $0.15$ & $7.57 \cdot 10^{-10}$ & $7.86 \cdot 10^{-13}$ & $9.36 \cdot 10^{-14}$ & $4.82 \cdot 10^{-14}$ \\ \hline
        ~ & $0.9$ & $0.05$ & $0.997336083$ & $0.994716126$ & $0.58606368$ & $0.59281647$ \\ \hline
        ~ & $0.9$ & $0.075$ & $0.982944993$ & $0.971342458$ & $0.387375146$ & $0.430096468$ \\ \hline
        ~ & $0.95$ & $0.025$ & $0.999253277$ & $0.998942475$ & $0.774883315$ & $0.775403991$ \\ \hline
        ~ & $0.95$ & $0.03$ & $0.999191886$ & $0.998651597$ & $0.735577616$ & $0.736130206$ \\ \hline
        ~ & $0.99$ & $0.001$ & $0.999455448$ & $0.999492983$ & $0.989459672$ & $0.989501555$ \\ \hline
        ~ & $0.99$ & $0.005$ & $0.999473495$ & $0.99946678$ & $0.950446251$ & $0.950453791$ \\ \hline
        Poisson$(10)$ & 
        $0.95$ & $0.025$ & $0.998824693$ & $0.999458906$ & $0.078497545$ & $0.081836495$ \\ \hline
        ~ & $0.95$ & $0.03$ & $0.073616655$ & $0.130342044$ & $7.7 \cdot 10^{-38}$ & $0.046636094$ \\ \hline
        ~ & $0.99$ & $0.001$ & $1$ & $1$ & $0.904837418$ & $0.904837418$ \\ \hline
        ~ & $0.99$ & $0.005$ & $1$ & $1$ & $0.60653066$ & $0.60653066$ \\ \hline
    \end{tabular}
\caption{Probabilities of draw for $\mathcal{G}(i,j,\kappa,T)[\phi]$, where $\kappa=3$, $(i,j) \in \{1,2\}^{2}$, $T$ has offspring distribution Poisson$(\lambda)$ for $\lambda \in \{2,5,10\}$, and various values of $p_{0}$ and $p_{1}$ are considered.}
\label{table:Poisson}
\end{table}

It is important to note here that any numerical estimation, no matter how precise, would involve some errors. In the simulations performed here, our task involves 
\begin{enumerate}
\item estimating all possible fixed points of the function $h$ within $[0,1]^{4}$, 
\item identifying the `largest' fixed point, i.e.\ the fixed point each of whose coordinates is larger than the corresponding coordinate of every other fixed point of $h$ in $[0,1]^{4}$ (that such a fixed point exists is ensured by the fact that, as proved in Equation~\eqref{smallest_largest_fixed_point} of Theorem~\ref{thm:1}, the `largest' fixed point equals the matrix 
$\begin{bmatrix}
1-w_{1,1} & 1-w_{1,2}\\
1-w_{2,1} & 1-w_{2,2}
\end{bmatrix}$),
\item identifying the `smallest' fixed point, i.e.\ the fixed point each of whose coordinates is smaller than the corresponding coordinate of every other fixed point of $h$ in $[0,1]^{4}$ (that such a fixed point exists is ensured by the fact that, as proved in Equation~\eqref{smallest_largest_fixed_point} of Theorem~\ref{thm:1}, the `smallest' fixed point equals the matrix 
$\begin{bmatrix}
\ell_{1,1} & \ell_{1,2}\\
\ell_{2,1} & \ell_{2,2}
\end{bmatrix}$),
\item and finally, obtaining the difference between these two fixed points, which, obviously, equals 
$\begin{bmatrix}
d_{1,1} & d_{1,2}\\
d_{2,1} & d_{2,2}
\end{bmatrix}$. 
\end{enumerate}
Evidently, when $h$ has a unique fixed point in $[0,1]^{4}$, the `largest' and `smallest' fixed points coincide, resulting in $d_{i,j}=0$ for each $(i,j) \in \{1,2\}^{2}$. In some of the rows of Tables~\ref{table:Dirac}, \ref{table:unif}, \ref{table:binomial} and \ref{table:Poisson} (for instance, in the rows corresponding to the offspring distribution Unif$[2]$ and $p_{0} \in \{0.8, 0.9\}$ in Table~\ref{table:unif}), we see that each entry of the difference between the `largest' and the `smallest' fixed points, i.e.\ the values of $d_{1,1}$, $d_{1,2}$, $d_{2,1}$ and $d_{2,2}$, is extremely small (of the order of magnitude $10^{-10}$, $10^{-11}$, $10^{-12}$ etc.). In such a situation, it is difficult for us to be certain whether there are actually multiple fixed points of the function $h$ in $[0,1]^{4}$, thereby leading to positive values for the probabilities of draw, or whether there is, in fact, only one fixed point of $h$ in $[0,1]^{4}$, and the apparent non-zero difference between the `largest' and the `smallest' fixed points is merely a result of fluctuations arising due to the method of numerical estimation used. On the other hand, there are instances where, despite one or more of the numerically estimated probabilities of draw being very small, we are able to conclude that they are indeed strictly positive and do not merely \emph{appear} to be positive due to errors in our estimation method. For example, when the offspring distribution is Binomial$(20,0.9)$, $p_{0}=0.9$ and $p_{1}=0.075$, we see that $d_{2,1}=5.91 \cdot 10^{-10}$ in Table~\ref{table:binomial} -- however, in this case, each of $d_{1,1}$, $d_{1,2}$ and $d_{2,2}$ assumes a significantly large positive value, so that by an application of Theorem~\ref{thm:2}, we can be certain that $d_{2,1}$ is indeed strictly positive.

In this context, it is fascinating to observe, from each of Tables~\ref{table:Dirac}, \ref{table:unif}, \ref{table:binomial} and \ref{table:Poisson}, that for a fixed offspring distribution, increasing the value of $p_{0}$ (which, in turn, results in a reduction in the value of at least one of $p_{1}$ and $p_{-1}$) \emph{usually} results in the probabilities of draw eventually becoming \emph{decisively} strictly positive. Here, by \emph{`decisively'} we mean that the value of at least one of $d_{1,1}$, $d_{1,2}$, $d_{2,1}$ and $d_{2,2}$ becomes significantly large -- large enough for us to be able to conclude, with confidence, that there are, indeed, multiple fixed points of $h$ in $[0,1]^{4}$, and that the probabilities of draw aren't merely \emph{appearing} to be positive due to errors of numerical estimation. Exceptions to this have also been observed, though: as seen from Table~\ref{table:Poisson}, when $T$ follows certain offspring distributions such as Poisson$(2)$, it may be the case that for all $p_{0} \in [0,1)$, and all $p_{1} \in [0,1-p_{0}]$, numerical estimations are unable to help us decide, with any extent of certainty, whether there is a unique fixed point of $h$ in $[0,1]^{4}$ or not.

Even though we find it challenging to conclude that the probability of draw, $d_{i,j}$, for any $(i,j) \in \{1,2\}^{2}$, increases as $p_{0}$ is increased, keeping the offspring distribution fixed (such a result seems untenable due to the adversarial element present in the game), it seems as if the following is true for every fixed offspring distribution $\chi$: Let $T$ be the rooted GW tree endowed with $\chi$. Note that the probability of the event that the game $\mathcal{G}(i,j,\kappa,T)[\phi]$ results in a draw relies heavily on the parameters $p_{0}$, $p_{1}$ and $p_{-1}$ (where $p_{-1}=1-p_{0}-p_{1}$). In order to emphasize this dependence, let us replace the notation $d_{i,j}$ by $d_{i,j}(p_{0},p_{1},p_{-1})$, for each $(i,j) \in \{1,2,\ldots,\kappa-1\}^{2}$. If there exist some $\rho_{0} \in [0,1)$ and some $\rho_{1} \in [0,1-\rho_{0}]$ such that $d_{i,j}(\rho_{0},\rho_{1},\rho_{-1})$, with $\rho_{-1}=1-\rho_{0}-\rho_{1}$, is strictly positive, then $d_{i,j}(p_{0},p_{1},p_{-1})$ ought to be strictly positive whenever
\begin{equation}
p_{0} \in (\rho_{0},1],\ p_{1} \in [0,\rho_{1}] \text{ and } p_{-1}=1-p_{0}-p_{1} \in [0,\rho_{-1}].\nonumber
\end{equation}
This can be posed as a conjecture that describes a \emph{phase transition phenomenon} occurring in the probability of draw in our game (when $(p_{0},p_{1})$ is varied appropriately while keeping the offspring distribution fixed). It can be substantiated via further simulations, and it deserves deeper and more intensive investigations.

Although we are not able to prove, analytically, that $d_{1,2}$ should be at least as large as $d_{2,1}$, we see that this is true in every row of each of Tables~\ref{table:Dirac}, \ref{table:unif}, \ref{table:binomial} and \ref{table:Poisson}. Despite the possibility of errors creeping into our numerical estimation, such a consistent domination of the estimated value of $d_{2,1}$ by the estimated value of $d_{1,2}$ leads us to state a second conjecture: that indeed, the true value of $d_{1,2}$ is greater than or equal to the true value of $d_{2,1}$. We further observe that, except for the cases where the offspring distribution is either Unif$[2]$ or Unif$[5]$, $p_{0}=0.99$ and $p_{1}=0.001$, the estimated value of $d_{1,1}$ exceeds the estimated value of $d_{2,2}$. There may be an inequality that holds in this case as well, but we refrain from posing a conjecture in this case because of the exceptions noted in Table~\ref{table:unif}.

\subsection{Simulations pertaining to Theorem \ref{thm:kappa=3}} We now present some numerical computations and simulations in order to illustrate the applicability of Theorem \ref{thm:kappa=3}, for various underlying offspring distributions of the rooted Galton-Watson tree under consideration. It is to be borne in mind that the criteria posited in Theorem~\ref{thm:kappa=3} is a \emph{sufficient} one for ensuring that the probability of draw, $d_{i,j}$, in each of the games $\mathcal{G}(i,j,3,T)[\phi]$, for $(i,j) \in \{1,2\}^{2}$, equals $0$, when $\kappa=3$. Therefore, if any of the $E_{i,j}$s, for $(i,j) \in \{1,2\}^{2}$, turns out to be at least as large as $1$, simulations prove to be the only useful means for examining whether the corresponding function $F$ (as defined via \eqref{ell_{1,1}}, \eqref{ell_{1,2}}, \eqref{ell_{2,1}} and \eqref{ell_{2,2}} in \S\ref{sec:kappa=3_proofs}) has multiple fixed points in $[0,1]^{4}$ (which, in turn, would imply that at least one of the probabilities $d_{i,j}$, for $(i,j) \in \{1,2\}^{2}$, is strictly positive).

For the values of $\lambda$, $p_{1}$ and $p_{-1}$ (the last two determine the value of $p_{0}=1-p_{-1}-p_{1}$) listed in Table~\ref{kappa=3_Poisson_11}, we show that $E_{i,j} < 1$ for each $(i,j) \in \{1,2\}^{2}$, thereby establishing, via Theorem~\ref{thm:kappa=3}, that $d_{i,j}=0$ for each $(i,j) \in \{1,2\}^{2}$ when $\kappa=3$, the offspring distribution of $T$ is Poisson$(\lambda)$, and the edge-weights worth $-1$, $0$ and $1$ are assigned to each edge of $T$ with probabilities $p_{-1}$, $p_{0}$ and $p_{1}$ respectively. For instance, when $\lambda=5$ and $p_{-1}=p_{1}=0.35$, we have 
\begin{align}
{}&A_{1,1}=A_{1,2}=G(p_{-1})=e^{5(0.35-1)} \approx 0.039,\nonumber\\
{}&B_{2,1}=B_{2,2}=G(1-p_{1})=e^{-5 \cdot 0.35}\approx 0.174,\nonumber
\end{align}
and slightly more involved computations yield
\begin{align}
A_{2,1}={}&G\left[p_{-1}\left\{1-G\left[1-\left(1-p_{-1}\right)G\left(p_{-1}\right)\right]\right\}\right]\nonumber\\
={}&\exp\left\{5\left(0.35\left\{1-\exp\left\{-5 \cdot (1-0.35)e^{5(0.35-1)}\right\}\right\}-1\right)\right\} \approx 0.00828,\nonumber\\
A_{2,2}={}&G\left[p_{-1}\left\{1-G\left\{1-G\left(p_{-1}\right)\right\}\right\}\right]+G\left[\left(1-p_{1}\right)\left\{1-G\left(1-p_{1}\right)\right\}\right]\nonumber\\&-G\left[p_{-1}\left[1-G\left(1-p_{1}\right)-G\left[1-G\left(p_{-1}\right)\right]+G\left[\left(1-p_{1}\right)\left\{1-G\left(p_{-1}\right)\right\}\right]\right]\right]\nonumber\\
={}&\exp\left\{5\left(0.35\left\{1-\exp\left\{-5 e^{5(0.35-1)}\right\}\right\}-1\right)\right\}+\exp\left\{5\left((1-0.35)\left\{1-e^{-5 \cdot 0.35}\right\}-1\right)\right\}\nonumber\\&-\exp\left\{5\left(0.35\left[1-e^{-5 \cdot 0.35}-\exp\left\{-5 \cdot e^{5(0.35-1)}\right\}+\exp\left\{5\left((1-0.35)\left\{1-e^{5(0.35-1)}\right\}-1\right)\right\}\right]-1\right)\right\}\nonumber\\& \approx 0.0991,\nonumber\\
B_{1,1}={}&G\left[1-\left(1-p_{-1}\right)G\left(p_{-1}\right)\right]-G\left[\left(1-p_{-1}\right)\left\{1-G\left(p_{-1}\right)\right\}+p_{-1}\right]+G[\left(1-p_{-1}\right)\left\{1-G\left(p_{-1}\right)\right\}\nonumber\\&+p_{-1}-p_{1}G\left\{1-G\left(1-p_{1}\right)\right\}+p_{1}G\left[p_{-1}\left\{1-G\left(1-p_{1}\right)\right\}\right]]\nonumber\\
={}&\exp\left\{-5(1-0.35) e^{5(0.35-1)}\right\}-\exp\left\{5\left((1-0.35)\left\{1-e^{5(0.35-1)}\right\}+0.35-1\right)\right\}\nonumber\\&+\exp\Big\{5\Big((1-0.35)\left\{1-e^{5(0.35-1)}\right\}+0.35-0.35\exp\left\{-5e^{-5 \cdot 0.35}\right\}+0.35\exp\left\{5\left(0.35\left\{1-e^{-5 \cdot 0.35}\right\}-1\right)\right\}\nonumber\\&-1\Big)\Big\} \approx 0.44488.\nonumber\\
B_{1,2}={}&G\left[1-p_{1}G\left[\left(1-p_{1}\right)\left\{1-G\left(1-p_{1}\right)\right\}\right]\right]=\exp\left\{-5 \cdot 0.35 \exp\left\{5\left((1-0.35)\left\{1-e^{-5 \cdot 0.35}\right\}-1\right)\right\}\right\}\nonumber\\& \approx 0.84123.\nonumber
\end{align}
Having computed these values, we plug them into the expressions for $E_{i,j}$ defined in \eqref{E_{i,j}_defn}, for each $(i,j) \in \{1,2\}^{2}$, to obtain $E_{1,1}\approx 0.68$, $E_{1,2}\approx 0.773$, $E_{2,1} \approx 0.749$ and $E_{2,2}\approx 0.851$, as stated in the corresponding row of Table~\ref{kappa=3_Poisson_11}. 

On the other hand, for the values of $\lambda$, $p_{1}$ and $p_{-1}$ listed in Table~\ref{kappa=3_Poisson_13}, at least one $E_{i,j}$, for $(i,j) \in \{1,2\}$ and $\kappa=3$, exceeds $1$. This forces us to further investigate the corresponding function $h$ (see \eqref{fixed_point_eq} for the definition of $h$), and examine, \emph{numerically}, whether it has multiple fixed points in $\mathcal{S}$ (which, in this case, is the space of all $2 \times 2$ matrices with each entry in $[0,1]$, as defined right before the statement of Theorem~\ref{thm:1}). In the proof of Theorem~\ref{kappa=3_Poisson_13}, outlined in \S\ref{sec:kappa=3_proofs}, we work with the function $F$ (see \eqref{ell_{1,1}}, \eqref{ell_{1,2}}, \eqref{ell_{2,1}} and \eqref{ell_{2,2}}), which is nothing but $h$ written as a $4 \times 1$ column vector as opposed to a $2 \times 2$ matrix, and we examine the fixed points of $F$ in $[0,1]^{4}$ as opposed to those of $h$ in $\mathcal{S}$. In the rest of \S\ref{sec:simulations}, we refer to this function $F$ as opposed to the function $h$. In some of the situations considered in Table~\ref{kappa=3_Poisson_13}, simulations have shown the existence of multiple fixed points of $F$ in $[0,1]^{4}$, thereby proving that $d_{i,j} > 0$ for each $(i,j) \in \{1,2\}^{2}$, while in the remainder of such situations, simulations indicate that $F$ has a unique fixed point in $[0,1]^{4}$. However, despite the fact that Theorem~\ref{thm:kappa=3} is not applicable in the situations listed in Table~\ref{kappa=3_Poisson_13}, a rough pattern becomes visible from this table. In particular, we note that when at least one $E_{i,j}$ is ``significantly" larger than $1$ (what counts as ``significant" depends, possibly among other aspects, on the value of $\lambda$ under consideration), simulations indicate that the function $F$ has multiple fixed points in $[0,1]^{4}$. For instance, when $\lambda=5$ and $p_{1}=p_{-1}=0.1$, computations show that $\max\{E_{i,j}: (i,j) \in \{1,2\}^{2}\} \approx 14.109$, and when $\lambda=10$ and $p_{1}=p_{-1}=0.1$, computations yield $\max\{E_{i,j}: (i,j) \in \{1,2\}^{2}\} \approx 61.853$, and in each of these cases, simulations reveal that $F$ has multiple fixed points in $[0,1]^{4}$, leading to strictly positive probabilities of draw. But for $\lambda=5$ and $(p_{1},p_{-1})=(0.3,0.1)$, computations yield $\max\{E_{i,j}: (i,j) \in \{1,2\}^{2}\} =4.345$, which, even though it exceeds $1$, is much closer to $1$ compared to the previous two cases, and in this case, simulations reveal that $F$ has a unique fixed point in $[0,1]^{4}$.

\begin{table}[]
\begin{tabular}{|l|l|l|l|l|l|l|l|l|l|l|l|l|l|l|l|} \hline
$\lambda$ & $p_1$ & $p_{-1}$ & $p_0$  & $E_{11}$ & $E_{12}$ & $E_{21}$ & $E_{22}$ \\ \hline 
5 &	0.34 &	0.34 &	0.32 &	0.818 &	0.861 &	0.818 &	0.86 \\
5         & 0.35   & 0.35      & 0.3      & 0.68 &	0.773 &	0.749 &	0.851 \\
5 &	0.375 &	0.375 &	0.25 &	0.414 &	0.579 &	0.594 &	0.827\\
\hline
7 &	0.34 &	0.34 &	0.32 &	0.93 &	0.981 &	0.894 &	0.943 \\
7         & 0.35  & 0.35     & 0.3      & 0.745 &	0.854 &	0.79 &	0.905 \\ 
7 &	0.375 &	0.375 &	0.25 &	0.416 &	0.592 &	0.576 &	0.818 \\
\hline 
10 &	0.34 &	0.34 &	0.32 &	0.628 &	0.663 &	0.615 &	0.649 \\
10        & 0.35  & 0.35     & 0.3       & 0.484 &	0.556 &	0.524 &	0.601 \\
10 &	0.375 &	0.375 &	0.25 &	0.247 &	0.353 &	0.348 &	0.497 \\
\hline
15        & 0.3  & 0.3     & 0.4     & 0.63 &	0.484 & 0.452 & 0.347 \\
15 &  0.325 &  0.325 &  0.35 &0.302 &0.282 &0.271 &0.253 \\
15        & 0.35  & 0.35     & 0.3     & 0.142 &	0.164 &	0.162 &	0.186 \\ \hline
20 &	0.3&	0.3&	0.4&	0.196&	0.151&	0.145&	0.111 \\
20 & 0.325 & 0.325 & 0.35 &0.086 &0.081 &0.079 &0.074\\
20 &	0.35 &	0.35 &	0.3	& 0.037 &	0.043 &	0.043 & 0.05 \\ \hline
25 & 0.25 & 0.25 &  0.5 & 0.373 & 0.196 & 0.184 & 0.097 \\
  25 & 0.30 & 0.30 &  0.4 & 0.061 & 0.047 & 0.045 & 0.035 \\
   25 & 0.35 & 0.35 &  0.3 & 0.010 & 0.011 & 0.011 & 0.013 \\ \hline
50 &	0.25 &	0.25 &	0.5 &	0.002 &	0.001 &	0.001 & 0.001 \\
50 &	0.3 &	0.3 &	0.4 &	0 &	0 &	0 &	0 \\ 
50 &	0.35 &	0.35 &	0.3 &	0 &	0 &	0 &	0 \\ \hline 
\end{tabular}
\caption{Here, $\kappa=3$, the offspring distribution of $T$ is Poisson$(\lambda)$, and $E_{i,j} < 1$ for each $(i,j) \in \{1,2\}^{2}$, so that by Theorem~\ref{thm:kappa=3}, we have $d_{i,j}=0$ for each $(i,j) \in \{1,2\}^{2}$.}\label{kappa=3_Poisson_11}
\end{table}


\begin{table}[]
\begin{tabular}{|l|l|l|l|l|l|l|l|l|l|l|l|l|l|l|l|l|} \hline
$\lambda$ & $p_1$ & $p_{-1}$ & $p_0$  & $\max E_{i,j}$ & Number of fixed points \\ \hline 
5         & 0.1  & 0.025     & 0.875       & 17.271 & 2      \\
5         & 0.1  & 0.05     & 0.85       & 16.186 & 2      \\
5         & 0.1   & 0.1      & 0.8      & 14.109
 &  2  \\ 
5         & 0.1  & 0.2     & 0.7      & 10.346
 & 1 \\
5         & 0.1   & 0.3      & 0.6      & 7.155
 & 1  \\
 \hline
 10        & 0.1  & 0.025     & 0.875      & 74.461
 &  2 \\
10        & 0.1  & 0.05     & 0.85      & 70.138
 &  2 \\
10        & 0.1   & 0.1      & 0.8      & 61.853
  & 2 \\ 
10        & 0.1  & 0.2 & 0.7      & 46.733
  & 1    \\ 
10 &	0.1 &	0.3 & 0.6 & 33.576
  & 1 \\
  \hline
5         & 0.025  & 0.1     & 0.875       & 17.455  &    6   \\
5         & 0.05  & 0.1     & 0.85       & 16.381 &    6   \\
5         & 0.09   & 0.1      & 0.81      & 14.580
 & 2   \\ 
5         & 0.2  & 0.1     & 0.7      & 9.001
& 1 \\
5         & 0.3   & 0.1      & 0.6      & 4.345
 & 1  \\
 \hline
10        & 0.025  & 0.1     & 0.875      &  77.057
&  6 \\
10        & 0.05  & 0.1     & 0.85      & 72.454
&  6 \\
10        & 0.09   & 0.1      & 0.81      & 64.288
& 2 \\ 
10        & 0.2  & 0.1 & 0.7      & 26.706
  & 1    \\ 
10 &	0.3 &	0.1 & 0.6 & 4.566
  & 1 \\
  \hline 
\end{tabular}
\caption{Here, $\kappa=3$, the offspring distribution of $T$ is Poisson$(\lambda)$, and at least one $E_{i,j}$ exceeds $1$, so that Theorem~\ref{thm:kappa=3} is not applicable. Further simulations reveal whether the function $F$, defined in \S\ref{sec:kappa=3_proofs}, has multiple fixed points in $[0,1]^{4}$ or not.}\label{kappa=3_Poisson_13}
\end{table}


We consider a rooted $d$-regular tree, i.e.\ a tree in which each vertex, including the root $\phi$, has precisely $d$ children (in other words, the offspring distribution is $\delta_{d}$, the Dirac delta measure that assigns probability $1$ to $d$), in Table~\ref{regular_kappa=3_11}. For each value of $d$ and the tuple $(p_{-1},p_{0},p_{1})$ listed in this table, computations yield $E_{i,j} < 1$ for each $(i,j) \in \{1,2\}^{2}$, thus letting us conclude, via Theorem~\ref{thm:kappa=3}, that the probability $d_{i,j}$ of draw in the game $\mathcal{G}(i,j,3,T_{d})[\phi]$ equals $0$ for each $(i,j) \in \{1,2\}^{2}$. For instance, when $d=15$ and $p_{-1}=p_{1}=0.35$, we have
\begin{align}
    {}&A_{1,1}=A_{1,2}=G(p_{-1})=(0.35)^{15} \approx 1.4488 \times 10^{-7},\nonumber\\
    {}&B_{2,1}=B_{2,2}=G(1-p_{1})=(1-0.35)^{15} \approx 0.0015620,\nonumber
\end{align}
and
\begin{align}
A_{2,1}={}&G\left[p_{-1}\left\{1-G\left[1-\left(1-p_{-1}\right)G\left(p_{-1}\right)\right]\right\}\right]=\left[0.35\left\{1-\left[1-(1-0.35)(0.35)^{15}\right]^{15}\right\}\right]^{15} \approx 2.57866 \times 10^{-95},\nonumber\\
A_{2,2}={}&G\left[p_{-1}\left\{1-G\left\{1-G\left(p_{-1}\right)\right\}\right\}\right]+G\left[\left(1-p_{1}\right)\left\{1-G\left(1-p_{1}\right)\right\}\right]\nonumber\\&-G\left[p_{-1}\left[1-G\left(1-p_{1}\right)-G\left[1-G\left(p_{-1}\right)\right]+G\left[\left(1-p_{1}\right)\left\{1-G\left(p_{-1}\right)\right\}\right]\right]\right],\nonumber\\
={}&\left[0.35\left\{1-\left\{1-(0.35)^{15}\right\}^{15}\right\}\right]^{15}+\left[(1-0.35)\left\{1-(1-0.35)^{15}\right\}\right]^{15}\nonumber\\&-\left[0.35\left[1-(1-0.35)^{15}-\left[1-(0.35)^{15}\right]^{15}+\left[(1-0.35)\left\{1-(0.35)^{15}\right\}\right]^{15}\right]\right]^{15} \approx 0.0015258,\nonumber\\
B_{1,1}={}&G\left[1-\left(1-p_{-1}\right)G\left(p_{-1}\right)\right]-G\left[\left(1-p_{-1}\right)\left\{1-G\left(p_{-1}\right)\right\}+p_{-1}\right]+G[\left(1-p_{-1}\right)\left\{1-G\left(p_{-1}\right)\right\}\nonumber\\&+p_{-1}-p_{1}G\left\{1-G\left(1-p_{1}\right)\right\}+p_{1}G\left[p_{-1}\left\{1-G\left(1-p_{1}\right)\right\}\right]]\nonumber\\
={}&\left[1-(1-0.35)(0.35)^{15}\right]^{15}-\left[(1-0.35)\left\{1-(0.35)^{15}\right\}+0.35\right]^{15}+\Big[(1-0.35)\left\{1-(0.35)^{15}\right\}\nonumber\\&+0.35-0.35\left\{1-(1-0.35)^{15}\right\}^{15}+0.35\left[0.35\left\{1-(1-0.35)^{15}\right\}\right]^{15}\Big]^{15} \approx 0.00188145,\nonumber\\
B_{1,2}={}&G\left[1-p_{1}G\left[\left(1-p_{1}\right)\left\{1-G\left(1-p_{1}\right)\right\}\right]\right]=\left[1-0.35\left[(1-0.35)\left\{1-(1-0.35)^{15}\right\}\right]^{15}\right]^{15} \approx 0.992019.\nonumber
\end{align}
We now plug in the values computed above into the expression fo $E_{i,j}$ defined in \eqref{E_{i,j}_defn}, for each $(i,j) \in \{1,2\}^{2}$, and find, as stated in Table~\ref{regular_kappa=3_11}, that $E_{i,j} < 1$ for each $(i,j) \in \{1,2\}^{2}$, thus allowing us to conclude, via Theorem~\ref{thm:kappa=3}, that the probability $d_{i,j}$ of the game $\mathcal{G}(i,j,3,T_{d})[\phi]$ resulting in a draw equals $0$ when $d=15$ and $p_{-1}=p_{1}=0.35$.

On the other hand, in Table~\ref{regular_kappa=3_13}, we list some values of $d$, $p_{-1}$ and $p_{1}$ (of course, $p_{0}=1-p_{-1}-p_{1}$ is then already determined) for which at least one $E_{i,j}$ is strictly greater than $1$, where $(i,j) \in \{1,2\}^{2}$. Consequently, Theorem~\ref{thm:kappa=3} cannot be applied, since, as also mentioned earlier, the criterion stated in Theorem~\ref{thm:kappa=3} is a \emph{sufficient} one for guaranteeing that the function $F$, defined via \eqref{ell_{1,1}}, \eqref{ell_{1,2}}, \eqref{ell_{2,1}} and \eqref{ell_{2,2}} in \S\ref{sec:kappa=3_proofs}, has a unique fixed point in $[0,1]^{4}$, but not \emph{necessary}. Therefore, we need to employ numerical analysis to understand better the properties of the function $F$, and the rightmost column states whether our numerical analysis of $F$ has indicated that it has multiple fixed points in $[0,1]^{4}$ or not (and it also states the number of fixed points found via numerical analysis when there are more than one of them).  

As discussed with regard to Table~\ref{kappa=3_Poisson_13}, while Theorem~\ref{thm:kappa=3} is not applicable to the situations listed in Table~\ref{regular_kappa=3_13}, a pattern can still be noted from the findings here. For instance, while $\max\{E_{i,j}: (i,j) \in \{1,2\}^{2}\}$ is strictly greater than $1$ in each of the rows in this table, in some cases, this value is ``significantly" higher than $1$ (as in the case of Table~\ref{kappa=3_Poisson_13}, what we mean by ``significant" here depends on the value of $d$ under consideration), and the corresponding function $F$ turns out to have multiple fixed points in $[0,1]^{4}$. In some other cases, $\max\{E_{i,j}: (i,j) \in \{1,2\}^{2}\}$ exceeds $1$ but not by a large margin, and there, numerical analysis indicates that $F$ has a unique fixed point in $[0,1]^{4}$. For instance, when $d=5$ and $p_{-1}=p_{1}=0.1$, we have $\max\{E_{i,j}: (i,j) \in \{1,2\}^{2}\} \approx 16.09$, and numerical investigation reveals that $F$ has $4$ distinct fixed points in $[0,1]^{4}$, whereas when $d=10$ and $(p_{1},p_{-1})=(0.3,0.1)$, we have $\max\{E_{i,j}: (i,j) \in \{1,2\}^{2}\} \approx 4.988$ -- a value much closer to $1$ than that in the previous case -- and numerical investigations indicate that $F$ has a unique fixed point in $[0,1]^{4}$.   

Table~\ref{kappa=3_Poisson_13} and Table~\ref{regular_kappa=3_13} hint at another interesting pattern in the probability of draw: fixing the value of $p_1$ (respectively, $p_{-1}$) and increasing the value of $p_{-1}$ (respectively, $p_1$) (and therefore, decreasing the value of $p_0$), leads to a decrease in the value of $\max E_{ij}$ and, in turn, in the number of fixed points of $F$ in $[0,1]^{4}$. As alluded to earlier, simulations seem to indicate that as we keep decreasing the value of $\max E_{ij}$ for fixed values of the parameter(s) underlying the offspring distribution under consideration, eventually, the function $F$ ends up having a unique fixed point in $[0,1]^{4}$. The patterns observed in Table~\ref{kappa=3_Poisson_13} and Table~\ref{regular_kappa=3_13} may even indicate that the probability of draw decreases monotonically as we increase $p_{1}$ (respectively, $p_{-1}$) while keeping $p_{-1}$ (respectively, $p_{1})$ fixed (and consequently, letting $p_{0}$ decrease). A different, and possibly far more direct, intuition for such a surmise is applicable at least for $\kappa=2$, as follows. Via a standard coupling argument, increasing $p_{1}$ to $p'_{1}$ while keeping $p_{-1}$ fixed (so that $p_{0}$ decreases to $p'_{0}=(1-p'_{1}-p_{-1})$) amounts to
\begin{enumerate}
    \item first assigning, independently, to each edge of our rooted Galton-Watson tree, an edge-weight that equals $+1$ with probability $p_{1}$, $-1$ with probability $p_{-1}$, and $0$ with probability $p_{0}=1-p_{-1}-p_{1})$,
    \item then changing the weight of each edge that previously bore the edge-weight $0$, independent of all else, into $+1$ with probability $(p'_{1}-p_{1})p_{0}^{-1}$. 
\end{enumerate}
All other edge-weights are kept unchanged. It is evident from this coupling that once the above-mentioned reassignment of weights has been performed, the number of edges with edge-weight equal to $0$ decreases. As our game continues for as long as the token keeps being moved along edges that bear edge-weight $0$, it seems intuitive that such a reassignment of edge-weights would decrease the chances of the game resulting in a draw. As mentioned earlier, this intuition is even stronger when $\kappa=2$, since, in this case, the game ends the moment an edge bearing edge-weight $\in \{-1,+1\}$ is traversed by the token. However, it is hard to find the right coupling to prove such a conjecture due to the adversarial nature of the game. 

\begin{table}[]
\begin{tabular}{|l|l|l|l|l|l|l|l|} \hline
$d$ & $p_1$ & $p_{-1}$ & $p_0$ & $E_{11}$ & $E_{12}$ & $E_{21}$ & $E_{22}$ \\ \hline
5 & 0.340 & 0.340 & 0.32 & 0.901 & 0.950 & 0.884 & 0.932 \\
5 &	0.35 &	0.35 &	0.3 &	0.732 &	0.837 &	0.789 &	0.904 \\ 
5 & 0.375 & 0.375 & 0.25 & 0.419 & 0.596 & 0.586 & 0.834 \\
\hline
7 & 0.340 & 0.340 & 0.32 & 0.755 & 0.798 & 0.744 & 0.786 \\
7 &	0.35 &	0.35 &	0.3 &	0.588 &	0.676 &	0.638 &	0.734 \\
7 & 0.375 & 0.375 & 0.25 & 0.304 & 0.438 & 0.429 & 0.617 \\
\hline
10 & 0.340 & 0.340 & 0.32 & 0.371 & 0.392 & 0.378 & 0.400 \\
10 &	0.35 &	0.35 &	0.3 &	0.275 &	0.317 &	0.309 &	0.356 \\
10 & 0.375 & 0.375 & 0.25 & 0.126 & 0.182 & 0.184 & 0.266 \\
\hline
15 &	0.3	& 0.3 &	0.4 &	0.22&	0.46&	0.237&	0.181 \\
15 & 0.325 & 0.325 & 0.35 & 0.137 & 0.128 & 0.126 & 0.117\\
15 &	0.35&	0.35&	0.3&	0.056&	0.065&	0.065&	0.075 \\
\hline
20	&0.25&	0.25&	0.5	&0.551&	0.287&	0.272&	0.142 \\
20 & 0.325 & 0.325 & 0.35 & 0.030 & 0.028 & 0.028 & 0.026\\
20 &	0.3 &	0.3 & 	0.4 &	0.081	&0.062 &	0.061 &	0.046 \\
20 &	0.35&	0.35 &	0.3 &	0.011 &	0.012 &	0.012 &	0.014 \\ \hline
25 &	0.25 &	0.25 &	0.5 &	0.176 &	0.091 &	0.087 &	0.045 \\
25	& 0.3 &	0.3 &	0.4 &	0.02 &	0.015 &	0.015 &	0.011 \\
25 &	0.35 &	0.35 &	0.3 &	0.002 &	0.002 &	0.002 &	0.003 \\ \hline
30 & 0.25 & 0.25 &  0.5 & 0.056 & 0.029 & 0.028 & 0.015 \\
 30 & 0.30 & 0.30 & 0.4 & 0.005 & 0.004 & 0.004 & 0.003 \\
30 & 0.35 & 0.35 &  0.3  & 0.000 &  0.000 &  0.000 &  0.000 \\
\hline
\end{tabular}
\caption{Here, $\kappa=3$ and the rooted tree under consideration is $T_{d}$, and for each tuple $(d, p_{1},p_{-1}, p_{0})$ listed in this table, we have $E_{i,j} < 1$ for each $(i,j) \in \{1,2\}^{2}$, so that by Theorem~\ref{thm:kappa=3}, we conclude that $d_{i,j}=0$ for each $(i,j) \in \{1,2\}^{2}$ in each of these cases.}\label{regular_kappa=3_11}
\end{table}

\begin{table}[]
\begin{tabular}{|l|l|l|l|l|l|l|l|l|l|l|l|l|l|l|l|l|} \hline
$d$ & $p_1$ & $p_{-1}$ & $p_0$  & $\max E_{i,j}$ & Number of fixed points \\ \hline 
5         & 0.1  & 0.05     & 0.85       & 18.07  & 6      \\
5         & 0.1   & 0.1      & 0.8      & 16.09
 &  4  \\ 
5         & 0.1  & 0.2     & 0.7      & 12.438 
 & 2 \\
5         & 0.1   & 0.3      & 0.6      &  9.158
 & 2  \\
 5         & 0.1  & 0.35     & 0.55      & 7.65
  & 1    \\ 
 \hline
10        & 0.1  & 0.05     & 0.85      & 71.242
 &  4 \\
10        & 0.1   & 0.1      & 0.8      & 62.957
  & 2 \\ 
10        & 0.1  & 0.2 & 0.7      & 47.836
  & 2    \\ 
10 &	0.1 &	0.3 & 0.6 & 34.696
  & 2 \\
 10 &	0.1	& 0.35 & 0.55 & 28.888
   & 2 \\ 
  \hline
   5         & 0.025  & 0.1     & 0.875       & 19.524   &  6     \\
  5         & 0.05  & 0.1     & 0.85       & 18.381  &    6   \\
5         & 0.09   & 0.1      & 0.81      &  16.558
 &   6 \\ 
5         & 0.2  & 0.1     & 0.7      & 10.73
& 1 \\
5         & 0.3   & 0.1      & 0.6      & 4.988
 & 1  \\
 \hline
 10        & 0.025  & 0.1     & 0.875      &  77.242
& 6   \\
10        & 0.05  & 0.1     & 0.85      & 72.745
&  6 \\
10        & 0.09   & 0.1      & 0.81      &  65.218
& 4 \\ 
10        & 0.2  & 0.1 & 0.7      & 24.927
  & 1    \\ 
10 &	0.3 &	0.1 & 0.6 & 2.95
  & 1 \\
  \hline 
\end{tabular}
\caption{For each tuple $(d,p_{-1},p_{0},p_{1})$ listed in this table, we find that at least one $E_{i,j}$, for $(i,j) \in \{1,2\}^{2}$, exceeds $1$, forcing us to investigate the function $F$ numerically. The last column states whether numerical analysis indicates the function $F$ has multiple fixed points in $[0,1]^{4}$ or not.}\label{regular_kappa=3_13}
\end{table}

\section{Formal Proofs}\label{sec:formal_proofs}
\subsection{Notations and preliminary observations}\label{subsec:notations_definitions}
We analyze the weighted-tree-percolation game for a fixed (but arbitrary) $\kappa$, which allows us to drop the parameter $\kappa$ from the notation $\mathcal{G}(i,j,\kappa,t)[v^*]$. Furthermore, when the initial vertex $v^*$ is clear from the context (such as when $v^{*}$ is the root $\phi$ (or $\phi(t)$) of $T$ (or $t$)), we omit $v^{*}$ from our notation as well. Thus, in such situations, our notation gets shortened to $\mathcal{G}(i,j,t)$ (or $\mathcal{G}(i,j,T)$, when the game is considered on the rooted edge-weighted GW tree $T$). Finally, when it is clear which random premise, i.e.\ $T$, the game is being played on (this entails the knowledge of the offspring distribution $\chi$ and the values of the edge-weight probabilities as defined in \eqref{edge_weight_prob}, i.e.\ $p_{-1}$ and $p_{0}$), we drop $T$ from our notation and are left with simply $\mathcal{G}(i,j)$. 

Given any (deterministic) rooted edge-weighted tree $t$, and $i,j \in \{1,2,\ldots,\kappa-1\}$, we partition the vertex set $V(t)$ of $t$ into subsets -- introducing these subsets will prove helpful in formulating the recurrence relations arising from our game, and these, in turn, are indispensable in the analysis of our game. 
\begin{enumerate}
\item We let $W_{i,j}(t)$ denote the set of all vertices $u \in V(t)$ such that $P_{1}$ wins the game $\mathcal{G}(i,j,t)[u]$.
\item We let $L_{i,j}(t)$ denote the set of all vertices $u \in V(t)$ such that $P_{1}$ loses the game $\mathcal{G}(i,j, t)[u]$.
\item We let $D_{i,j}(t)$ denote the set of all vertices $u \in V(t)$ such that $\mathcal{G}(i,j, t)[u]$ culminates in a draw. 
\end{enumerate}
Evidently, $W_{i,j}(t)$, $L_{i,j}(t)$ and $D_{i,j}(t)$ are pairwise disjoint and their union yields $V(t)$. Next, for $i,j \in \{1,2,\ldots,\kappa-1\}$ and for each $n \in \mathbb{N}$,  
\begin{enumerate}
\item we let $W_{i,j}^{(n)}(t) \subseteq W_{i,j}(t)$ denote the set of all $u \in V(t)$ such that the game $\mathcal{G}(i,j,t)[u]$ ends in less than $n$ rounds, 
\item we let $L_{i,j}^{(n)}(t) \subseteq L_{i,j}(t)$ denote the set of all $u \in V(t)$ such that the game $\mathcal{G}(i,j, t)[u]$ ends in less than $n$ rounds, 
\item we let $D_{i,j}^{(n)}(t)$ denote the set of all $u \in V(t)$ such that $\mathcal{G}(i,j, t)[u]$ lasts for at least $n$ rounds.
\end{enumerate}
We set $W_{i,j}^{(0)}(t) = L_{i,j}^{(0)}(t) = \emptyset$. Once again, $W_{i,j}^{(n)}(t)$, $L_{i,j}^{(n)}(t)$ and $D_{i,j}^{(n)}(t)$ form a partition of $V(t)$, for each $n$. We further note, following our remarks towards the end of \S\ref{subsec:deterministic_game}, that $W_{\kappa,j}^{(0)}(t) = W_{\kappa,j}(t) = V(t)$ for each $j \in \{0,1,\ldots,\kappa-1\}$, and $W_{i,0}^{(0)}(t)  = W_{i,0}(t) = V(t)$ for each $i \in \{1,2,\ldots,\kappa\}$. Likewise, $L_{i,\kappa}^{(0)}(t) = L_{i,\kappa}(t) = V(t)$ for each $i \in \{0,1,\ldots,\kappa-1\}$, and $L_{0,j}^{(0)}(t) = L_{0,j}(t) = V(t)$ for each $j \in \{1,2,\ldots,\kappa\}$.

In an analogous manner, we define the corresponding \emph{random} subsets of vertices, namely $W_{i,j}(T)$, $L_{i,j}(T)$, $D_{i,j}(T)$, $W_{i,j}^{(n)}(T)$, $L_{i,j}^{(n)}(T)$ and $D_{i,j}^{(n)}(T)$, of the rooted, edge-weight GW tree $T$, for all $i,j \in \{1,2,\ldots,\kappa-1\}$ and all $n \in \mathbb{N}_{0}$. 

Having defined these subsets, it is now time to define the corresponding probabilities, as these will constitute the recursive distributional equations derived from the recurrence relations involving the above-mentioned subsets. We let $w_{i,j}$ denote the probability of the event that the root $\phi$ of $T$ belongs to $W_{i,j}(T)$, and $w_{i,j}^{(n)}$ the probability of the event that $\phi$ is in $W_{i,j}^{(n)}(T)$, for each $n \in \mathbb{N}_{0}$. Formally, 
\begin{equation}
w_{i,j} = \Prob\left[\left\{t : \phi(t) \in W_{i,j}(t)\right\}\right] \text{ and } w^{(n)}_{i,j} = \Prob\left[\left\{t : \phi(t) \in W^{(n)}_{i,j}(t)\right\}\right],\label{win_probab_defn}
\end{equation}
where $\Prob$ indicates the probability measure induced by the rooted, edge-weighted GW tree $T$ on the space of all rooted, edge-weighted, locally finite trees with each edge-weight belonging to the set $\{-1,0,1\}$. Likewise, we define the probabilities $\ell_{i,j}$, $\ell_{i,j}^{(n)}$, $d_{i,j}$ and $d_{i,j}^{(n)}$, for each $n \in \mathbb{N}_{0}$. It is worthwhile to note here that the probabilities $w_{i,j}$ and $\ell_{i,j}$, for $i,j \in \{1,2,\ldots,\kappa-1\}$, are the same as those described in the statement of Theorem~\ref{thm:1}. Furthermore, from the discussion above, it is clear that $w_{\kappa,j}^{(0)}=w_{\kappa,j}=1$ for each $j \in \{0,1,\ldots,\kappa-1\}$, and $w_{i,0}^{(0)}=w_{i,0}=1$ for each $i \in \{1,2,\ldots,\kappa\}$. Likewise, $\ell_{i,\kappa}^{(0)}=\ell_{i,\kappa}=1$ for each $i \in \{0,1,\ldots,\kappa-1\}$ and $\ell_{0,j}^{(0)}=\ell_{0,j}=1$ for each $j \in \{1,2,\ldots,\kappa\}$.

We now state a lemma whose proof is straightforward and hence has been omitted, but whose usefulness is evident in the derivation of the recurrence relations and the subsequent recursive distributional equations in \S\ref{subsec:recurrence}. 
\begin{lemma}\label{lem:containment}
For any rooted, edge-weighted tree $t$ and any $n \in \mathbb{N}$, we have $L_{i+1,j}^{(n)}(t) \subseteq L_{i,j}^{(n)}(t) \subseteq L_{i,j+1}^{(n)}(t)$, whereas $W_{i,j+1}^{(n)}(t) \subseteq W_{i,j}^{(n)}(t) \subseteq W_{i+1,j}^{(n)}(t)$. Likewise, we have $L_{i+1,j}(t) \subseteq L_{i,j}(t) \subseteq L_{i,j+1}(t)$ and $W_{i,j+1}(t) \subseteq W_{i,j}(t) \subseteq W_{i+1,j}(t)$.
\end{lemma}

A less obvious observation is that, in a game that does \emph{not} result in a draw, the player who is destined to win can guarantee to do so in a \emph{finite} number of rounds. Mathematically, this can be stated as follows:
\begin{lemma}\label{lem:compact}
For $i, j \in \{1,2,\ldots,\kappa-1\}$, let us set $W'_{i,j}(T) = W_{i,j}(T) \setminus \bigcup_{n=1}^{\infty}W_{i,j}^{(n)}(T)$ and $L'_{i,j}(T) = L_{i,j}(T) \setminus \bigcup_{n=1}^{\infty}L_{i,j}^{(n)}(T)$. Then $W'_{i,j}(T) = L'_{i,j}(T) = \emptyset$.
\end{lemma}
\begin{remark}\label{rem:1}
Before we outline the proof of Lemma~\ref{lem:compact}, we make the simple observation that $W'_{i,j}(T)$ and $L'_{i,j}(T)$ are trivially empty sets when either of $0$ or $\kappa$ belongs to the set $\{i,j\}$.
\end{remark}
\begin{proof}
The proof of Lemma~\ref{lem:compact} adopts an argument very similar to that of Proposition 7 of \cite{holroyd2021galton}, so we only briefly outline the crux of it here. A vertex $u_{0}$ of $T$ is in $W'_{i,j}(T)$ if and only if
\begin{enumerate}
\item $u_{0}$ has no child $u$ such that 
\begin{enumerate}
\item $u \in L_{j,i-1}^{(n)}(T)$ for some $n \in \mathbb{N}$, 
\item or $\omega_{t}(u_{0},u) \in \{0,1\}$ and $u \in L_{j,i}^{(n)}(T)$ for some $n \in \mathbb{N}$, 
\item or $\omega_{t}(u_{0},u) = 1$ and $u \in L_{j,i+1}^{(n)}(T)$ for some $n \in \mathbb{N}$,
\end{enumerate}
(since, if such a child $u$ exists, $P_{1}$, under our assumption of optimality (recall from \S\ref{subsec:deterministic_game}), moves the token from $u_{0}$ to $u$ in the first round, and wins the game $\mathcal{G}(i,j,\kappa,T)[u_{0}]$ in less than $n+1$ rounds, implying that $u_{0} \in W_{i,j}^{(n+1)}(T)$ and leading to a contradiction),
\item $u_{0}$ has at least one child $u_{1}$ such that 
\begin{enumerate}
\item either $u_{1} \in L'_{j,i-1}(T)$ (in which case $i > 1$ by Remark~\ref{rem:1}), 
\item or $\omega_{t}(u_{0},u_{1}) \in \{0,1\}$ and $u_{1} \in L'_{j,i}(T)$, 
\item or $\omega_{t}(u_{0},u_{1})=1$ and $u_{1} \in L'_{j,i+1}(T)$ (in which case, $i < \kappa-1$ by Remark~\ref{rem:1}).
\end{enumerate}
\end{enumerate}
To win the game, $P_{1}$ moves the token from $u_{0}$ to such a child $u_{1}$ in the first round. 

Next, $u_{1} \in L'_{j,i}(T)$ (an analogous argument works if $u_{1}$ belongs to $L'_{j,i-1}(T)$ or $L'_{j,i+1}(T)$) if and only if 
\begin{enumerate}
\item every child $v$ of $u_{1}$ is 
\begin{enumerate}
\item either in $W_{i,j+1}(T)$, 
\item or in $W_{i,j}(T)$ with $\omega_{t}(u_{1},v) \in \{-1,0\}$, 
\item or in $W_{i,j-1}(T)$ with $\omega_{t}(u_{1},v) = -1$,
\end{enumerate}
\item at least one child $u_{2}$ of $u_{1}$ is such that 
\begin{enumerate}
\item either $u_{2} \in W'_{i,j+1}(T)$ (so that $j < \kappa-1$ by Remark~\ref{rem:1}),  
\item or $u_{2} \in W'_{i,j}(T)$ with $\omega_{t}(u_{1},u_{2}) \in \{-1,0\}$,
\item or $u_{2} \in W'_{i,j-1}(T)$ (so that $j > 1$ by Remark~\ref{rem:1}) with $\omega_{t}(u_{1},u_{2}) = -1$.
\end{enumerate}
\end{enumerate}
$P_{2}$, under our assumption of optimality, moves the token from $u_{1}$ to such a child $u_{2}$ in the second round.

The argument above makes it evident that, the initial vertex $u_{0}$ is in $W'_{i,j}(T)$ only if in \emph{every} round of the game $\mathcal{G}(i,j,\kappa,T)[u_{0}]$, the player whose turn it is to make a move does so without having their entire capital squandered, without managing to accumulate a total capital of $\kappa$, and without getting stranded at a leaf vertex. Therefore, the game continues forever, leading to a draw and thereby implying that $u_{0}$ is actually in $D_{i,j}(T)$. This yields the desired contradiction, allowing us to conclude that $W'_{i,j}(T) = \emptyset$.
\end{proof}

The most important consequence of Lemma~\ref{lem:compact} is that, for each $i,j \in \{1,2,\ldots,\kappa-1\}$,
\begin{equation}
\bigcup_{n=1}^{\infty}W_{i,j}^{(n)}(T) = W_{i,j}(T) \text{ and } \bigcup_{n=1}^{\infty}L_{i,j}^{(n)}(T) = L_{i,j}(T).\label{eq:compact}
\end{equation}
Since the sequences $\left\{W_{i,j}^{(n)}(T)\right\}_{n \in \mathbb{N}_{0}}$ and $\left\{L_{i,j}^{(n)}(T)\right\}_{n \in \mathbb{N}_{0}}$ are, by definition, both increasing, hence \eqref{eq:compact} leads to
\begin{equation}
\lim_{n \rightarrow \infty}w_{i,j}^{(n)} = w_{i,j} \text{ and } \lim_{n \rightarrow \infty}\ell_{i,j}^{(n)} = \ell_{i,j}.\label{eq:limits}
\end{equation}

\subsection{The recurrence relations arising from our game}\label{subsec:recurrence}
The key to analyzing the game described in \S\ref{subsec:random_game} lies in deducing the recurrence relations that govern how the probabilities $\left\{w_{i,j}^{(n)}\right\}_{n \in \mathbb{N}_{0}}$ and $\left\{\ell_{i,j}^{(n)}\right\}_{n \in \mathbb{N}_{0}}$ relate with one another, leading to conclusions regarding the respective sequential limits $w_{i,j}$ and $\ell_{i,j}$ (by \eqref{eq:limits}), where $i,j \in \mathbb{N}$ with $\max\{i,j\} < \kappa$. 

For any $n \in \mathbb{N}_{0}$, a vertex $u$ of $T$ is in $L_{i,j}^{(n+1)}(T)$ if and only if either $u$ is childless, or \emph{every} child $v$ of $u$ satisfies one of the following conditions (we make use of Lemma~\ref{lem:containment} here):
\begin{enumerate}
\item $v \in W_{j,i+1}^{(n)}(T)$
\item or $v \in W_{j,i}^{(n)}(T) \setminus W_{j,i+1}^{(n)}(T)$ and $\omega_{T}(u,v) \in \{-1,0\}$,
\item or $v \in W_{j,i-1}^{(n)}(T) \setminus W_{j,i}^{(n)}(T)$ and $\omega_{T}(u,v) = -1$.
\end{enumerate}   
Conditioning on the event that $u$ has precisely $m$ children (recall, from \S\ref{subsec:random_game}, that the probability of this event is $\chi(m)$), we let $k$ denote the number of children that are in $W_{j,i+1}^{(n)}(T)$, $r$ denote the number of children $v$ that are in $W_{j,i}^{(n)}(T) \setminus W_{j,i+1}^{(n)}(T)$ and $\omega_{T}(u,v) \in \{-1,0\}$, and the remaining $m-k-r$ are children $v$ that are in $W_{j,i-1}^{(n)}(T)\setminus W_{j,i}^{(n)}(T)$ and $\omega_{T}(u,v) = -1$. We then have (noting that the case where $u$ is childless corresponds to $m=0$)
\begin{align}
\ell_{i,j}^{(n+1)} ={}& \sum_{m=0}^{\infty}\sum_{k=0}^{m}\sum_{r=0}^{m-k}{m \choose k}{m-k \choose r}\left(w_{j,i+1}^{(n)}\right)^{k}\left(w_{j,i}^{(n)}-w_{j,i+1}^{(n)}\right)^{r}(p_{-1}+p_{0})^{r}\left(w_{j,i-1}^{(n)}-w_{j,i}^{(n)}\right)^{m-k-r}p_{-1}^{m-k-r}\chi(m)\nonumber\\
={}& \sum_{m=0}^{\infty}\left\{w_{j,i+1}^{(n)} + \left(w_{j,i}^{(n)}-w_{j,i+1}^{(n)}\right)(p_{-1}+p_{0}) + \left(w_{j,i-1}^{(n)}-w_{j,i}^{(n)}\right)p_{-1}\right\}^{m}\chi(m)\nonumber\\
={}& G\left(p_{1}w_{j,i+1}^{(n)} + p_{0}w_{j,i}^{(n)} + p_{-1}w_{j,i-1}^{(n)}\right),\label{recurrence_1}
\end{align}
where, recall from \eqref{pgf}, that $G$ is the probability generating function corresponding to $\chi$.

Next, we note that a vertex $u$ of $T$ is in $W_{i,j}^{(n+1)}(T)$ if and only if $u$ has at least one child $v$ such that 
\begin{enumerate}
\item either $v \in L_{j,i-1}^{(n)}(T)$,
\item or $v \in L_{j,i}^{(n)}(T) \setminus L_{j,i-1}^{(n)}(T)$ and $\omega_{T}(u,v) \in \{0,1\}$,
\item or $v \in L_{j,i+1}^{(n)}(T) \setminus L_{j,i}^{(n)}(T)$ and $\omega_{T}(u,v) = 1$.
\end{enumerate}
We once again condition on the event that $u$ has precisely $m$ children, with $m$ varying over $\mathbb{N}$. In the second sum below, we further condition on the event that $u$ has $k$ children that are in $L_{j,i}^{(n)}(T) \setminus L_{j,i-1}^{(n)}(T)$, so that $1-p_{-1}^{k}$ denotes the probability that at least one of these children, say $v$, satisfies $\omega_{T}(u,v) \in \{0,1\}$. In the third sum below, we condition on the event that $u$ has $k$ children that are in $L_{j,i}^{(n)}(T) \setminus L_{j,i-1}^{(n)}(T)$, so that $p_{-1}^{k}$ denotes the probability that each such child $u'$ satisfies $\omega_{T}(u,u')=-1$, and that $u$ has $r$ children that are in $L_{j,i+1}^{(n)}(T) \setminus L_{j,i}^{(n)}(T)$, so that $1-(p_{-1}+p_{0})^{r}$ denotes the probability that at least one of these children, say $v$, satisfies $\omega_{T}(u,v)=1$.
\begin{align}
w_{i,j}^{(n+1)} ={}& \sum_{m=1}^{\infty}\Prob\left[\text{at least one child in } L_{j,i-1}^{(n)}(T)\right]\chi(m) \nonumber\\&+ \sum_{m=1}^{\infty}\Prob\left[\text{no child in } L_{j,i-1}^{(n)}(T), \text{ at least one child } v \text{ in } L_{j,i}^{(n)}(T) \setminus L_{j,i-1}^{(n)}(T) \text{ with } \omega_{T}(u,v)\in \{0,1\}\right]\chi(m)\nonumber\\& + \sum_{m=1}^{\infty}\Prob\Big[\text{no child in } L_{j,i-1}^{(n)}(T), \text{ for each child } u' \text{ in } L_{j,i}^{(n)}(T) \setminus L_{j,i-1}^{(n)}(T) \text{ we have } \omega_{T}(u,u') = -1,\nonumber\\& \text{ at least one child } v \text{ in } L_{j,i+1}^{(n)}(T) \setminus L_{j,i}^{(n)}(T) \text{ with } \omega_{T}(u,v) = 1\Big]\chi(m)\nonumber\\
={}& \sum_{m=1}^{\infty}\left\{1-\left(1-\ell_{j,i-1}^{(n)}\right)^{m}\right\}\chi(m) + \sum_{m=1}^{\infty}\sum_{k=1}^{m}{m \choose k}\left(\ell_{j,i}^{(n)}-\ell_{j,i-1}^{(n)}\right)^{k}\left\{1-p_{-1}^{k}\right\}\left(1-\ell_{j,i}^{(n)}\right)^{m-k}\chi(m) + \nonumber\\& \sum_{m=1}^{\infty}\sum_{k=0}^{m}\sum_{r=1}^{m-k}{m \choose k}{m-k \choose r}\left(\ell_{j,i}^{(n)}-\ell_{j,i-1}^{(n)}\right)^{k}p_{-1}^{k}\left(\ell_{j,i+1}^{(n)}-\ell_{j,i}^{(n)}\right)^{r}\left\{1-(p_{-1}+p_{0})^{r}\right\}\left(1-\ell_{j,i+1}^{(n)}\right)^{m-k-r}\chi(m)\nonumber\\
={}& 1 - G\left(1-\ell_{j,i-1}^{(n)}\right) + \sum_{m=0}^{\infty}\sum_{k=0}^{m}{m \choose k}\left(\ell_{j,i}^{(n)}-\ell_{j,i-1}^{(n)}\right)^{k}\left(1-\ell_{j,i}^{(n)}\right)^{m-k}\chi(m)\nonumber\\& - \sum_{m=0}^{\infty}\sum_{k=0}^{m}{m \choose k}\left(\ell_{j,i}^{(n)}-\ell_{j,i-1}^{(n)}\right)^{k}p_{-1}^{k}\left(1-\ell_{j,i}^{(n)}\right)^{m-k}\chi(m)\nonumber\\&+ \sum_{m=0}^{\infty}\sum_{k=0}^{m}\sum_{r=0}^{m-k}{m \choose k}{m-k \choose r}\left(\ell_{j,i}^{(n)}-\ell_{j,i-1}^{(n)}\right)^{k}p_{-1}^{k}\left(\ell_{j,i+1}^{(n)}-\ell_{j,i}^{(n)}\right)^{r}\left(1-\ell_{j,i+1}^{(n)}\right)^{m-k-r}\chi(m)\nonumber\\& - \sum_{m=0}^{\infty}\sum_{k=0}^{m}\sum_{r=0}^{m-k}{m \choose k}{m-k \choose r}\left(\ell_{j,i}^{(n)}-\ell_{j,i-1}^{(n)}\right)^{k}p_{-1}^{k}\left(\ell_{j,i+1}^{(n)}-\ell_{j,i}^{(n)}\right)^{r}(p_{-1}+p_{0})^{r}\left(1-\ell_{j,i+1}^{(n)}\right)^{m-k-r}\chi(m)\nonumber\\
={}& 1 - G\left(1-\ell_{j,i-1}^{(n)}\right) + G\left(1-\ell_{j,i-1}^{(n)}\right) - G\left(1-(p_{0}+p_{1})\ell_{j,i}^{(n)}-p_{-1}\ell_{j,i-1}^{(n)}\right)\nonumber\\&+ G\left(1-(p_{0}+p_{1})\ell_{j,i}^{(n)}-p_{-1}\ell_{j,i-1}^{(n)}\right) - G\left(1-p_{1}\ell_{j,i+1}^{(n)}-p_{0}\ell_{j,i}^{(n)}-p_{-1}\ell_{j,i-1}^{(n)}\right)\nonumber\\
={}& 1 - G\left(1-p_{1}\ell_{j,i+1}^{(n)}-p_{0}\ell_{j,i}^{(n)}-p_{-1}\ell_{j,i-1}^{(n)}\right).\label{recurrence_2}
\end{align}

\subsection{Proof of Theorem~\ref{thm:1}}\label{subsec:thm:1_proof}
Since the probability generating function $G$ is a continuous function, we take limits, as $n \rightarrow \infty$, of both sides of each of \eqref{recurrence_1} and \eqref{recurrence_2}, to deduce, via \eqref{eq:limits}, that for $i, j \in \{1,2,\ldots,\kappa-1\}$,
\begin{equation}\label{recurrence_3}
\ell_{i,j} = G\left(p_{1}w_{j,i+1} + p_{0}w_{j,i} + p_{-1}w_{j,i-1}\right)
\end{equation}
and
\begin{align}
w_{i,j} ={}& 1 - G\left(1-p_{1}\ell_{j,i+1}-p_{0}\ell_{j,i}-p_{-1}\ell_{j,i-1}\right)\label{recurrence_4'}\\
={}& 1-G\left(p_{1}+p_{0}+p_{-1}-p_{1}\ell_{j,i+1}-p_{0}\ell_{j,i}-p_{-1}\ell_{j,i-1}\right)\nonumber\\
={}& 1-G\left\{p_{1}\left(1-\ell_{j,i+1}\right)+p_{0}\left(1-\ell_{j,i}\right)+p_{-1}\left(1-\ell_{j,i-1}\right)\right\}.\label{recurrence_4}
\end{align}
We note two special cases of each of \eqref{recurrence_3} and \eqref{recurrence_4}. First, setting $i=1$ in \eqref{recurrence_3}, we obtain
\begin{align}
\ell_{1,j} = G\left(p_{1}w_{j,2} + p_{0}w_{j,1} + p_{-1}w_{j,0}\right) = G\left(p_{-1}+p_{1}w_{j,2} + p_{0}w_{j,1}\right),\label{recurrence_3_1}
\end{align}
since $w_{j,0}=1$ (which follows from $W_{j,0}(T) = V(T)$, as remarked right after defining the subsets in \S\ref{subsec:notations_definitions}); setting $i=\kappa-1$ in \eqref{recurrence_3}, we obtain
\begin{align}
\ell_{\kappa-1,j} = G\left(p_{1}w_{j,\kappa} + p_{0}w_{j,\kappa-1} + p_{-1}w_{j,\kappa-2}\right) = G\left(p_{0}w_{j,\kappa-1} + p_{-1}w_{j,\kappa-2}\right),\label{recurrence_3_2}
\end{align}
since $w_{j,\kappa}=0$ (which follows from $L_{j,\kappa}(T)=V(T)$, as remarked right after defining the subsets in \S\ref{subsec:notations_definitions}). Likewise, setting $i=1$ in \eqref{recurrence_4}, we obtain
\begin{align}
w_{1,j} ={}& 1-G\left\{p_{1}\left(1-\ell_{j,2}\right)+p_{0}\left(1-\ell_{j,1}\right)+p_{-1}\left(1-\ell_{j,0}\right)\right\}\nonumber\\
={}& 1-G\left\{p_{1}\left(1-\ell_{j,2}\right)+p_{0}\left(1-\ell_{j,1}\right)+p_{-1}\right),\label{recurrence_4_1}
\end{align}
since $\ell_{j,0}=0$ (which, again, follows from $W_{j,0}(T) = V(T)$), and setting $i=\kappa-1$ in \eqref{recurrence_4}, we obtain
\begin{align}
w_{\kappa-1,j} ={}& 1-G\left\{p_{1}\left(1-\ell_{j,\kappa}\right)+p_{0}\left(1-\ell_{j,\kappa-1}\right)+p_{-1}\left(1-\ell_{j,\kappa-2}\right)\right\} \nonumber\\
={}& 1-G\left\{p_{0}\left(1-\ell_{j,\kappa-1}\right)+p_{-1}\left(1-\ell_{j,\kappa-2}\right)\right\},\label{recurrence_4_2}
\end{align}
since $\ell_{j,\kappa}=1$ (which, again, follows from $L_{j,\kappa}(T)=V(T)$).

Defining the matrices $W$, $L$, $P$ and $J$, as well as the vectors $\mathbf{1}$ and $\mathbf{e}_{1}$, as in the statement of Theorem~\ref{thm:1}, we see that \eqref{recurrence_3}, \eqref{recurrence_3_1} and \eqref{recurrence_3_2} can be combined via the following matrix equation:
\begin{equation}
L = f\left(p_{-1}\mathbf{e}_{1}\mathbf{1}^{T} + PW^{T}\right).\label{recurrence_matrix_1}
\end{equation}
Next, we see that \eqref{recurrence_4}, \eqref{recurrence_4_1} and \eqref{recurrence_4_2} can be combined via the following matrix equation:
\begin{equation}\label{recurrence_matrix_2}
J - W^{T} = f\left(p_{-1}\mathbf{1}\mathbf{e}_{1}^{T}+\left(\mathbf{1}\mathbf{1}^{T}-L\right)P^{T}\right).
\end{equation}
Combining \eqref{recurrence_matrix_1} and \eqref{recurrence_matrix_2}, we obtain
\begin{equation}
L = f\left[p_{-1}\mathbf{e}_{1}\mathbf{1}^{T} + P\left\{J - f\left(p_{-1}\mathbf{1}\mathbf{e}_{1}^{T}+\left(J-L\right)P^{T}\right)\right\}\right],\label{recurrence_matrix}
\end{equation}
which proves the first assertion of Theorem~\ref{thm:1}, i.e.\ that $L$ is a fixed point of the function $h$ as defined in \eqref{fixed_point_eq}.

We now prove that $J - W$ is another fixed point of $h$, as follows. Recall that the function $f$, defined in the statement of Theorem~\ref{thm:1}, involves applying the pgf $G$ to \emph{each coordinate} of the matrix on which $f$ is being applied, so that if $f(A) = B$ for any two $(\kappa-1)\times(\kappa-1)$ matrices $A$ and $B$, we also have $f(A^{T}) = B^{T}$. Using this and the fact that $(PA)^{T} = A^{T}P^{T}$ for any $(\kappa-1)\times(\kappa-1)$ matrix $A$, we rewrite \eqref{recurrence_matrix_2} as 
\begin{align}
J-W=\left(J - W^{T}\right)^{T} = f\left(\left[p_{-1}\mathbf{1}\mathbf{e}_{1}^{T}+\left(J-L\right)P^{T}\right]^{T}\right) = f\left(p_{-1}\mathbf{e}_{1}\mathbf{1}^{T}+P\left(J-L\right)^{T}\right),\label{matrix_recurrence_3}
\end{align} 
and we rewrite \eqref{recurrence_matrix_1} as 
\begin{align}
{}&L = f\left(p_{-1}\mathbf{e}_{1}\mathbf{1}^{T} + PW^{T}\right) \Longleftrightarrow \left(J-L\right)^{T} = \left\{J-f\left(p_{-1}\mathbf{e}_{1}\mathbf{1}^{T} + PW^{T}\right)\right\}^{T}\nonumber\\
\Longleftrightarrow{}& \left(J-L\right)^{T} = J-f\left\{\left(p_{-1}\mathbf{e}_{1}\mathbf{1}^{T} + PW^{T}\right)^{T}\right\} = J-f\left\{p_{-1}\mathbf{1}\mathbf{e}_{1}^{T}+WP^{T}\right\}\nonumber\\
\Longleftrightarrow{}& \left(J-L\right)^{T} = J-f\left[p_{-1}\mathbf{1}\mathbf{e}_{1}^{T}+\left\{J-\left(J-W\right)\right\}P^{T}\right].\label{matrix_recurrence_4}
\end{align}
Combining \eqref{matrix_recurrence_3} and \eqref{matrix_recurrence_4}, we obtain
\begin{align}
J-W = f\left(p_{-1}\mathbf{e}_{1}\mathbf{1}^{T}+P\left\{J-f\left[p_{-1}\mathbf{1}\mathbf{e}_{1}^{T}+\left\{J-\left(J-W\right)\right\}P^{T}\right]\right\}\right),\nonumber
\end{align}
thus proving that $J-W$ is a fixed point of the function $h$, as claimed in Theorem~\ref{thm:1}.

We now prove \eqref{smallest_largest_fixed_point}, by first stating and proving Lemma~\ref{lem:monotone_h}. To this end, recall the relation $A \preceq B$ between two matrices of the same order (defined right before the statement of Theorem~\ref{thm:1}). 
\begin{lemma}\label{lem:monotone_h}
Let $X_{1}$ and $X_{2}$ be two matrices belonging to the space $\mathcal{S}$ defined in Theorem~\ref{thm:1}. If $X_{2} \preceq X_{1}$, then $h(X_{2}) \preceq h(X_{1})$ as well, where $h$ is as defined in \eqref{fixed_point_eq}.
\end{lemma}
\begin{proof}
Let the $(i,j)$-th entries of $X_{1}$ and $X_{2}$ be $x^{(1)}_{i,j}$ and $x^{(2)}_{i,j}$ respectively, for each $i,j \in \{1,2,\ldots,\kappa-1\}$. For $i,j \in \{2,\ldots,\kappa-2\}$, the $(i,j)$-th element of $h(X_{1})$ is given by
\begin{align}
h(X_{1})_{i,j} ={}& G\Big[1-p_{-1}G\Big\{1-p_{-1}x^{(1)}_{i-1,j-1}-p_{0}x^{(1)}_{i-1,j}-p_{1}x^{(1)}_{i-1,j+1}\Big\}-p_{0}G\Big\{1-p_{-1}x^{(1)}_{i,j-1}-p_{0}x^{(1)}_{i,j}-p_{1}x^{(1)}_{i,j+1}\Big\}\nonumber\\&-p_{1}G\Big\{1-p_{-1}x^{(1)}_{i+1,j-1}-p_{0}x^{(1)}_{i+1,j}-p_{1}x^{(1)}_{i+1,j+1}\Big\}\Big],\nonumber
\end{align}
while the element in the $i$-th row and $j$-th column of $h(X_{2})$ is
\begin{align}
h(X_{2})_{i,j} ={}& G\Big[1-p_{-1}G\Big\{1-p_{-1}x^{(2)}_{i-1,j-1}-p_{0}x^{(2)}_{i-1,j}-p_{1}x^{(2)}_{i-1,j+1}\Big\}-p_{0}G\Big\{1-p_{-1}x^{(2)}_{i,j-1}-p_{0}x^{(2)}_{i,j}-p_{1}x^{(2)}_{i,j+1}\Big\}\nonumber\\&-p_{1}G\Big\{1-p_{-1}x^{(2)}_{i+1,j-1}-p_{0}x^{(2)}_{i+1,j}-p_{1}x^{(2)}_{i+1,j+1}\Big\}\Big].\nonumber
\end{align}
Since $x^{(2)}_{i,j} \leqslant x^{(1)}_{i,j}$ for each $i,j \in \{1,2,\ldots,\kappa-1\}$, and since $G$ is monotonically increasing, we have 
\begin{align}
{}&-p_{-1}G\Big\{1-p_{-1}x^{(1)}_{i-1,j-1}-p_{0}x^{(1)}_{i-1,j}-p_{1}x^{(1)}_{i-1,j+1}\Big\} \geqslant -p_{-1}G\Big\{1-p_{-1}x^{(2)}_{i-1,j-1}-p_{0}x^{(2)}_{i-1,j}-p_{1}x^{(2)}_{i-1,j+1}\Big\},\label{ineq_1}\\
{}&-p_{0}G\Big\{1-p_{-1}x^{(1)}_{i,j-1}-p_{0}x^{(1)}_{i,j}-p_{1}x^{(1)}_{i,j+1}\Big\} \geqslant -p_{0}G\Big\{1-p_{-1}x^{(2)}_{i,j-1}-p_{0}x^{(2)}_{i,j}-p_{1}x^{(2)}_{i,j+1}\Big\},\label{ineq_2}\\
{}&-p_{1}G\Big\{1-p_{-1}x^{(1)}_{i+1,j-1}-p_{0}x^{(1)}_{i+1,j}-p_{1}x^{(1)}_{i+1,j+1}\Big\} \geqslant -p_{1}G\Big\{1-p_{-1}x^{(2)}_{i+1,j-1}-p_{0}x^{(2)}_{i+1,j}-p_{1}x^{(2)}_{i+1,j+1}\Big\}.\label{ineq_3}
\end{align}
Note, further, that if any of the inequalities
\begin{equation}
x^{(1)}_{i-1,j-1} \geqslant x^{(2)}_{i-1,j-1}, \ x^{(1)}_{i-1,j} \geqslant x^{(2)}_{i-1,j}, \text{ and } x^{(1)}_{i-1,j+1} \geqslant x^{(2)}_{i-1,j+1}\nonumber
\end{equation}
is strict, then so will be the inequality in \eqref{ineq_1} as long as $p_{-1} > 0$, since $G$ is strictly increasing (via the assumption made right after defining $G$ in \eqref{pgf} in \S\ref{subsec:random_game}). Similar conclusions are true for the inequalities in \eqref{ineq_2} and \eqref{ineq_3}. Combining the inequalities in \eqref{ineq_1}, \eqref{ineq_2} and \eqref{ineq_3} and once again applying the monotonically increasing nature of $G$, we conclude that $h(X_{1})_{i,j} \geqslant h(X_{2})_{i,j}$ for all $i,j \in \{2,\ldots,\kappa-2\}$. Very similar arguments lead to the inequalities $h(X_{1})_{i,j} \geqslant h(X_{2})_{i,j}$ for all $i,j \in \{1,2,\ldots,\kappa-1\}$ such that $\{i,j\} \cap \{1,\kappa-1\} \neq \emptyset$. This concludes the proof of Lemma~\ref{lem:monotone_h}.
\end{proof} 

Recall from \S\ref{subsec:notations_definitions} that for $i,j \in \{1,2,\ldots,\kappa-1\}$, we set $L_{i,j}^{(0)}(T) = \emptyset$. Consequently, $\ell_{i,j}^{(0)}=0$ for $i,j \in \{1,2,\ldots,\kappa-1\}$. Let $L^{(n)}$, for each $n \in \mathbb{N}_{0}$, denote the $(\kappa-1)\times(\kappa-1)$ matrix in which $L^{(n)}_{i,j}=\ell_{i,j}^{(n)}$, for all $i,j \in \{1,2,\ldots,\kappa-1\}$. Recall, from above, how we showed, using \eqref{recurrence_3} and \eqref{recurrence_4} (and their special versions \eqref{recurrence_3_1}, \eqref{recurrence_3_2}, \eqref{recurrence_4_1} and \eqref{recurrence_4_2}), that \eqref{recurrence_matrix} holds. Adopting this same approach, but now, utilizing \eqref{recurrence_1} and \eqref{recurrence_2} instead of \eqref{recurrence_3} and \eqref{recurrence_4}, we obtain, for each $n \in \mathbb{N}$,
\begin{equation}
L^{(n+1)} = f\left[p_{-1}\mathbf{e}_{1}\mathbf{1}^{T} + P\left\{J - f\left(p_{-1}\mathbf{1}\mathbf{e}_{1}^{T}+\left(J-L^{(n-1)}\right)P^{T}\right)\right\}\right] = h\left(L^{(n-1)}\right).\label{recurrence_matrix_n}
\end{equation}
Since $\ell_{i,j}^{(0)}=0$ for all $i,j \in \{1,2,\ldots,\kappa-1\}$, each element of $L^{(0)}$ is $0$. Let $\hat{X}$ denote a fixed point of the function $h$, with $\hat{X} \in \mathcal{S}$ (so that each entry of $\hat{X}$ is non-negative). Then $L^{(0)} \preceq \hat{X}$. Applying Lemma~\ref{lem:monotone_h} repeatedly, and using \eqref{recurrence_matrix_n}, we obtain:
\begin{align}
{}&L^{(2)} = h\left(L^{(0)}\right) \preceq h\left(\hat{X}\right) = \hat{X}\nonumber\\
{}&\vdots\nonumber\\
{}&L^{(2n)} = h\left(L^{(2n-2)}\right) = \cdots = h^{(n)}\left(L^{(0)}\right) \preceq h^{(n)}\left(\hat{X}\right) = \hat{X},\label{L^{(2n)}_hat{X}_ineq}
\end{align}
where we recall that $h^{(n)}$ indicates the $n$-fold composition of $h$ with itself. The inequality in \eqref{L^{(2n)}_hat{X}_ineq} holds for every $n \in \mathbb{N}_{0}$. As we take the limit of both sides of \eqref{L^{(2n)}_hat{X}_ineq} as $n \rightarrow \infty$ (which means that we consider the limit of each entry of the matrix $L^{(2n)}$), we see, via \eqref{eq:limits}, that
\begin{equation}
L = \lim_{n \rightarrow \infty}L^{(2n)} \preceq \hat{X}.\nonumber
\end{equation}
This completes the proof of the first inequality in \eqref{smallest_largest_fixed_point}. 

The second inequality in \eqref{smallest_largest_fixed_point} follows via a similar argument, but this time, we define, for each $n \in \mathbb{N}_{0}$, $W^{(n)}$ to be the $(\kappa-1)\times(\kappa-1)$ matrix in which $W^{(n)}_{i,j}=w_{i,j}^{(n)}$ for all $i,j \in \{1,2,\ldots,\kappa-1\}$, and we replace \eqref{recurrence_matrix_n} by 
\begin{equation}
J - W^{(n+1)} = h\left(J-W^{(n-1)}\right).\label{recurrence_matrix_n_W}
\end{equation}
Moreover, since we set $W_{i,j}^{(0)}(T)=\emptyset$ for each $i,j \in \{1,2,\ldots,\kappa-1\}$, we have $w_{i,j}^{(0)}=0$ for each $i,j \in \{1,2,\ldots,\kappa-1\}$. Once again, let $\hat{X}$ denote a fixed point of the function $h$, with $\hat{X} \in \mathcal{S}$ (so that each entry of $\hat{X}$ is bounded above by $1$). Consequently, we have $\hat{X} \preceq J-W^{(0)}$. Applying Lemma~\ref{lem:monotone_h} and \eqref{recurrence_matrix_n_W} repeatedly, we obtain, analogous to \eqref{L^{(2n)}_hat{X}_ineq}, the inequality
\begin{equation}
\hat{X} \preceq h^{(n)}\left(J-W^{(0)}\right) = J-W^{(2n)},\label{W^{(2n)}_hat{X}_ineq}
\end{equation}
for each $n \in \mathbb{N}_{0}$. Taking the limit as $n \rightarrow \infty$ on both sides of \eqref{W^{(2n)}_hat{X}_ineq}, we obtain $\hat{X} \preceq J-W$, completing the proof of the second inequality in \eqref{smallest_largest_fixed_point}.

We now come to the fourth assertion stated in Theorem~\ref{thm:1}, and the first pertaining to the probabilities of draw. Suppose $h$ has at least two distinct fixed points, $\hat{X}$ and $\hat{Y}$, in $\mathcal{S}$, with $\hat{X}_{i,j}=\hat{x}_{i,j}$ and $\hat{Y}_{i,j}=\hat{y}_{i,j}$ for $i,j \in \{1,2,\ldots,\kappa-1\}$. There must exist at least one pair $(i_{0},j_{0})$ of coordinates $i_{0}, j_{0} \in \{1,2,\ldots,\kappa-1\}$ such that $\hat{x}_{i_{0},j_{0}} < \hat{y}_{i_{0},j_{0}}$. By \eqref{smallest_largest_fixed_point}, we must have
\begin{equation}
\ell_{i_{0},j_{0}} \leqslant \hat{x}_{i_{0},j_{0}} < \hat{y}_{i_{0},j_{0}} \leqslant 1-w_{i_{0},j_{0}},\nonumber
\end{equation}
so that $d_{i_{0},j_{0}}>0$ in this case. On the other hand, if $h$ has a unique fixed point in $\mathcal{S}$, then this fixed point must equal \emph{both} $L$ and $J-W$, thereby implying that $\ell_{i,j}=1-w_{i,j} \implies d_{i,j}=0$ for each $i,j \in \{1,2,\ldots,\kappa-1\}$. This completes the proof of the assertion in Theorem~\ref{thm:1} that $d_{i,j}=0$ for each $i,j \in \{1,2,\ldots,\kappa-1\}$ if and only if $h$ has a unique fixed point in $\mathcal{S}$.

Defining the function $g$ as in \eqref{g_composition}, we see that $h = g^{(2)}$, i.e.\ $h$ is obtained by taking the composition of $g$ with itself. Therefore, any fixed point of $g$ is also a fixed point of $h$. Since $g$ maps $\mathcal{S}$ to itself, and since $\mathcal{S}$ is a compact, convex set, hence, by Brouwer's fixed point theorem, $g$ has at least one fixed point in $\mathcal{S}$ (see \cite{brouwer1911abbildung}). When $d_{i,j}=0$ for each $i,j \in \{1,2,\ldots,\kappa-1\}$, since $h$ has a unique fixed point in $\mathcal{S}$, hence so does $g$. Moreover, this fixed point then equals $L$ (which, in turn, is equal to $J-W$). 

This brings us to the end of the proof of Theorem~\ref{thm:1}. \qed

\subsection{Proof of Theorem~\ref{thm:2}}\label{sec:proof_thm_2}
We begin with the proof of part \eqref{draw_1} of Theorem~\ref{thm:2}. Fix any $i,j \in \{2,\ldots,\kappa-1\}$. Recall from \S\ref{subsec:notations_definitions} that $d_{i,j}$ indicates the probability of the event that the root $\phi$ of our GW tree $T$ belongs to the set $D_{i,j}(T)$, i.e.\ that the game $\mathcal{G}(i,j,\kappa,T)[\phi]$ culminates in a draw. Recall also the probabilities $\ell_{i,j}$ and $w_{i,j}$ of the events that $\phi$ belongs to $L_{i,j}(T)$ and $W_{i,j}(T)$ respectively. Note that $d_{i,j}$ equals $0$ if and only if $\ell_{i,j}+w_{i,j}=1$, which is equivalent to
\begin{align}
{}&G\left(p_{1}w_{j,i+1} + p_{0}w_{j,i} + p_{-1}w_{j,i-1}\right)+1-G\left\{p_{1}\left(1-\ell_{j,i+1}\right)+p_{0}\left(1-\ell_{j,i}\right)+p_{-1}\left(1-\ell_{j,i-1}\right)\right\}=1\nonumber\\
\Longleftrightarrow{}& G\left(p_{1}w_{j,i+1} + p_{0}w_{j,i} + p_{-1}w_{j,i-1}\right)=G\left\{p_{1}\left(1-\ell_{j,i+1}\right)+p_{0}\left(1-\ell_{j,i}\right)+p_{-1}\left(1-\ell_{j,i-1}\right)\right\}\nonumber\\
\Longleftrightarrow{}& p_{1}w_{j,i+1} + p_{0}w_{j,i} + p_{-1}w_{j,i-1}=p_{1}\left(1-\ell_{j,i+1}\right)+p_{0}\left(1-\ell_{j,i}\right)+p_{-1}\left(1-\ell_{j,i-1}\right)\nonumber\\
{}& p_{1}\left(1-\ell_{j,i+1}-w_{j,i+1}\right)+p_{0}\left(1-\ell_{j,i}-w_{j,i}\right)+p_{-1}\left(1-\ell_{j,i-1}-w_{j,i-1}\right)=0,\label{draw_0_equiv_cond_1}
\end{align}
where we make use of \eqref{recurrence_3} and \eqref{recurrence_4}, and the assumption, stated in Theorem~\ref{thm:2}, that $G$ is strictly increasing on $[0,1]$ (and hence bijective). Note that, under the assumption that each of $p_{-1}$, $p_{0}$ and $p_{1}$ is strictly positive, \eqref{draw_0_equiv_cond_1} holds if and only if 
\begin{align}
{}&\ell_{j,i+1}+w_{j,i+1}=\ell_{j,i}+w_{j,i}=\ell_{j,i-1}+w_{j,i-1}=1 \Longleftrightarrow d_{j,i+1}=d_{j,i}=d_{j,i-1}=0.\label{draw_0_equiv_cond_2}
\end{align}
Utilizing \eqref{draw_0_equiv_cond_2} for various values of $i$ and $j$, we obtain
\begin{align}
{}&d_{j,i-1}=0 \Longleftrightarrow d_{i-1,j-1}=d_{i-1,j}=d_{i-1,j+1}=0,\label{draw_0_equiv_cond_3}\\
{}&d_{j,i}=0 \Longleftrightarrow d_{i,j-1}=d_{i,j}=d_{i,j+1}=0,\label{draw_0_equiv_cond_4}\\
{}&d_{i-1,j}=0 \Longleftrightarrow d_{j,i-2}=d_{j,i-1}=d_{j,i}=0,\label{draw_0_equiv_cond_5}\\
{}&d_{i,j-1}=0 \Longleftrightarrow d_{j-1,i-1}=d_{j-1,i}=d_{j-1,i+1}=0,\label{draw_0_equiv_cond_10}\\
{}&d_{j-1,i}=0 \Longleftrightarrow d_{i,j-2}=d_{i,j-1}=d_{i,j}=0.\label{draw_0_equiv_cond_11}
\end{align}
We thus deduce the following sequence of implications, from \eqref{draw_0_equiv_cond_2} and \eqref{draw_0_equiv_cond_3}:
\begin{align}
d_{i,j}=0 \implies d_{j,i-1}=0 \implies d_{i-1,j}=0\label{draw_0_equiv_cond_6}
\end{align}
and using  \eqref{draw_0_equiv_cond_4} and \eqref{draw_0_equiv_cond_5}, we deduce:
\begin{align}
d_{i-1,j}=0 \implies d_{j,i}=0 \implies d_{i,j}=0,\label{draw_0_equiv_cond_7}
\end{align}
so that combining \eqref{draw_0_equiv_cond_6} and \eqref{draw_0_equiv_cond_7}, we conclude that
\begin{equation}
d_{i,j}=0 \Longleftrightarrow d_{i-1,j}=0.\label{draw_0_equiv_cond_8}
\end{equation}
On the other hand, from \eqref{draw_0_equiv_cond_2} and \eqref{draw_0_equiv_cond_4}, we deduce that
\begin{equation}
d_{i,j}=0 \implies d_{j,i}=0 \implies d_{i,j-1}=0,\label{draw_0_equiv_cond_9}
\end{equation}
whereas using \eqref{draw_0_equiv_cond_10} and \eqref{draw_0_equiv_cond_11}, we obtain
\begin{equation}
d_{i,j-1}=0 \implies d_{j-1,i}=0 \implies d_{i,j}=0,\label{draw_0_equiv_cond_12}
\end{equation}
so that combining \eqref{draw_0_equiv_cond_9} and \eqref{draw_0_equiv_cond_12}, we conclude that
\begin{equation}
d_{i,j}=0 \Longleftrightarrow d_{i,j-1}=0.\label{draw_0_equiv_cond_13}
\end{equation}

Next, let us consider $i,j \in \{1,2,\ldots,\kappa-1\}$ such that precisely one of $i$ and $j$ equals $1$. Without loss of generality, let $i=1$ and $j \in \{2,\ldots,\kappa-1\}$. In this case, $d_{1,j}=0$ if and only if $w_{1,j}+\ell_{1,j}=1$, and from \eqref{recurrence_3_1} and \eqref{recurrence_4_1}, this becomes equivalent to
\begin{align}
{}& G\left\{p_{-1}+p_{0}w_{j,1}+p_{1}w_{j,2}\right\}+1-G\left\{p_{-1}+p_{0}\left(1-\ell_{j,1}\right)+p_{1}\left(1-\ell_{j,2}\right)\right\}=1\nonumber\\
\Longleftrightarrow{}& G\left\{p_{-1}+p_{0}w_{j,1}+p_{1}w_{j,2}\right\}=G\left\{p_{-1}+p_{0}\left(1-\ell_{j,1}\right)+p_{1}\left(1-\ell_{j,2}\right)\right\}\nonumber\\
\Longleftrightarrow{}& p_{-1}+p_{0}w_{j,1}+p_{1}w_{j,2}=p_{-1}+p_{0}\left(1-\ell_{j,1}\right)+p_{1}\left(1-\ell_{j,2}\right)\nonumber\\
\Longleftrightarrow{}& p_{0}\left(1-\ell_{j,1}-w_{j,1}\right)+p_{1}\left(1-\ell_{j,2}-w_{j,2}\right)=0,\label{draw_0_equiv_cond_14}
\end{align}
where, once again, we make use of the strictly increasing (and hence bijective) nature of $G$. When $p_{0}$ and $p_{1}$ are both strictly positive, the identity in \eqref{draw_0_equiv_cond_14} holds if and only if 
\begin{equation}
\ell_{j,1}+w_{j,1}=\ell_{j,2}+w_{j,2}=1 \Longleftrightarrow d_{j,1}=d_{j,2}=0.\label{draw_0_equiv_cond_15}
\end{equation}
On the other hand, from \eqref{draw_0_equiv_cond_2}, we deduce that
\begin{equation}
d_{j,1}=0 \Longleftrightarrow d_{1,j-1}=d_{1,j}=d_{1,j+1}=0.\label{draw_0_equiv_cond_16}
\end{equation}
From \eqref{draw_0_equiv_cond_15} and \eqref{draw_0_equiv_cond_16}, we obtain
\begin{align}
d_{1,j}=0 \implies d_{j,1}=0 \implies d_{1,j-1}=0.\label{draw_0_equiv_cond_17}
\end{align}
Replacing $j$ by $j-1$ in \eqref{draw_0_equiv_cond_15}, we obtain
\begin{equation}
d_{1,j-1}=0 \Longleftrightarrow d_{j-1,1}=d_{j-1,2}=0, \label{draw_0_equiv_cond_18}
\end{equation}
and from \eqref{draw_0_equiv_cond_2}, we deduce that
\begin{equation}
d_{j-1,1}=0 \implies d_{1,j-1}=d_{1,j}=0.\label{draw_0_equiv_cond_19}
\end{equation}
Combining \eqref{draw_0_equiv_cond_18} and \eqref{draw_0_equiv_cond_19}, we obtain
\begin{equation}
d_{1,j-1}=0 \implies d_{j-1,1}=0 \implies d_{1,j}=0.\label{draw_0_equiv_cond_20}
\end{equation}
From \eqref{draw_0_equiv_cond_17} and \eqref{draw_0_equiv_cond_20}, we conclude that
\begin{equation}
d_{1,j}=0 \Longleftrightarrow d_{1,j-1}=0.\label{draw_0_equiv_cond_21}
\end{equation}
In a very similar manner, when $i \in \{2,\ldots,\kappa-1\}$ and $j=1$, we can show that $d_{i,1}=0$ if and only if $d_{i-1,1}=0$. Therefore, $d_{i,j}=0$ for \emph{any} pair $(i,j) \in \{1,2,\ldots,\kappa-1\}^{2}$ if and only if $d_{i',j'}=0$ for \emph{every} pair $(i',j') \in \{1,2,\ldots,\kappa-1\}^{2}$, thus concluding the proof of \eqref{draw_1} of Theorem~\ref{thm:2}.

We now come to the proof of part \eqref{draw_2} of Theorem~\ref{thm:2}. Making use of \eqref{recurrence_3} and \eqref{recurrence_4}, but this time, with $p_{-1}=0$, we see that for each $i, j \in \{1,2,\ldots,\kappa-2\}$, when $G$ is strictly increasing,
\begin{align}
{}& d_{i,j}=0 \Longleftrightarrow \ell_{i,j}+w_{i,j}=1 \Longleftrightarrow G\left(p_{1}w_{j,i+1}+p_{0}w_{j,i}\right)=G\left(1-p_{1}\ell_{j,i+1}-p_{0}\ell_{j,i}\right)\nonumber\\
{}&\Longleftrightarrow p_{1}\left(w_{j,i+1}+\ell_{j,i+1}\right)+p_{0}\left(w_{j,i}+\ell_{j,i}\right)=1 \Longleftrightarrow d_{j,i+1}=d_{j,i}=0.\label{p_{-1}_draw_eq_1}
\end{align}
Interchanging the roles of $i$ and $j$ in \eqref{p_{-1}_draw_eq_1}, we obtain
\begin{equation}
d_{j,i}=0 \Longleftrightarrow d_{i,j+1}=d_{i,j}=0.\label{p_{-1}_draw_eq_2}
\end{equation}
Combining \eqref{p_{-1}_draw_eq_1} and \eqref{p_{-1}_draw_eq_2}, we conclude that $d_{i,j}=0$ if and only if $d_{j,i}=0$, and $d_{i,j}=0 \implies d_{j,i}=0 \implies d_{i,j+1}=0$. Moreover, $d_{j,i+1}=0$ if and only if $d_{i+1,j}=0$, which tells us, via \eqref{p_{-1}_draw_eq_1}, that $d_{i,j}=0$ implies $d_{i+1,j}=0$. 

If $i=\kappa-1$ while $j \in \{1,2,\ldots,\kappa-2\}$, we similarly argue that $d_{\kappa-1,j}=0$ if and only if $d_{j,\kappa-1}=0$, and $d_{\kappa-1,j}=0$ implies that $d_{\kappa-1,j+1}=0$, Likewise, if $j=\kappa-1$ while $i \in \{1,2,\ldots,\kappa-2\}$, we deduce that if $d_{i,\kappa-1}=0$, then $d_{i+1,\kappa-1}=0$. This concludes the proof of \eqref{draw_2} of Theorem~\ref{thm:2}.

The proof of part \eqref{draw_3} of Theorem~\ref{thm:2} is accomplished exactly the same way as that of \eqref{draw_2}.

To prove part \eqref{draw_4} of Theorem~\ref{thm:2}, we, once again, make use of \eqref{recurrence_3} and \eqref{recurrence_4}, but with $p_{0}=0$. Consequently, we obtain:
\begin{align}
{}&d_{i,j}=0 \Longleftrightarrow \ell_{i,j}+w_{i,j}=0 \Longleftrightarrow G\left(p_{1}w_{j,i+1}+p_{-1}w_{j,i-1}\right)=G\left(1-p_{1}\ell_{j,i+1}-p_{-1}\ell_{j,i-1}\right)\nonumber\\
\Longleftrightarrow{}& p_{1}\left(w_{j,i+1}+\ell_{j,i+1}\right)+p_{-1}\left(w_{j,i-1}+\ell_{j,i-1}\right)=1 \Longleftrightarrow d_{j,i+1}=d_{j,i-1}=0.\label{p_{0}_draw_eq_1}
\end{align} 
Via a similar argument, we see that 
\begin{equation}
d_{j,i+1}=0 \Longleftrightarrow d_{i+1,j-1}=d_{i+1,j+1}=0 \quad \text{and} \quad d_{j,i-1}=0 \Longleftrightarrow d_{i-1,j-1}=d_{i-1,j+1}=0.\label{p_{0}_draw_eq_2}
\end{equation}
Combining \eqref{p_{0}_draw_eq_1} and \eqref{p_{0}_draw_eq_2}, we find that $d_{i,j}=0$ implies that $d_{i-1,j-1}=d_{i-1,j+1}=d_{i+1,j-1}=d_{i+1,j+1}=0$. On the other hand, we deduce from \eqref{p_{0}_draw_eq_1} that
\begin{equation}
d_{i-1,j-1}=0 \implies d_{j-1,i}=0 \implies d_{i,j}=0,\nonumber
\end{equation}
which, combined with our previous finding, shows that $d_{i,j}=0$ if and only if $d_{i-1,j-1}=0$. Likewise, we can show that 
\begin{enumerate*}
\item $d_{i,j}=0$ if and only if $d_{i-1,j+1}=0$, 
\item $d_{i,j}=0$ if and only if $d_{i+1,j-1}=0$,
\item and $d_{i,j}=0$ if and only if $d_{i+1,j+1}=0$.
\end{enumerate*}
Having established these, let us consider pairs $(i,j)$ and $(i+r,j+s)$, both in $\{1,2,\ldots,\kappa-1\}^{2}$, such that $r$ and $s$ have the same parity. Without loss of generality, let us assume that $r$ and $s$ are both positive (the argument would be nearly identical if one or both of them is / are non-positive), and let us assume that $r < s$, with $(s-r)=2t$ for some $t \geqslant 0$. Using our deductions above, we know that
\begin{align}
d_{i,j}=0 \Longleftrightarrow d_{i+1,j+1}=0 \Longleftrightarrow \cdots \Longleftrightarrow d_{i+r,j+r}=0,\nonumber
\end{align}
following which we have
\begin{multline}
d_{i+r,j+r}=0 \Longleftrightarrow d_{i+r-1,j+r+1}=0 \Longleftrightarrow d_{i+r,j+r+2}=0 \Longleftrightarrow \cdots \Longleftrightarrow d_{i+r-1,j+r+2t-1}=0 \\ \Longleftrightarrow d_{i+r,j+r+2t}=d_{i+r,j+s}=0.\nonumber
\end{multline}
This completes the proof of \eqref{draw_4} of Theorem~\ref{thm:2}.

\subsection{Proof of Theorem~\ref{thm:3}}\label{sec:proof_thm_3}
We begin by fixing an $\epsilon > 0$ whose value we determine later. Recall that, for any $i,j \in \{1,2,\ldots,\kappa-1\}$, the random variable $\mathcal{T}_{i,j}$ indicates the (random) duration of the game $\mathcal{G}(i,j,\kappa,T)[\phi]$, where $\phi$ is the root of our GW tree $T$. Also recall, from \S\ref{subsec:notations_definitions}, that $D_{i,j}^{(n)}(T)$ indicates the (random) subset of $V(T)$ comprising all those vertices $v$ such that $\mathcal{G}(i,j,\kappa,T)[v]$ lasts for at least $n$ rounds, and $d_{i,j}^{(n)}$ is the probability of the event that the root $\phi$ belongs to $D_{i,j}^{(n)}(T)$. It is evident that under the assumption that $d_{i,j}=0$, $\mathcal{T}_{i,j}$ is finite almost surely, and takes only non-negative integer values. We thus have
\begin{align}
\E[\mathcal{T}_{i,j}] ={}& \sum_{n=1}^{\infty}\Prob[\mathcal{T}_{i,j} \geqslant n] = \sum_{n=1}^{\infty}\Prob\left[\phi \in D_{i,j}^{(n)}(T)\right] = \sum_{n=1}^{\infty}d_{i,j}^{(n)} = \sum_{n=1}^{\infty}\left\{1-w_{i,j}^{(n)}-\ell_{i,j}^{(n)}\right\}\nonumber\\
={}& \sum_{n=1}^{\infty}\left\{w_{i,j}+\ell_{i,j}-w_{i,j}^{(n)}-\ell_{i,j}^{(n)}\right\}, \text{ since } d_{i,j}=0 \Longleftrightarrow w_{i,j}+\ell_{i,j}=1;\nonumber\\
={}& \sum_{n=1}^{\infty}\left(w_{i,j}-w_{i,j}^{(n)}\right)+\sum_{n=1}^{\infty}\left(\ell_{i,j}-\ell_{i,j}^{(n)}\right),\label{exp_time_1}
\end{align}
where, in the last step, we are able to rearrange the summands because each $w_{i,j}-w_{i,j}^{(n)}$ and each $\ell_{i,j}-\ell_{i,j}^{(n)}$ is non-negative. 

In what follows, in order to make the exposition less cluttered, we adopt the following abbreviations, in addition to \eqref{alpha} and \eqref{beta}:
\begin{align}
{}&\alpha_{i,j}^{(n)} = p_{-1}w_{j,i-1}^{(n)}+p_{0}w_{j,i}^{(n)}+p_{1}w_{j,i+1}^{(n)},\nonumber\\
{}&\beta_{i,j}^{(n)} = p_{-1}\left(1-\ell_{j,i-1}^{(n)}\right)+p_{0}\left(1-\ell_{j,i}^{(n)}\right)+p_{1}\left(1-\ell_{j,i+1}^{(n)}\right),\nonumber
\end{align}
for each $i,j \in \{1,2,\ldots,\kappa-1\}$ and for each $n \in \mathbb{N}$. Note that, when $i,j \in \{2,\ldots,\kappa-2\}$, we have, from \eqref{recurrence_2} and \eqref{recurrence_4} and using the notations introduced above, for each $n \in \mathbb{N}_{0}$,
\begin{align}
w_{i,j}-w_{i,j}^{(n+1)}={}&G\left(\beta_{i,j}^{(n)}\right)-G\left(\beta_{i,j}\right) = G'(\xi)\left\{p_{1}\left(\ell_{j,i+1}-\ell_{j,i+1}^{(n)}\right)+p_{0}\left(\ell_{j,i}-\ell_{j,i}^{(n)}\right)+p_{-1}\left(\ell_{j,i-1}-\ell_{j,i-1}^{(n)}\right)\right\},\label{intermediate_1}
\end{align}
where the last step is obtained by the use of the mean value theorem on the differentiable pgf $G$, for some $\beta_{i,j} < \xi < \beta_{i,j}^{(n)}$. Since $G$ is a pgf, $G'$ is monotonically increasing, so that we have
\begin{equation}
G'\left(\beta_{i,j}\right) \leqslant G'(\xi) \leqslant G'\left(\beta_{i,j}^{(n)}\right).\label{intermediate_2}  
\end{equation} 
Since, by \eqref{eq:limits} and by the continuity of $G'$, we have $\lim_{n \rightarrow \infty}G'\left(\beta_{i,j}^{(n)}\right) = G'\left(\beta_{i,j}\right)$, hence, for $\epsilon$ fixed at the very outset of \S\ref{sec:proof_thm_3}, there exists $N_{i,j} \in \mathbb{N}$ such that for all $n \geqslant N_{i,j}$, we have
\begin{equation}
G'\left(\beta_{i,j}^{(n)}\right) \leqslant (1+\epsilon)G'\left(\beta_{i,j}\right).\label{intermediate_3}
\end{equation}
From \eqref{intermediate_1}, \eqref{intermediate_2} and \eqref{intermediate_3}, we obtain, for $n \geqslant N_{i,j}$,
\begin{align}
w_{i,j}-w_{i,j}^{(n+1)} \leqslant{}& (1+\epsilon) G'\left(\beta_{i,j}\right)\left\{p_{1}\left(\ell_{j,i+1}-\ell_{j,i+1}^{(n)}\right)+p_{0}\left(\ell_{j,i}-\ell_{j,i}^{(n)}\right)+p_{-1}\left(\ell_{j,i-1}-\ell_{j,i-1}^{(n)}\right)\right\}.\label{intermediate_4}
\end{align}
It is worthwhile to note here that when $i=1$ and $j \in \{1,2,\ldots,\kappa-1\}$, the inequality in \eqref{intermediate_4} is replaced by
\begin{equation}
w_{1,j}-w_{1,j}^{(n+1)} \leqslant (1+\epsilon) G'\left(\beta_{1,j}\right)\left\{p_{1}\left(\ell_{j,2}-\ell_{j,2}^{(n)}\right)+p_{0}\left(\ell_{j,1}-\ell_{j,1}^{(n)}\right)\right\}\label{intermediate_5}
\end{equation}
for all $n \geqslant N_{1,j}$, and when $i=\kappa-1$ while $j \in \{1,2,\ldots,\kappa-1\}$, the inequality in \eqref{intermediate_4} morphs into
\begin{align}
w_{\kappa-1,j}-w_{\kappa-1,j}^{(n+1)} \leqslant{}& (1+\epsilon) G'\left(\beta_{\kappa-1,j}\right)\left\{p_{0}\left(\ell_{j,\kappa-1}-\ell_{j,\kappa-1}^{(n)}\right)+p_{-1}\left(\ell_{j,\kappa-2}-\ell_{j,\kappa-2}^{(n)}\right)\right\}\label{intermediate_6}
\end{align}
for all $n \geqslant N_{\kappa-1,j}$. On the other hand, making use of \eqref{recurrence_1} and \eqref{recurrence_3}, and using the notations introduced above, we obtain
\begin{align}
\ell_{i,j}-\ell_{i,j}^{(n+1)} ={}& G\left(\alpha_{i,j}\right)-G\left(\alpha_{i,j}^{(n)}\right) = G'(\zeta)\left\{p_{-1}\left(w_{j,i-1}-w_{j,i-1}^{(n)}\right)+p_{0}\left(w_{j,i}-w_{j,i}^{(n)}\right)+p_{1}\left(w_{j,i+1}-w_{j,i+1}^{(n)}\right)\right\}\nonumber
\end{align}
for some $\zeta$ with $\alpha_{i,j}^{(n)} < \zeta < \alpha_{i,j}$, by the mean value theorem. As above, using the monotonically increasing nature of $G'$, we can then deduce the bound
\begin{equation}
\ell_{i,j}-\ell_{i,j}^{(n+1)} \leqslant G'\left(\alpha_{i,j}\right)\left\{p_{-1}\left(w_{j,i-1}-w_{j,i-1}^{(n)}\right)+p_{0}\left(w_{j,i}-w_{j,i}^{(n)}\right)+p_{1}\left(w_{j,i+1}-w_{j,i+1}^{(n)}\right)\right\}.\label{intermediate_7}
\end{equation}
Once again, two special cases are also to be noted: when $i=1$ and $j \in \{1,2,\ldots,\kappa-1\}$, we have
\begin{equation}
\ell_{1,j}-\ell_{1,j}^{(n+1)} \leqslant G'\left(\alpha_{1,j}\right)\left\{p_{0}\left(w_{j,1}-w_{j,1}^{(n)}\right)+p_{1}\left(w_{j,2}-w_{j,2}^{(n)}\right)\right\},\label{intermediate_8}
\end{equation}
and when $i=\kappa-1$ and $j \in \{1,2,\ldots,\kappa-1\}$, we have
\begin{equation}
\ell_{\kappa-1,j}-\ell_{\kappa-1,j}^{(n+1)} \leqslant G'\left(\alpha_{\kappa-1,j}\right)\left\{p_{-1}\left(w_{j,\kappa-2}-w_{j,\kappa-2}^{(n)}\right)+p_{0}\left(w_{j,\kappa-1}-w_{j,\kappa-1}^{(n)}\right)\right\}.\label{intermediate_9}
\end{equation}

From \eqref{intermediate_4}, \eqref{intermediate_5}, \eqref{intermediate_6}, \eqref{intermediate_7}, \eqref{intermediate_8} and \eqref{intermediate_9}, we see that the term $\left(w_{i,j}-w_{i,j}^{(n-1)}\right)$ appears
\begin{enumerate}
\item with coefficient $(1+\epsilon)G'\left(\beta_{i-1,j-1}\right)p_{1}G'\left(\alpha_{j-1,i}\right)p_{1}$ in the upper bound for $\left(w_{i-1,j-1}-w_{i-1,j-1}^{(n+1)}\right)$,
\item with coefficient $(1+\epsilon)G'\left(\beta_{i-1,j}\right)p_{1}G'\left(\alpha_{j,i}\right)p_{0}$ in the upper bound for $\left(w_{i-1,j}-w_{i-1,j}^{(n+1)}\right)$,
\item with coefficient $(1+\epsilon)G'\left(\beta_{i-1,j+1}\right)p_{1}G'\left(\alpha_{j+1,i}\right)p_{-1}$ in the upper bound for $\left(w_{i-1,j+1}-w_{i-1,j+1}^{(n+1)}\right)$,
\item with coefficient $(1+\epsilon)G'\left(\beta_{i,j-1}\right)p_{0}G'\left(\alpha_{j-1,i}\right)p_{1}$ in the upper bound for $\left(w_{i,j-1}-w_{i,j-1}^{(n+1)}\right)$,
\item with coefficient $(1+\epsilon)G'\left(\beta_{i,j}\right)p_{0}G'\left(\alpha_{j,i}\right)p_{0}$ in the upper bound for $\left(w_{i,j}-w_{i,j}^{(n+1)}\right)$,
\item with coefficient $(1+\epsilon)G'\left(\beta_{i,j+1}\right)p_{0}G'\left(\alpha_{j+1,i}\right)p_{-1}$ in the upper bound for $\left(w_{i,j+1}-w_{i,j+1}^{(n+1)}\right)$,
\item with coefficient $(1+\epsilon)G'\left(\beta_{i+1,j-1}\right)p_{-1}G'\left(\alpha_{j-1,i}\right)p_{1}$ in the upper bound for $\left(w_{i+1,j-1}-w_{i+1,j-1}^{(n+1)}\right)$,
\item with coefficient $(1+\epsilon)G'\left(\beta_{i+1,j}\right)p_{-1}G'\left(\alpha_{j,i}\right)p_{0}$ in the upper bound for $\left(w_{i+1,j}-w_{i+1,j}^{(n+1)}\right)$,
\item with coefficient $(1+\epsilon)G'\left(\beta_{i+1,j+1}\right)p_{-1}G'\left(\alpha_{j+1,i}\right)p_{-1}$ in the upper bound for $\left(w_{i+1,j+1}-w_{i+1,j+1}^{(n+1)}\right)$.
\end{enumerate}
At this point, we look for positive constants $\gamma_{(i,j)}$, for all $i,j \in \{1,2,\ldots,\kappa-1\}$, and $\lambda \in (0,1)$, such that
\begin{align}
\lambda \gamma_{(i,j)} = \sum_{s,t \in \{1,2,\ldots,\kappa-1\}}C_{(s,t)}^{(i,j)}\gamma_{(s,t)}\label{intermediate_10}
\end{align}
where, for each $s, t \in \{1,2,\ldots,\kappa-1\}$, the constants $C_{(s,t)}^{(i,j)}$ are defined exactly as in \eqref{intermediate_11}. The intuition behind \eqref{intermediate_10} is as follows. If we consider the linear combination 
\begin{equation}
\sum_{i,j \in \{1,2,\ldots,\kappa-1\}}\gamma_{(i,j)}\left(w_{i,j}-w_{i,j}^{(n+1)}\right),\label{lin_comb}
\end{equation}
then it will be bounded above by
\begin{equation}
(1+\epsilon)\sum_{i,j \in \{1,2,\ldots,\kappa-1\}}\left(\sum_{s,t \in \{1,2,\ldots,\kappa-1\}}C_{(s,t)}^{(i,j)}\gamma_{(s,t)}\right)\left(w_{i,j}-w_{i,j}^{(n-1)}\right),\nonumber
\end{equation}
a fact that is obtained from the coefficients of $\left(w_{i,j}-w_{i,j}^{(n+1)}\right)$ enumerated above in the upper bounds for the various summands of \eqref{lin_comb}. If \eqref{intermediate_10} holds, then we are able to conclude that
\begin{equation}
\sum_{i,j \in \{1,2,\ldots,\kappa-1\}}\gamma_{i,j}\left(w_{i,j}-w_{i,j}^{(n+1)}\right) \leqslant (1+\epsilon)\lambda\sum_{i,j \in \{1,2,\ldots,\kappa-1\}}\gamma_{i,j}\left(w_{i,j}-w_{i,j}^{(n-1)}\right).\label{intermediate_12}
\end{equation}
If $(1+\epsilon)\lambda \in (0,1)$, then we see that each of the sequences $\left\{\sum_{i,j \in \{1,2,\ldots,\kappa-1\}}\gamma_{i,j}\left(w_{i,j}-w_{i,j}^{(2n)}\right)\right\}_{n \in \mathbb{N}_{0}}$ and $\left\{\sum_{i,j \in \{1,2,\ldots,\kappa-1\}}\gamma_{i,j}\left(w_{i,j}-w_{i,j}^{(2n+1)}\right)\right\}_{n \in \mathbb{N}_{0}}$ decays geometrically, and convergence will be ensured. To this end, we consider a $(\kappa-1)^{2} \times (\kappa-1)^{2}$ matrix $\mathcal{C}$ in which the rows and columns are indexed by $(s,t)$, for $s,t \in \{1,2,\ldots,\kappa-1\}$ (thus, the first row is indexed $(1,1)$, the second row is indexed $(1,2)$, and so on, and the last row is indexed $(\kappa-1,\kappa-1)$; likewise, the first column is indexed $(1,1)$, the second column is indexed $(1,2)$, and so on, and the last column is indexed $(\kappa-1,\kappa-1)$). The entry in the $(u,v)$-th row and $(s,t)$-th column is given by $C_{(s,t)}^{(u,v)}$, for each $s,t,u,v \in \{1,2,\ldots,\kappa-1\}$. Let $\Gamma$ be the $(\kappa-1)^{2}\times 1$ column vector in which the rows are indexed by $(s,t)$, for $s,t \in \{1,2,\ldots,\kappa-1\}$ (thus, the first row indexed by $(1,1)$, the second $(1,2)$ and so on, and the last row indexed by $(\kappa-1,\kappa-1)$), and the entry in the $(s,t)$-th row is $\gamma_{(s,t)}$. With these notations, it is evident that the expression in the right side of \eqref{intermediate_10} is equal to the entry in the $(i,j)$-th row of $\mathcal{C}\Gamma$. If \eqref{intermediate_10} holds, it would be equivalent to the identity $\mathcal{C}\Gamma = \lambda \Gamma$. This would make $\lambda$ an eigenvalue of $\mathcal{C}$ and $\Gamma$ a corresponding eigenvector. 

Recall that a non-negative square matrix (i.e.\ each entry is non-negative) $A$, of dimension $n \times n$, is called \emph{irreducible} if for each pair $(i,j)$, where $i,j \in \{1,2,\ldots,n\}$, there exists some $k \in \mathbb{N}$ such that the $(i,j)$-th entry of $A^{k}$ is strictly positive. If we form a graph $G(A)$ on the vertex set $\{1,2,\ldots,n\}$ such that there is an edge between vertices $i$ and $j$ if and only if the $(i,j)$-th entry of $A$ is strictly positive, then the above amounts to saying that, for each pair $(i,j)$ with $i,j \in \{1,2,\ldots,n\}$, there exists some $k \in \mathbb{N}$ such that there exists a path of length that begins in $i$ and ends in $j$ (i.e.\ that one can find vertices $\ell_{1}, \ell_{2}, \ldots, \ell_{k-1} \in \{1,2,\ldots,n\}$ such that $i$ and $\ell_{1}$ are adjacent, $\ell_{1}$ and $\ell_{2}$ are adjacent, $\ldots$, $\ell_{k-1}$ and $j$ are adjacent).

\sloppy Next, recall the Perron-Frobenius Theorem for non-negative, irreducible matrices (we only state the part of it that is necessary for this paper):
\begin{theorem}\label{thm:perron_frobenius}
Let $A$ be a non-negative, irreducible matrix. Then it has a right eigenvector $v$ each of whose entries is strictly positive, and the corresponding eigenvalue of $A$ is equal to the spectral radius of $A$. 
\end{theorem}
Moreover, \cite{perron1907theorie} shows that the spectral radius $\rho(A)$ is bounded above by $\max\left\{R_{i}(A): i \in \{1,2,\ldots,n\}\right\}$, where $R_{i}(A)$ is the $i$-th row-sum of $A$, i.e.\ $R_{i}(A) = \sum_{j=1}^{n}A_{i,j}$, where $A_{i,j}$ is the $(i,j)$-th entry of $A$. In other words, 
\begin{equation}
\rho(A) \leqslant \max\left\{R_{i}(A): i \in \{1,2,\ldots,n\}\right\}.\label{spec_radius_bound}
\end{equation}

Note that $\mathcal{C}$ is a non-negative matrix. It is, moreover, an \emph{irreducible} matrix, because of the following reason: let us consider $(i,j)$ and $(i',j')$, where $i,j,i',j' \in \{1,2,\ldots,\kappa-1\}$. Let $i'=i+t_{1}$ and $j'=j+t_{2}$, where $t_{1}, t_{2} \in \mathbb{Z}$. Without loss of generality, let us assume that $t_{1} \geqslant 0$ whereas $t_{2} \leqslant 0$ (the other three cases would be analogous). Since it has been assumed that each of $p_{-1}$, $p_{0}$ and $p_{1}$ is strictly positive, we conclude, from \eqref{intermediate_11}, that each of $C_{(i+1,j)}^{(i,j)}$, $C_{(i+2,j)}^{(i+1,j)}$, $\ldots$, $C_{(i+t_{1},j)}^{(i+t_{1}-1,j)}$ is strictly positive, implying that there is a path from the vertex $(i,j)$ to the vertex $(i+t_{1},j)$ in the graph $G(\mathcal{C})$. Likewise, each of $C_{(i+t_{1},j-1)}^{(i+t_{1},j)}$, $C_{(i+t_{1},j-2)}^{(i+t_{1},j-1)}$, $\ldots$, $C_{(i+t_{1},j+t_{2})}^{(i+t_{1},j-t_{2}+1)}$ is strictly positive, implying that there is a path from the vertex $(i+t_{1},j)$ to $(i+t_{1},j+t_{2})$ in the graph $G(\mathcal{C})$. This completes the proof of the assertion that $\mathcal{C}$ is irreducible.

Consequently, Theorem~\ref{thm:perron_frobenius} as well as \eqref{spec_radius_bound} are applicable to our set-up. In \eqref{intermediate_10}, we then want $\lambda$ to be the spectral radius of $\mathcal{C}$, and the vector $\Gamma$ to be the (normalized, i.e.\ of unit Euclidean norm) eigenvector of $\mathcal{C}$ corresponding to $\lambda$. This ensures that each coordinate of $\Gamma$, i.e.\ each $\gamma_{i,j}$, is strictly positive, as intended. Moreover, \eqref{thm:3:eq} implies that 
\begin{align}
\max\left\{R_{(i,j)}(\mathcal{C}): (i,j) \in \{1,2,\ldots,\kappa-1\}^{2}\right\} < 1,\nonumber
\end{align}
where $R_{(i,j)}(\mathcal{C})$ is the sum of the entries in the $(i,j)$-th row of $\mathcal{C}$, so that an application of \eqref{spec_radius_bound} yields $\lambda < 1$, as desired. The time has now come to specify the value of $\epsilon$, as promised at the beginning of \S\ref{sec:proof_thm_3}: we choose $\epsilon$ such that $(1+\epsilon)\lambda < 1$ (this is possible since \eqref{thm:3:eq} and \eqref{spec_radius_bound} together ensure that $\lambda$ is \emph{strictly} less than $1$). As asserted above, we then obtain a geometric decay in each of the sequences $\left\{\sum_{i,j \in \{1,2,\ldots,\kappa-1\}}\gamma_{i,j}\left(w_{i,j}-w_{i,j}^{(2n)}\right)\right\}_{n \in \mathbb{N}_{0}}$ and $\left\{\sum_{i,j \in \{1,2,\ldots,\kappa-1\}}\gamma_{i,j}\left(w_{i,j}-w_{i,j}^{(2n+1)}\right)\right\}_{n \in \mathbb{N}_{0}}$, making them summable. Since each $\gamma_{i,j}$ is strictly positive, this conclusion further implies that each of the sequences $\left\{\left(w_{i,j}-w_{i,j}^{(2n)}\right)\right\}_{n \in \mathbb{N}_{0}}$ and $\left\{\left(w_{i,j}-w_{i,j}^{(2n+1)}\right)\right\}_{n \in \mathbb{N}_{0}}$ is summable for \emph{every} $i,j \in \{1,2,\ldots,\kappa-1\}$, so that the sum $\sum_{n=1}^{\infty}\left(w_{i,j}-w_{i,j}^{(n)}\right)$ is finite for \emph{every} $i,j \in \{1,2,\ldots,\kappa-1\}$. This, along with the bounds in \eqref{intermediate_7}, \eqref{intermediate_8} and \eqref{intermediate_9}, ensures that the sum $\sum_{n=1}^{\infty}\left(\ell_{i,j}-\ell_{i,j}^{(n)}\right)$, too, is finite for \emph{every} $i,j \in \{1,2,\ldots,\kappa-1\}$. Therefore, from \eqref{exp_time_1}, we conclude that $\E[\mathcal{T}_{i,j}]$ is finite for \emph{every} $i,j \in \{1,2,\ldots,\kappa-1\}$ when \eqref{thm:3:eq} holds.
 
\subsection{Proof of Theorem~\ref{thm:kappa=2}}\label{sec:kappa=2_proofs}
When $\kappa=2$, \eqref{recurrence_3} and \eqref{recurrence_4} boil down to $\ell_{1,1}=G\left(p_{-1}+p_{0}w_{1,1}\right)$ and $w_{1,1}=1-G\left(1-p_{0}\ell_{1,1}-p_{1}\right)$. Setting $w'=p_{-1}+p_{0}w_{1,1}$ and $\ell'=1-p_{1}-p_{0}\ell_{1,1}$, we see that $w'$ is the smallest fixed point and $\ell'$ is the largest fixed point of the function
\begin{equation}
h(x)=1-p_{1}-p_{0}G\left\{1-p_{1}-p_{0}G(x)\right\}=g\big(g(x)\big), \quad \text{where}\quad g(x)=1-p_{1}-p_{0}G(x),\label{g_h_defn_kappa=2}
\end{equation}
in the interval $[0,1]$ (in fact, we have $p_{-1} \leqslant w' \leqslant \ell' \leqslant 1-p_{1}$, and there lies no fixed point of $h$ in the intervals $[0,p_{-1})$ and $(1-p_{1},1]$), and the probability $d_{1,1}$ of draw equals $0$ if and only if $h$ has a unique fixed point in the interval $[p_{-1},1-p_{1}]$.

Since each of these functions involves a single variable, analyzing them and seeking their fixed points in a specified interval is an easier task to accomplish (compared to when we consider higher values of $\kappa$). As stated in Theorem~\ref{thm:kappa=2}, we now focus on \emph{specific} offspring distributions $\chi$, which also specifies the form of the pgf $G$. We begin by stating and proving a lemma that will be useful for each of \S\ref{subsec:binomial}, \S\ref{subsec:Poisson}, \S\ref{subsec:neg_bin} and \S\ref{subsec:0,d}:
\begin{lemma}\label{lem:max_derivative_leq_1}
As long as the derivative $h'(x)$ of the function $h$, defined in \eqref{g_h_defn_kappa=2}, is bounded above by $1$ for all $x \in [0,1]$, the function $h$ has a unique fixed point in $[0,1]$.
\end{lemma}
\begin{proof}
Assuming $p_{1}, p_{0} \in [0,1)$, we have $h(0) > 0$, i.e.\ the curve $y=h(x)$ lies above $y=x$ at $x=0$. Since $w'$ is the smallest positive fixed point of $h$, it must be the case that the curve $y=h(x)$ travels from \emph{above} $y=x$ to \emph{beneath} $y=x$ at $x=\ell_{1,1}$. If there exists yet another fixed point $\beta \in (0,1)$ of $h$, then $w' < \beta$. Assuming that there exists no fixed point of $h$ in the sub-interval $(w',\beta)$, the curve $y=h(x)$ must travel from \emph{beneath} $y=x$ to \emph{above} $y=x$ at $x=\beta$, and therefore, the slope of $y=h(x)$ must exceed the slope of $y=x$ at $x=\beta$. But the slope of $y=x$ equals $1$, and the slope of $y=h(x)$ is bounded above by $1$ throughout $[0,1]$. This leads to a contradiction, implying that the fixed point $\beta$ cannot exist.
\end{proof}

\subsubsection{When the offspring distribution is Binomial}\label{subsec:binomial} Let $\chi$ be Binomial$(d,\pi)$ for some $d \geqslant 2$ and $\pi \in (0,1]$. We have $G(x)=\left\{(1-\pi)+\pi x\right\}^{d}$, which yields $G'(x)=\pi d \left\{(1-\pi)+\pi x\right\}^{d-1}$, so that from \eqref{g_h_defn_kappa=2}, we have:
\begin{align}
h'(x)={}&g'(g(x))g'(x)=p_{0}^{2}\pi^{2}d^{2}\left[(1-\pi p_{1})\left\{\pi x + (1-\pi)\right\} - \pi p_{0} \left\{\pi x + (1-\pi)\right\}^{d+1}\right]^{d-1}.\label{h'_binomial}
\end{align}
Let us define the function $f_{1}(x)=(1-\pi p_{1})\left\{\pi x + (1-\pi)\right\} - \pi p_{0} \left\{\pi x + (1-\pi)\right\}^{d+1}$, so that its derivative, $f'_{1}(x)=\pi (1-\pi p_{1}) - (d+1) \pi^{2} p_{0} \left\{\pi x + (1-\pi)\right\}^{d}$, is strictly positive if and only if $x < x_{0}$, where $x_{0}=\pi^{-1}\left(1 - \pi p_{1}\right)^{1/d}\left\{(d+1) \pi p_{0}\right\}^{-1/d} + 1 - \pi^{-1}$. This tells us that the maximum of the derivative $h'(x)$ is attained at $x=x_{0}$, and this maximum value equals $h'(x_{0})=p_{0}^{(d+1)/d}\pi^{(d+1)/d}d^{d+1}\left(1-\pi p_{1}\right)^{(d+1)(d-1)/d}(d+1)^{-(d-1)(d+1)/d}$. Consequently, $h'(x_{0}) \leqslant 1$, implying that the derivative of $h$ is bounded above \emph{everywhere} by $1$, if and only if 
\begin{equation}\label{h'_bounded_binomial}
p_{0}\pi\left(1-\pi p_{1}\right)^{d-1} \leqslant (d+1)^{d-1}d^{-d}.
\end{equation}
From Lemma~\ref{lem:max_derivative_leq_1}, we conclude that $h$ has a unique fixed point in $[0,1]$ whenever \eqref{h'_bounded_binomial} holds. 

Next, assuming $p_{0} > 0$, we show that $h(x_{0}) \leqslant x_{0}$, which allows us to conclude that $w' \leqslant x_{0}$ (since, in the interval $[0,w')$, the curve $y=h(x)$ lies \emph{above} the line $y=x$). A few careful algebraic manipulations reveal that the inequality $h(x_{0}) \leqslant x_{0}$ is true whenever 
\begin{align}
\pi p_{0} d^{d} \gamma^{-d} (d+1)^{-d} + \gamma (d+1)^{-1/d} \pi^{-1/d} p_{0}^{-1/d} \geqslant 1, \text{ where } \gamma = (1-\pi p_{1})^{-(d-1)/d}.\label{objective_binomial}
\end{align}
Setting $f_{2}(x) = \pi p_{0} d^{d} x^{-d} (d+1)^{-d} + x (d+1)^{-1/d} \pi^{-1/d} p_{0}^{-1/d}$, we see that $f_{2}$ is strictly decreasing on $0 \leqslant x < \pi^{1/d} p_{0}^{1/d} d (d+1)^{-(d-1)/d}$ and strictly increasing on $x > \pi^{1/d} p_{0}^{1/d} d (d+1)^{-(d-1)/d}$. Therefore, $f_{2}$ attains its minimum value at $x=\pi^{1/d} p_{0}^{1/d} d (d+1)^{-(d-1)/d}$, and it is straightforward to verify that this minimum value equals $1$, thus establishing that \eqref{objective_binomial} indeed holds for all values of $p_{1}$, $p_{0}$ and $p_{-1}$. Ultimately, this leads to the conclusion that $w' \leqslant x_{0}$ for all values of $p_{1}$, $p_{0}$ and $p_{-1}$.

Note that the function $g$, defined in \eqref{g_h_defn_kappa=2}, is strictly decreasing (as long as $p_{0} > 0$), with $g(0) > 0$ and $g(1) < 1$, so that $g$ must have a unique fixed point, say $\alpha$, in $[0,1]$, and $\alpha$ is a fixed point of $h$ as well. Our final task is to show that $x_{0} < \alpha$ whenever \eqref{h'_bounded_binomial} \emph{fails} to hold. We prove this by showing that $g(x_{0}) > x_{0}$ (following which, the conclusion can be drawn from the observation that the curve $y=g(x)$ lies above the line $y=x$ for $x \in [0,\alpha)$ and beneath the line $y=x$ for $x \in (\alpha,1]$). Straightforward algebraic manipulations reveal that the inequality $g(x_{0}) > x_{0}$ becomes equivalent to the inequality $\pi^{1/d} p_{0}^{1/d} (1-\pi p_{1})^{(d-1)/d} > d^{-1} (d+1)^{(d-1)/d}$, which is exactly the negation of the criterion in \eqref{h'_bounded_binomial}. Therefore, when \eqref{h'_bounded_binomial} does \emph{not} hold, we have $w' \leqslant x_{0} < \alpha$, asserting the existence of two distinct fixed points of $h$, namely $w'$ and $\alpha$, in $[0,1]$. 

The final conclusion, therefore, is that the probability of draw, $d_{1,1}$, is $0$ if and only if the inequality in \eqref{h'_bounded_binomial} is true. Note that, when $p_{0}=0$, the fate of the game is decided in its very first round, and $d_{1,1}$ must equal $0$ in that case. 

\subsubsection{When the offspring distribution is Poisson}\label{subsec:Poisson} Let $p_{0} > 0$ and $\chi$ be Poisson$(\lambda)$, so that $G(x)=e^{\lambda(x-1)}$, leading to $G'(x)=\lambda G(x)$. This, along with \eqref{g_h_defn_kappa=2}, yields
\begin{align}
h'(x)={}&p_{0}^{2} \lambda^{2}\exp\left\{\lambda(x-1)-\lambda p_{1} - \lambda p_{0} e^{\lambda(x-1)}\right\}.\nonumber
\end{align}
Setting $f_{1}(x)=\lambda(x-1)-\lambda p_{1} - \lambda p_{0} e^{\lambda(x-1)}$, we see that $f'_{1}(x) > 0$ if and only if $x < x_{0}$, where $x_{0}=1-\lambda^{-1}\log \lambda-\lambda^{-1}\log p_{0}$. The maximum value of $h'(x)$ is thus attained at $x=x_{0}$, and this maximum value equals $h'(x_{0})=p_{0} \lambda e^{-\lambda p_{1} - 1}$. Note that $h'(x_{0}) \leqslant 1$ if and only if 
\begin{equation}\label{h'_bounded_Poisson}
p_{0} \lambda e^{-\lambda p_{1}} \leqslant e,
\end{equation}\
so that, by Lemma~\ref{lem:max_derivative_leq_1}, the function $h$ has a unique fixed point in $[0,1]$ whenever \eqref{h'_bounded_Poisson} holds. 

As in \S\ref{subsec:binomial}, we now show that $w' \leqslant x_{0}$ by showing that $h(x_{0}) \leqslant x_{0}$, which is equivalent to the inequality $\log x \leqslant e^{-1}x$. Defining the function $f_{2}(x)=xe^{-1}-\log x$, we see that $f_{2}$ is strictly decreasing for $0 < x < e$ and strictly increasing for $x > e$, leading to its minimum value, attained at $x=e$, being equal to $0$. This proves that indeed, for all $p_{-1}$, $p_{0}$ and $p_{1}$, we have $w' \leqslant x_{0}$.

Next, we show that $x_{0} < \alpha$ whenever \eqref{h'_bounded_Poisson} does not hold, where recall, from \S\ref{subsec:binomial}, that $\alpha$ is the unique fixed point, in $[0,1]$, of the function $g$ defined in \eqref{g_h_defn_kappa=2}. We show this by proving that $g(x_{0}) > x_{0}$ (and making use of the fact that the curve $y=g(x)$ lies above the line $y=x$ for $x \in [0,\alpha)$, and beneath $y=x$ for $x \in (\alpha,1]$). Straightforward algebraic manipulations reveal that the inequality $g(x_{0}) > x_{0}$ is equivalent to $p_{0} \lambda e^{-\lambda p_{1}} > e$, which is precisely the negation of the inequality in \eqref{h'_bounded_Poisson}. This goes to show, much like in \S\ref{subsec:binomial}, that when \eqref{h'_bounded_Poisson} fails to hold, $w' \leqslant x_{0} < \alpha$, so that $h$ has two distinct fixed points, $w'$ and $\alpha$, in $[0,1]$. 

We conclude that the probability of draw, $d_{1,1}$, is $0$ if and only if \eqref{h'_bounded_Poisson} holds. When $p_{0}=0$, the game ends after the very first round, leading to $d_{1,1}=0$, and \eqref{h'_bounded_Poisson} holds in that case.

\subsubsection{When the offspring distribution is negative binomial}\label{subsec:neg_bin} Here, we consider the offspring distribution $\chi$ of our GW tree to be negative binomial$(r,\pi)$, i.e.\ where the number of children of any vertex of $T_{\chi}$ represents the total number of trials resulting in failures that are required in order to obtain precisely $r$ many successes, where $r \in \mathbb{N}$ is pre-specified. A trial here refers to that of a random experiment that either results in success with probability $\pi \in (0,1)$, or in failure with probability $1-\pi$, and the trials are assumed to be mutually independent. The pgf corresponding to $\chi$ is $G(x) = \pi^{r}\left\{1-(1-\pi)x\right\}^{-r}$, so that from \eqref{g_h_defn_kappa=2}, we obtain $g(x)=1-p_{1}-p_{0}\pi^{r}\left\{1-(1-\pi)x\right\}^{-r}$. This leads to
\begin{align}
h'(x)={}& r^{2}p_{0}^{2}(1-\pi)^{2}\pi^{2r} \left[\left(p_{1}+\pi-p_{1}\pi\right)\left\{1-(1-\pi)x\right\} + \frac{(1-\pi)p_{0}\pi^{r}}{\left\{1-(1-\pi)x\right\}^{(r-1)}}\right]^{-(r+1)},\label{derivative_g_{p,q,pi,r}^{(2)}_negative_binomial}
\end{align}
where $h$ is as defined in \eqref{g_h_defn_kappa=2}. We define the function $f_{1}(x) = \left(p_{1}+\pi-p_{1}\pi\right)\left\{1-(1-\pi)x\right\} + (1-\pi)p_{0}\pi^{r}\left\{1-(1-\pi)x\right\}^{-(r-1)}$. The derivative of $f_{1}$, given by $f'_{1}(x)=-\left(p_{1}+\pi-p_{1}\pi\right)(1-\pi) + (r-1)(1-\pi)^{2}p_{0}\pi^{r}\left\{1-(1-\pi)x\right\}^{-r}$, is strictly positive if and only if
\begin{align}
&x > x_{0}, \text{ where } x_{0} = \frac{1}{1-\pi}\left[1 - \frac{(r-1)^{\frac{1}{r}}(1-\pi)^{\frac{1}{r}}p_{0}^{\frac{1}{r}}\pi}{\left(p_{1}+\pi-p_{1}\pi\right)^{\frac{1}{r}}}\right].\nonumber
\end{align}
Consequently, the maximum of $h'(x)$ is attained at $x=x_{0}$, and this maximum value equals
\begin{align}
h'(x_{0})=\left(\frac{(r-1)^{(r+1)/r}(1-\pi)^{1/r}p_{0}^{1/r}\pi}{\left(p_{1}+\pi-p_{1}\pi\right)^{(r+1)/r}r}\right)^{r-1}.\nonumber
\end{align}
Note that $h'(x_{0}) \leqslant 1$, and by Lemma~\ref{lem:max_derivative_leq_1}, the probability of draw, $d_{1,1}$, equals $0$, whenever \begin{equation}
(r-1)^{r+1}(1-\pi)p_{0}\pi^{r} \leqslant \left(p_{1}+\pi-p_{1}\pi\right)^{r+1}r^{r}.\label{eq:lem_neg_bin_2}
\end{equation}

As in \S\ref{subsec:binomial} and \S\ref{subsec:Poisson}, we show that $w' \leqslant x_{0}$ by proving that $h(x_{0}) \leqslant x_{}$ for all values of $p_{-1}$, $p_{0}$ and $p_{1}$. Tedious algebraic manipulations reveal that $h(x_{0}) \leqslant x_{0}$ if and only if
\begin{align}
(r-1)^{1/r}\gamma - \frac{(r-1)^{r} \gamma^{r}}{r^{r}} \leqslant 1, \text{ where } \gamma = (1-\pi)^{1/r}p_{0}^{1/r}\pi\left(p_{1}+\pi-p_{1}\pi\right)^{-(r+1)/r}.\label{objective_neg_bin}
\end{align}
Letting $f_{2}(x) = (r-1)^{\frac{1}{r}}x - (r-1)^{r} r^{-r} x^{r}$, we see that its derivative $f'_{2}(x)=(r-1)^{\frac{1}{r}} - (r-1)^{r} r^{-(r-1)} x^{r-1}$ is strictly positive if and only if $x < r (r-1)^{-(r+1)/r}$. Therefore, $f_{2}(x)$ is strictly increasing for $0 \leqslant x < r (r-1)^{-(r+1)/r}$ and strictly decreasing for $r (r-1)^{-(r+1)/r} < x \leqslant 1$, and its maximum value, attained at $x=r (r-1)^{-(r+1)/r}$, equals $1$. This shows that \eqref{objective_neg_bin} does hold, thereby establishing that $w' \leqslant x_{0}$.

As in \S\ref{subsec:binomial} and \S\ref{subsec:Poisson}, the final step of our argument involves proving that $x_{0} < \alpha$ whenever \eqref{eq:lem_neg_bin_2} does not hold, where $\alpha$, recall, is the unique fixed point, in $[0,1]$, of the function $g$, defined in \eqref{g_h_defn_kappa=2}. This is accomplished by showing that the inequality $g_{0}(x_{0}) > x_{0}$ holds whenever \eqref{eq:lem_neg_bin_2} does not (and subsequently, making use of the fact that the curve $y=g(x)$ lies above the line $y=x$ for $x \in [0,\alpha)$, and beneath $y=x$ for $x \in (\alpha,1]$). Routine algebraic manipulations reveal that this inequality is equivalent to $(r-1)^{r+1}(1-\pi)p_{0}\pi^{r} > r^{r}\left(p_{1}+\pi-p_{1}\pi\right)^{r+1}$, which is exactly the negation of \eqref{eq:lem_neg_bin_2}. This goes to show that when \eqref{eq:lem_neg_bin_2} does not hold, we have $w' \leqslant x_{0} < \alpha$, so that $h$ has two distinct fixed points, $w'$ and $\alpha$, in $[0,1]$.

The final conclusion from the previous few paragraphs is that the probability of draw, $d_{1,1}$, equals $0$ if and only if \eqref{eq:lem_neg_bin_2} holds.

\begin{remark}
It is worthwhile to note that when we set $r=1$ in \eqref{eq:lem_neg_bin_2}, the inequality is automatically satisfied for \emph{all} values of $p_{-1}$, $p_{0}$ and $p_{1}$. This goes to show that when $\chi$ is the Geometric$(\pi)$ distribution, the probability of draw, $d_{1,1}$, is $0$ for \emph{all} values of $p_{-1}$, $p_{0}$ and $p_{1}$, and no phenomenon of phase transition is observed in this case.
\end{remark}

\subsubsection{When the offspring distribution is supported on $\{0,d\}$, for $d \in \mathbb{N}$ and $d \geqslant 2$}\label{subsec:0,d} Let $\chi(0) = 1-\pi$ and $\chi(d) = \pi$, for some $0 < \pi < 1$ and some $d \in \mathbb{N}$ with $d \geqslant 2$. In this case, the pgf is given by $G(x) = (1-\pi) + \pi x^{d}$. This yields $h'(x)=p_{0}^{2}\pi^{2}d^{2}\left[\pi(1-p_{1}) + p_{-1}(1-\pi) - p_{0}\pi x^{d}\right]^{d-1}x^{d-1}$, so that $h$ is strictly increasing for $x \in [0,x_{0})$ and strictly decreasing for $x \in (x_{0},1]$, where  $x_{0}=\left\{\pi(1-p_{1}) + p_{-1}(1-\pi)\right\}^{1/d}\left(p_{0} \pi (d+1)\right)^{-1/d}$. The maximum of $h'(x)$ is attained at $x=x_{0}$, and this maximum value equals $p_{0}^{(d+1)/d} \pi^{(d+1)/d} d^{d+1} \left\{\pi(1-p_{1}) + p_{-1}(1-\pi)\right\}^{(d+1)(d-1)/d} (d+1)^{-(d+1)(d-1)/d}$. The derivative of $h$ is bounded above by $1$ throughout the interval $[0,1]$ if and only if $h'(x_{0}) \leqslant 1$, which is equivalent to the inequality
\begin{equation}\label{eq:lem_0_d_2}
p_{0} \pi \left\{\pi(1-p_{1}) + p_{-1}(1-\pi)\right\}^{d-1} \leqslant (d+1)^{d-1}d^{-d}.
\end{equation} 
Lemma~\ref{lem:max_derivative_leq_1} then allows us to conclude that $h$ has a unique fixed point, and consequently, the probability of draw, $d_{1,1}$, equals $0$, whenever \eqref{eq:lem_0_d_2} holds.

As in \S\ref{subsec:binomial}, \S\ref{subsec:Poisson} and \S\ref{subsec:neg_bin}, we show that $w' \leqslant x_{0}$ for all values of $p_{-1}$, $p_{0}$ and $p_{1}$, by showing that $h(x_{0}) \leqslant x_{0}$. When $p_{0} > 0$, judicious algebraic manipulations lead to this inequality being equivalent to 
\begin{equation}
\gamma^{-1/d}p_{0}^{-1/d}\pi^{-1/d}(d+1)^{-1/d} + p_{0}\pi d^{d}\gamma(d+1)^{-d} > 1, \text{ where } \gamma=\left\{\pi(1-p_{1}) + p_{-1}(1-\pi)\right\}^{d-1}.\label{objective_0_d}
\end{equation}
Setting $f_{2}(x)=x^{-1/d}p_{0}^{-1/d}\pi^{-1/d}(d+1)^{-1/d} + p_{0}\pi d^{d}x(d+1)^{-d}$, we see that $f_{2}$ is strictly decreasing for $0 < x < p_{0}^{-1} \pi^{-1} d^{-d} (d+1)^{d-1}$, and strictly increasing for $p_{0}^{-1} \pi^{-1} d^{-d} (d+1)^{d-1} < x \leqslant 1$, so that the minimum value of $f_{2}(x)$ is attained at $x=p_{0}^{-1} \pi^{-1} d^{-d} (d+1)^{d-1}$, and this minimum value turns out to be equal to $1$, proving \eqref{objective_0_d} to be indeed true. 

Finally, we show that $x_{0} < \alpha$, where we recall that $\alpha$ is the unique fixed point of $g$ defined in \eqref {g_h_defn_kappa=2}, whenever \eqref{eq:lem_0_d_2} fails to hold. We show that the negation of the inequality in \eqref{eq:lem_0_d_2} is equivalent to the inequality that $g(x_{0}) > x_{0}$, so that, upon using the observation that the curve $y=g(x)$ lies above the line $y=x$ when $x \in [0,\alpha)$, and beneath $y=x$ when $x \in (\alpha,1]$), we conclude that, indeed, $x_{0} < \alpha$ when \eqref{eq:lem_0_d_2} is false. This leads to the conclusion that when \eqref{eq:lem_0_d_2} fails to hold, we have $w' \leqslant x_{0} < \alpha$, proving the existence of two distinct fixed points, $w'$ and $\alpha$, of the function $h$ in $[0,1]$.

Our final conclusion from the previous few paragraphs of argument is that the probability of draw, $d_{1,1}$, equals $0$ if and only if the inequality in \eqref{eq:lem_0_d_2} holds. 

This brings us to the end of the proof of Theorem~\ref{thm:kappa=2}.

\subsection{Proof of Theorems~\ref{thm:kappa=3} and \ref{thm:kappa=3_special}}\label{sec:kappa=3_proofs}
Recall, from \S\ref{subsec:notations_definitions}, that $w_{i,0}=w_{3,j}=1$ for each $i \in \{1,2,3\}$ and each $j \in \{0,1,2\}$, and likewise, $\ell_{i,3}=\ell_{0,j}=1$ for each $i \in \{0,1,2\}$ and each $j \in \{1,2,3\}$, when $\kappa=3$. Therefore, \eqref{recurrence_3} and \eqref{recurrence_4'} together yield, along with the same argument that is used to prove \eqref{smallest_largest_fixed_point} of Theorem~\ref{thm:1}, that $\left(\ell_{1,1},\ell_{1,2},\ell_{2,1},\ell_{2,2}\right)$ and $\left(1-w_{1,1},1-w_{1,2},1-w_{2,1},1-w_{2,2}\right)$ are both fixed points of the function $F: [0,1]^{4} \rightarrow [0,1]^{4}$, with $F(\mathbf{x})=\left(F_{1,1}(\mathbf{x}),F_{1,2}(\mathbf{x}),F_{2,1}(\mathbf{x}),F_{2,2}(\mathbf{x})\right)$ (where $\mathbf{x}=\left(x_{1,1},x_{1,2},x_{2,1},x_{2,2}\right)$), defined as follows:
\begin{align}
{}&F_{1,1}(\mathbf{x})=G\{1-p_{1}G(1-p_{1}x_{2,2}-p_{0}x_{2,1})-p_{0}G(1-p_{1}x_{1,2}-p_{0}x_{1,1})\},\label{ell_{1,1}}\\
{}&F_{1,2}(\mathbf{x})=G\{1-p_{1}G(p_{0}+p_{-1}-p_{0}x_{2,2}-p_{-1}x_{2,1})-p_{0}G(p_{0}+p_{-1}-p_{0}x_{1,2}-p_{-1}x_{1,1})\},\label{ell_{1,2}}\\
{}&F_{2,1}(\mathbf{x})=G\{p_{0}+p_{-1}-p_{0}G(1-p_{1}x_{2,2}-p_{0}x_{2,1})-p_{-1}G(1-p_{1}x_{1,2}-p_{0}x_{1,1})\},\label{ell_{2,1}}\\
{}&F_{2,2}(\mathbf{x})=G\{p_{0}+p_{-1}-p_{0}G(p_{0}+p_{-1}-p_{0}x_{2,2}-p_{-1}x_{2,1})-p_{-1}G(p_{0}+p_{-1}-p_{0}x_{1,2}-p_{-1}x_{1,1})\},\label{ell_{2,2}}
\end{align}
and if $\mathbf{y}=\left(y_{1,1},y_{1,2},y_{2,1},y_{2,2}\right)$ is any fixed point of $F$, then $\ell_{i,j} \leqslant y_{i,j} \leqslant 1-w_{i,j}$ for all $(i,j) \in \{1,2\}^{2}$. Consequently, $d_{i,j}=0$ for all $(i,j) \in \{1,2\}^{2}$ if and only if $F$ has a unique fixed point in $[0,1]^{4}$. As discussed after the statement of Theorem~\ref{thm:kappa=3}, few tools are at our disposal for finding necessary and sufficient condition(s), in terms of $p_{-1}$, $p_{0}$ and $p_{1}$ (and the parameter(s) underlying the offspring distribution under consideration), under which we may assert the uniqueness of the fixed point of $F$ in $[0,1]^{4}$. A sufficient condition is obtained, via an application of the Banach Fixed Point Theorem, when $F$ is Lipschitz with a Lipschitz constant that is strictly less than $1$. In what follows, this is what we attempt to accomplish.

\subsubsection{Suitable lower bounds on the various probabilities}\label{subsec:lower_bounds} We begin by finding lower bounds for $\ell_{i,j}$ and $w_{i,j}$, for all $(i,j) \in \{1,2\}^{2}$. Immediately, we observe that $\min\left\{\ell_{1,1},\ell_{1,2}\right\} \geqslant G\left(p_{-1}\right)$, since $P_{1}$ loses each of $\mathcal{G}(1,1,3,T)[\phi]$ and $\mathcal{G}(1,2,3,T)[\phi]$ if $\omega_{T}\left(\phi,u\right)=-1$ for every child $u$ of the root $\phi$ of the GW tree $T$. Finding lower bounds on $\ell_{2,1}$ and $\ell_{2,2}$ is a little less straightforward.

A lower bound on $\ell_{2,1}$ is provided by the probability of the event that, for every child $u$ of $\phi$,
\begin{enumerate*}
\item $\omega_{T}(\phi,u)=-1$, and
\item there exists a child $v$ of $u$ such that $\omega_{T}(u,v)\neq -1$ and $\omega_{T}(v,w)=-1$ for every child $w$ of $v$.
\end{enumerate*}
For a fixed vertex $u$, the probability that it has a child $v$ satisfying these conditions is given by $1-G\left[1-\left(1-p_{-1}\right)G\left(p_{-1}\right)\right]$. 
We thus obtain
\begin{align}
\ell_{2,1} \geqslant{}& G\left[p_{-1}\left\{1-G\left[1-\left(1-p_{-1}\right)G\left(p_{-1}\right)\right]\right\}\right].\label{ell_{2,1}_lower_bound}
\end{align}
A lower bound on $\ell_{2,2}$ is given by $\Prob\left[E_{1}\cup E_{2}\right]$, where $E_{1}$ is the event that, for every child $u$ of $\phi$, we have $\omega_{T}(\phi,u)=-1$, and there exists a child $v$ of $u$ with $\omega_{T}(v,w)=-1$ for every child $w$ of $v$, and $E_{2}$ is the event that, for every child $u$ of $\phi$, we have $\omega_{T}(\phi,u)\neq +1$ and there exists a child $v$ of $u$ with $\omega_{T}(u,v)=+1$. Similar to the derivation of \eqref{ell_{2,1}_lower_bound}, we obtain $\Prob\left[E_{1}\right]=G\left[p_{-1}\left\{1-G\left\{1-G\left(p_{-1}\right)\right\}\right\}\right]$. The probability that $u$ has at least one child $v$ with $\omega_{T}(u,v)=+1$ is given by $1-G\left(1-p_{1}\right)$, so that $\Prob\left[E_{2}\right]=G\left[\left(1-p_{1}\right)\left\{1-G\left(1-p_{1}\right)\right\}\right]$. We now have to find $\Prob\left[E_{1}\cap E_{2}\right]$. To this end, note that the probability of the event that $u$ has at least one child $v$ with $\omega_{T}(u,v)=+1$, but no child $v'$ such that $\omega_{T}(v',w)=-1$ for every child $w$ of $v'$, is given by
\begin{align}
{}&\sum_{m=0}^{\infty}\sum_{i=1}^{m}{m \choose i}\left\{\Prob\left[\omega_{T}(u,v)=+1 \text{ and there exists a child } w \text{ of } v \text{ with } \omega_{T}(v,w) \neq -1\right]\right\}^{i}\nonumber\\&\left\{\Prob\left[\omega_{T}(u,v)\in \{0,-1\} \text{ and there exists a child } w \text{ of } v \text{ with } \omega_{T}(v,w) \neq -1\right]\right\}^{m-i}\nonumber\\
={}&G\left[1-G\left(p_{-1}\right)\right]-G\left[\left(1-p_{1}\right)\left\{1-G\left(p_{-1}\right)\right\}\right].\nonumber
\end{align}
Thus, the probability that $u$ has a child $v$ with $\omega_{T}(u,v)=+1$ and a child $v'$ with $\omega_{T}(v',w)=-1$ for each child $w$ of $v'$, is given by $1-G\left(1-p_{1}\right)-G\left[1-G\left(p_{-1}\right)\right]+G\left[\left(1-p_{1}\right)\left\{1-G\left(p_{-1}\right)\right\}\right]$. This yields $\Prob\left[E_{1}\cap E_{2}\right]=G\left[p_{-1}\left[1-G\left(1-p_{1}\right)-G\left[1-G\left(p_{-1}\right)\right]+G\left[\left(1-p_{1}\right)\left\{1-G\left(p_{-1}\right)\right\}\right]\right]\right]$. We thus obtain
\begin{align}
\ell_{2,2} \geqslant{}& G\left[p_{-1}\left\{1-G\left\{1-G\left(p_{-1}\right)\right\}\right\}\right]+G\left[\left(1-p_{1}\right)\left\{1-G\left(1-p_{1}\right)\right\}\right]\nonumber\\&-G\left[p_{-1}\left[1-G\left(1-p_{1}\right)-G\left[1-G\left(p_{-1}\right)\right]+G\left[\left(1-p_{1}\right)\left\{1-G\left(p_{-1}\right)\right\}\right]\right]\right].\label{ell_{2,2}_lower_bound}
\end{align}

If $\phi$ has a child $u$ with $\omega_{T}(\phi,u)=+1$, then $P_{1}$ wins each of $\mathcal{G}(2,1,3,T)[\phi]$ and $\mathcal{G}(2,2,3,T)[\phi]$, so that $\min\left\{w_{2,1},w_{2,2}\right\} \geqslant 1-G\left(1-p_{1}\right)$. A lower bound for $w_{1,1}$ is provided by $\Prob\left[E_{3}\cup E_{4}\right]$, where $E_{3}$ is the event that $\phi$ has a child $u$ with $\omega_{T}(\phi,u) \neq -1$ and $\omega_{T}(u,v)=-1$ for every child $v$ of $u$, and $E_{4}$ is the event that $\phi$ has a child $u$ with $\omega_{T}(\phi,u)=1$, and for every child $v$ of $u$, there exists a child $w$ of $v$ such that $\omega_{T}(v,w)=1$. It is straightforward to see that $\Prob\left[E_{3}\right]=1-G\left[1-\left(1-p_{-1}\right)G\left(p_{-1}\right)\right]$. The task of finding $\Prob\left[E_{4}\cap E_{3}^{c}\right]$ is a relatively complicated one. For a fixed vertex $u$, let $A_{u}$ be the event that for every child $v$ of $u$, there exists a child $w$ of $v$ such that $\omega_{T}(v,w)=1$, so that $\Prob\left[A_{u}\right]=G\left\{1-G\left(1-p_{1}\right)\right\}$, and let $B_{u}$ be the event that there exists a child $v'$ of $u$ with $\omega_{T}(u,v') \neq -1$, so that $\Prob\left[B_{u}\right]=1-G\left(p_{-1}\right)$. We note that 
\begin{align}
\Prob\left[A_{u} \cap B_{u}\right]={}&\Prob\left[A_{u}\right]-\Prob\left[A_{u}\cap B_{u}^{c}\right]=G\left\{1-G\left(1-p_{1}\right)\right\}-G\left[p_{-1}\left\{1-G\left(1-p_{1}\right)\right\}\right],\nonumber
\end{align} 
so that $\Prob\left[E_{4} \cap E_{3}^{c}\right]$ equals
\begin{align}
{}&\sum_{m=0}^{\infty}\sum_{(i,j,k,m-i-j-k) \in \mathbb{N} \times \mathbb{N}_{0}^{3}}{m \choose i,j,k,m-i-j-k}\left\{p_{1}\Prob\left[A_{u} \cap B_{u}\right]\right\}^{i}\left\{p_{1}\Prob\left[A_{u}^{c}\cap B_{u}\right]\right\}^{j}\left\{p_{0}\Prob\left[B_{u}\right]\right\}^{k}p_{-1}^{m-i-j-k}\chi(m)\nonumber\\
={}&G\left[\left(1-p_{-1}\right)\left\{1-G\left(p_{-1}\right)\right\}+p_{-1}\right]-G[\left(1-p_{-1}\right)\left\{1-G\left(p_{-1}\right)\right\}+p_{-1}-p_{1}G\left\{1-G\left(1-p_{1}\right)\right\}\nonumber\\&+p_{1}G\left[p_{-1}\left\{1-G\left(1-p_{1}\right)\right\}\right]].\nonumber
\end{align}
Combining the above, we obtain
\begin{align}
w_{1,1} \geqslant{}& 1-G\left[1-\left(1-p_{-1}\right)G\left(p_{-1}\right)\right]+G\left[\left(1-p_{-1}\right)\left\{1-G\left(p_{-1}\right)\right\}+p_{-1}\right]-G[\left(1-p_{-1}\right)\left\{1-G\left(p_{-1}\right)\right\}\nonumber\\&+p_{-1}-p_{1}G\left\{1-G\left(1-p_{1}\right)\right\}+p_{1}G\left[p_{-1}\left\{1-G\left(1-p_{1}\right)\right\}\right]].\label{w_{1,1}_lower_bound}
\end{align}
Finally, $w_{1,2}$ is bounded below by the probability of the event that $\phi$ has at least one child $u$ with $\omega_{T}(\phi,u)=+1$, such that for every child $v$ of $u$, $\omega_{T}(u,v) \neq +1$ and there exists a child $w$ of $v$ with $\omega_{T}(v,w)=+1$. For a fixed vertex $u$, the probability that, for every child $v$ of $u$, $\omega_{T}(u,v)\neq +1$ and $v$ has a child $w$ with $\omega_{T}(v,w)=+1$, is given by $G\left[\left(1-p_{1}\right)\left\{1-G\left(1-p_{1}\right)\right\}\right]$. Consequently,
\begin{align}
w_{1,2}\geqslant 1-G\left[1-p_{1}G\left[\left(1-p_{1}\right)\left\{1-G\left(1-p_{1}\right)\right\}\right]\right].\label{w_{1,2}_lower_bound}
\end{align}

The bounds derived above, such as in \eqref{ell_{2,1}_lower_bound}, \eqref{ell_{2,2}_lower_bound}, \eqref{w_{1,1}_lower_bound} and \eqref{w_{1,2}_lower_bound}, after a comparison with the notation introduced in the statement of Theorem~\ref{thm:kappa=3}, reveal that $\ell_{i,j} \geqslant A_{i,j}$ and $\left(1-w_{i,j}\right) \leqslant B_{i,j}$ for each $(i,j) \in \{1,2\}^{2}$. From the discussion right after \eqref{ell_{2,2}}, it is enough for us to consider $F$ to be defined on $\left[A_{1,1},B_{1,1}\right]\times\left[A_{1,2},B_{1,2}\right]\times\left[A_{2,1},B_{2,1}\right]\times\left[A_{2,2},B_{2,2}\right]$. This is what is utilized in \S\ref{subsec:F_contraction}.

\subsubsection{Finding regime(s) where the function $F$ is a contraction}\label{subsec:F_contraction}
In what follows, we change the notation introduced just before Theorem~\ref{thm:kappa=3} slightly, and let $\partial_{i,j}F_{i',j'}(\mathbf{x})$, for any $\mathbf{x}=(x_{1,1},x_{1,2},x_{2,1},x_{2,2})\in [0,1]^{4}$, denote the partial derivative of $F_{i',j'}(\mathbf{x})$ with respect to $x_{i,j}$, for each $i,j,i',j' \in \{1,2\}$.

We consider any $\mathbf{x}=\left(x_{1,1},x_{1,2},x_{2,1},x_{2,2}\right)$ and $\mathbf{y}=\left(y_{1,1},y_{1,2},y_{2,1},y_{2,2}\right)$, where each of $x_{i,j}$ and $y_{i,j}$ belongs to the interval $\left[A_{i,j},B_{i,j}\right]$, for each $(i,j) \in \{1,2\}^{2}$.  Applying the mean value theorem to \eqref{ell_{1,1}}, we know that there exists $\pmb{\xi}=\left(\xi_{1,1},\xi_{1,2},\xi_{2,1},\xi_{2,2}\right)$, lying on the line connecting $\mathbf{x}$ and $\mathbf{y}$, such that $F_{1,1}(\mathbf{x})-F_{1,1}(\mathbf{y})=\nabla F_{1,1}(\pmb{\xi}) \cdot (\mathbf{x}-\mathbf{y})$. Since $\xi_{i,j} \in \left[A_{i,j},B_{i,j}\right]$ for each $(i,j) \in \{1,2\}^{2}$ as well, and both $G$ and $G'$ are increasing, we have
\begin{align}
\partial_{1,1}F_{1,1}(\pmb{\xi})={}&p_{0}^{2}G'(1-p_{1}\xi_{1,2}-p_{0}\xi_{1,1})G'\{1-p_{1}G(1-p_{1}\xi_{2,2}-p_{0}\xi_{2,1})-p_{0}G(1-p_{1}\xi_{1,2}-p_{0}\xi_{1,1})\}\nonumber\\
\leqslant{}&p_{0}^{2}G'(1-p_{1}A_{1,2}-p_{0}A_{1,1})G'\{1-p_{1}G(1-p_{1}B_{2,2}-p_{0}\xi_{2,1})-p_{0}G(1-p_{1}B_{1,2}-p_{0}B_{1,1})\}.\nonumber
\end{align}
Applying this and similarly derived upper bounds on $\partial_{1,2}F_{1,1}(\pmb{\xi})$, $\partial_{2,1}F_{1,1}(\pmb{\xi})$ and $\partial_{2,2}F_{1,1}(\pmb{\xi})$, we find that $\left|F_{1,1}(\mathbf{x})-F_{1,1}(\mathbf{y})\right|\leqslant \sum_{(i,j)\in\{1,2\}^{2}}\partial_{i,j}F_{1,1}(\pmb{\xi})\left|x_{i,j}-y_{i,j}\right|$ is bounded above by
\begin{align}
{}&\big[p_{0}\left\{p_{0}\left|x_{1,1}-y_{1,1}\right|+p_{1}\left|x_{1,2}-y_{1,2}\right|\right\}G'\left(1-p_{1}A_{1,2}-p_{0}A_{1,1}\right)+p_{1}\left\{p_{0}\left|x_{2,1}-y_{2,1}\right|+p_{1}\left|x_{2,2}-y_{2,2}\right|\right\}\nonumber\\&G'\left(1-p_{1}A_{2,2}-p_{0}A_{2,1}\right)\big]G'\left\{1-p_{1}G\left(1-p_{1}B_{2,2}-p_{0}B_{2,1}\right)-p_{0}G\left(1-p_{1}B_{1,2}-p_{0}B_{1,1}\right)\right\}.\nonumber
\end{align}
Each of the remaining differences, i.e.\ $\left|F_{1,2}(\mathbf{x})-F_{1,2}(\mathbf{y})\right|$, $\left|F_{2,1}(\mathbf{x})-F_{2,1}(\mathbf{y})\right|$ and $\left|F_{2,2}(\mathbf{x})-F_{2,2}(\mathbf{y})\right|$, can be bounded above in an analogous manner, so that we obtain 
\begin{align}
\sum_{(i,j) \in \{1,2\}^{2}}\left|F_{i,j}(\mathbf{x})-F_{i,j}(\mathbf{y})\right| \leqslant \sum_{(i,j) \in \{1,2\}^{2}}E_{i,j}\left|x_{i,j}-y_{i,j}\right|,\nonumber
\end{align}
where $E_{i,j}$, for each $(i,j) \in \{1,2\}^{2}$, is as defined in the statement of Theorem~\ref{thm:kappa=3}.

\subsubsection{Proof of part \eqref{kappa=3_part_1} of Theorem~\ref{thm:kappa=3_special}} We now come to the proof of part \eqref{kappa=3_part_1} of Theorem~\ref{thm:kappa=3}. Since the underlying tree is the rooted $2$-regular tree, denoted henceforth by $T_{2}$, the pgf of the offspring distribution is simply $G(x)=x^{2}$. When, in addition, $p_{-1}$, $p_{0}$ and $p_{1}$ satisfy \eqref{special_form}, straightforward algebraic manipulations reveal that $\ell_{2,2}=\alpha^{-2}\left(\sqrt{\ell_{1,2}}-p_{-1}\right)^{2}$ and $\ell_{2,1}=\alpha^{-2}\left(\sqrt{\ell_{1,1}}-p_{-1}\right)^{2}$. Incorporating these two relations into \eqref{ell_{1,1}} and \eqref{ell_{1,2}}, we see that $\left(\sqrt{\ell_{1,1}},\sqrt{\ell_{1,2}}\right)$ is a fixed point of the function $f: [0,1]^{2}\rightarrow [0,1]^{2}$, defined as $f(y_{1},y_{2}) = \left(f_{1}(y_{1},y_{2}), f_{2}(y_{1},y_{2})\right)$, where
\begin{align}
{}& f_{1}(y_{1},y_{2})=1-\frac{\alpha^{2}}{\alpha^{2}+\alpha+1}\left\{1-\frac{1}{\alpha^{2}+\alpha+1}\left(y_{2}-\frac{1}{\alpha^{2}+\alpha+1}\right)^{2} - \frac{1}{\alpha\left(\alpha^{2}+\alpha+1\right)}\left(y_{1}-\frac{1}{\alpha^{2}+\alpha+1}\right)^{2}\right\}^{2}\nonumber\\&-\frac{\alpha}{\alpha^{2}+\alpha+1}\left\{1-\frac{\alpha^{2}y_{2}^{2}}{\alpha^{2}+\alpha+1}-\frac{\alpha y_{1}^{2}}{\alpha^{2}+\alpha+1}\right\}^{2},\nonumber\\
{}&f_{2}(y_{1},y_{2})=1-\frac{\alpha^{2}}{\alpha^{2}+\alpha+1}\Bigg\{1-\frac{\alpha^{2}}{\alpha^{2}+\alpha+1}-\frac{1}{\alpha\left(\alpha^{2}+\alpha+1\right)}\left(y_{2}-\frac{1}{\alpha^{2}+\alpha+1}\right)^{2}\nonumber\\&-\frac{1}{\alpha^{2}\left(\alpha^{2}+\alpha+1\right)}\left(y_{1}-\frac{1}{\alpha^{2}+\alpha+1}\right)^{2}\Bigg\}^{2}-\frac{\alpha}{\alpha^{2}+\alpha+1}\left(1-\frac{\alpha^{2}}{\alpha^{2}+\alpha+1}-\frac{\alpha y_{2}^{2}}{\alpha^{2}+\alpha+1}-\frac{y_{1}^{2}}{\alpha^{2}+\alpha+1}\right)^{2}.\nonumber
\end{align}
In a similar manner, we deduce that $\left(\sqrt{1-w_{1,1}},\sqrt{1-w_{1,2}}\right)$ is yet another fixed point of the function $f$, and in fact, arguing the same way as in the proof of \eqref{smallest_largest_fixed_point} of Theorem~\ref{thm:1}, we can prove that if $(x_{1},x_{2})$ is any fixed point of $f$ with $x_{1}, x_{2} \in [0,1]$, then $\left(\sqrt{\ell_{1,1}},\sqrt{\ell_{1,2}}\right) \preceq (x_{1},x_{2}) \preceq \left(\sqrt{1-w_{1,1}},\sqrt{1-w_{1,2}}\right)$. This tells us that $d_{1,1}=d_{1,2}=0$ (and by Theorem~\ref{thm:2}, this is equivalent to $d_{2,1}=d_{2,2}=0$, since each of $p_{-1}$, $p_{0}$ and $p_{1}$ is strictly positive when they satisfy \eqref{special_form}) if and only if $f$ has a unique fixed point in $[0,1]^{2}$. 

Our goal is to now find the set of values of $\alpha$ (where $\alpha$ is as in \eqref{special_form}) for which the function $f$ turns out to be Lipschitz, with the Lipschitz constant being strictly less than $1$, on the domain $[p_{-1},1]^{2}$. Algebraic manipulations reveal that, for all $y_{1}, y_{2}, z_{1}, z_{2} \in [p_{-1},1]$,
\begin{align}
{}&\left|f_{1}(y_{1},y_{2})-f_{1}(z_{1},z_{2})\right| \leqslant \left|y_{2}-z_{2}\right|\Bigg[\frac{\alpha^{2}}{\left(\alpha^{2}+\alpha+1\right)^{2}}\left(y_{2}+z_{2}-\frac{2}{\alpha^{2}+\alpha+1}\right)(2-A)+\frac{\alpha^{3}}{\left(\alpha^{2}+\alpha+1\right)^{2}}(y_{2}+z_{2})(2-B)\Bigg] \nonumber\\&+ \left|y_{1}-z_{1}\right|\Bigg[\frac{\alpha}{\left(\alpha^{2}+\alpha+1\right)^{2}}\left(y_{1}+z_{1}-\frac{2}{\alpha^{2}+\alpha+1}\right)(2-A)+\frac{\alpha^{2}}{\left(\alpha^{2}+\alpha+1\right)^{2}}(y_{1}+z_{1})(2-B)\Bigg]\label{intermediate_1}\\
{}&\left|f_{2}(y_{1},y_{2})-f_{2}(z_{1},z_{2})\right|\leqslant \left|y_{2}-z_{2}\right|\Bigg[\frac{\alpha}{\left(\alpha^{2}+\alpha+1\right)^{2}}\left(y_{2}+z_{2}-\frac{2}{\alpha^{2}+\alpha+1}\right)\left(2-\frac{2\alpha^{2}}{\alpha^{2}+\alpha+1}-\frac{A}{\alpha}\right)\nonumber\\&+\frac{\alpha^{2}}{\left(\alpha^{2}+\alpha+1\right)^{2}}(y_{2}+z_{2})\left(2-\frac{2\alpha^{2}}{\alpha^{2}+\alpha+1}-\frac{B}{\alpha}\right)\Bigg]+\left|y_{1}-z_{1}\right|\Bigg[\frac{1}{\left(\alpha^{2}+\alpha+1\right)^{2}}\left(y_{1}+z_{1}-\frac{2}{\alpha^{2}+\alpha+1}\right)\nonumber\\&\left(2-\frac{2\alpha^{2}}{\alpha^{2}+\alpha+1}-\frac{A}{\alpha}\right)+\frac{\alpha}{\left(\alpha^{2}+\alpha+1\right)^{2}}(y_{1}+z_{1})\left(2-\frac{2\alpha^{2}}{\alpha^{2}+\alpha+1}-\frac{B}{\alpha}\right)\Bigg],\label{intermediate_2}
\end{align} 
where 
\begin{align}
{}&A=\frac{1}{\alpha^{2}+\alpha+1}\left(y_{2}-\frac{1}{\alpha^{2}+\alpha+1}\right)^{2}+\frac{1}{\alpha^{2}+\alpha+1}\left(z_{2}-\frac{1}{\alpha^{2}+\alpha+1}\right)^{2}+\frac{1}{\alpha\left(\alpha^{2}+\alpha+1\right)}\left(y_{1}-\frac{1}{\alpha^{2}+\alpha+1}\right)^{2}\nonumber\\&+\frac{1}{\alpha\left(\alpha^{2}+\alpha+1\right)}\left(z_{1}-\frac{1}{\alpha^{2}+\alpha+1}\right)^{2} \text{ and } B=\frac{\alpha^{2}y_{2}^{2}}{\alpha^{2}+\alpha+1}+\frac{\alpha^{2}z_{2}^{2}}{\alpha^{2}+\alpha+1}+\frac{\alpha y_{1}^{2}}{\alpha^{2}+\alpha+1}+\frac{\alpha z_{1}^{2}}{\alpha^{2}+\alpha+1}.\nonumber
\end{align}
Adding \eqref{intermediate_1} and \eqref{intermediate_2}, using the simple bound of $A \geqslant 0$, and the inequality that $B \geqslant 2(\alpha^{2}+\alpha)\left(\alpha^{2}+\alpha+1\right)^{-3}$ (which is true since each of $y_{1}$, $y_{2}$, $z_{1}$ and $z_{2}$ is bounded below by $p_{-1}$), we obtain:
\begin{align}
{}&\left|f_{1}(y_{1},y_{2})-f_{1}(z_{1},z_{2})\right|+\left|f_{2}(y_{1},y_{2})-f_{2}(z_{1},z_{2})\right| \leqslant 4\Bigg[\frac{\alpha^{3}(\alpha+1)}{\left(\alpha^{2}+\alpha+1\right)^{3}}+\frac{\alpha^{3}}{\left(\alpha^{2}+\alpha+1\right)^{2}}\left\{1-\frac{\alpha^{2}+\alpha}{\left(\alpha^{2}+\alpha+1\right)^{3}}\right\}\nonumber\\&+\frac{\alpha^{2}(\alpha+1)^{2}}{\left(\alpha^{2}+\alpha+1\right)^{4}} + \frac{\alpha^{2}(\alpha+1)}{\left(\alpha^{2}+\alpha+1\right)^{3}}\left\{1-\frac{1}{\left(\alpha^{2}+\alpha+1\right)^{2}}\right\}\Bigg]\left|y_{2}-z_{2}\right|+4\Bigg[\frac{\alpha^{2}(\alpha+1)}{\left(\alpha^{2}+\alpha+1\right)^{3}}+\frac{\alpha^{2}}{\left(\alpha^{2}+\alpha+1\right)^{2}}\nonumber\\&\left\{1-\frac{\alpha^{2}+\alpha}{\left(\alpha^{2}+\alpha+1\right)^{3}}\right\}+\frac{\alpha(\alpha+1)^{2}}{\left(\alpha^{2}+\alpha+1\right)^{4}}+\frac{\alpha(\alpha+1)}{\left(\alpha^{2}+\alpha+1\right)^{3}}\left\{1-\frac{1}{\left(\alpha^{2}+\alpha+1\right)^{2}}\right\}\Bigg]\left|y_{1}-z_{1}\right|,\label{Lipschitz_ineq}
\end{align}
and we see that the coefficient of $\left|y_{2}-z_{2}\right|$ in the right side of \eqref{Lipschitz_ineq} is strictly less than $1$ whenever $0 \leqslant \alpha < 0.727384$ or $\alpha > 2.57162$, whereas the coefficient of $\left|y_{1}-z_{1}\right|$ in the right side of \eqref{Lipschitz_ineq} is strictly less than $1$ whenever $0 \leqslant \alpha < 0.242915$ or $\alpha > 1.21108$. We thus conclude that as long as $0 \leqslant \alpha < 0.242915$ or $\alpha > 2.57162$, the function $f$ is Lipschitz with a Lipschitz constant that is strictly less than $1$, implying that $f$ has a unique fixed point in $[p_{-1},1]^{2}$, and consequently, we have $d_{i,j}=0$ for each $(i,j) \in \{1,2\}^{2}$.

\subsubsection{Proof of part \eqref{kappa=3_part_2} of Theorem~\ref{thm:kappa=3_special}} We now come to the proof of part \eqref{kappa=3_part_2} of Theorem~\ref{thm:kappa=3}. Since $p_{0}=0$, \eqref{ell_{1,1}}, \eqref{ell_{1,2}}, \eqref{ell_{2,1}} and \eqref{ell_{2,2}} together yield:
\begin{enumerate}
\item $\ell_{1,1}$ is the smallest non-negative fixed point and $\left(1-w_{1,1}\right)$ the largest non-negative fixed point, bounded above by $1$, of the function $b^{(2)}\circ a^{(2)}$,
\item whereas $\ell_{1,2}$ is the smallest non-negative fixed point and $\left(1-w_{1,2}\right)$ the largest non-negative fixed point, bounded above by $1$, of the function $b\circ a^{(2)} \circ b$,
\end{enumerate}
where the functions $a$ and $b$ are as defined in the statement of part \eqref{kappa=3_part_2} of Theorem~\ref{thm:kappa=3_special}. Consequently, $d_{1,1}=0$ if and only if the function $b^{(2)} \circ a^{(2)}$ has a unique fixed point in $[0,1]$, and $d_{1,2}=0$ if and only if the function $b\circ a^{(2)} \circ b$ has a unique fixed point in $[0,1]$. From part \eqref{draw_4} of Theorem~\ref{thm:2}, we see that $d_{1,1}=0$ if and only if $d_{2,2}=0$ (likewise, $d_{1,2}=0$ if and only if $d_{2,1}=0$), so that it suffices for us to characterize those values of $p_{-1}$ (note that $p_{1}=1-p_{-1}$ here, so we are really only dealing with a single parameter here) for which $d_{1,1}=0$ (respectively, $d_{1,2}=0$).

As in the case of \eqref{kappa=3_part_1}, the goal here is to find those values of $p_{-1}$ for which the function $b^{(2)}\circ a^{(2)}$ (respectively, the function $b\circ a^{(2)}\circ b$) is Lipschitz on $[0,1]$ with a Lipschitz constant that is strictly less than $1$. As in \eqref{kappa=3_part_1}, an immediate lower bound for each of $\ell_{1,1}$ and $\ell_{1,2}$ is obtained when we consider the scenario where $\omega(\phi,u)=-1$ for every child $u$ of the root $\phi$ of $T_{\chi}$, with the probability of this event being $\sum_{m=0}^{\infty}p_{-1}^{m}\chi(m)=G\left(p_{-1}\right)$. It thus suffices for us to focus on the domain $\left[G\left(p_{-1}\right),1\right]$ for each of the functions $b$ and $a$. We then deduce, using the mean value theorem and the monotonically increasing nature of the derivative $G'$ of the pgf, for any $G(p) \leqslant x < y \leqslant 1$, the inequalities:
\begin{align}
{}&a(x)-a(y) \leqslant p_{-1}G'\left\{p_{-1}-p_{-1}G\left(p_{-1}\right)\right\}(y-x),\label{g_Lipschitz}\\
{}&b(x)-b(y) \leqslant \left(1-p_{-1}\right)G'\left\{1-\left(1-p_{-1}\right)G\left(p_{-1}\right)\right\}(y-x).\label{h_Lipschitz}
\end{align}
This allows us to write, for all $G(p_{-1}) \leqslant x < y \leqslant 1$, the inequality:
\begin{multline}
\left|b^{(2)}\left(a^{(2)}(x)\right)-b^{(2)}\left(a^{(2)}(y)\right)\right|\leqslant \left[\left(1-p_{-1}\right)G'\left\{1-\left(1-p_{-1}\right)G\left(p_{-1}\right)\right\}\right]^{2}\\ \left[p_{-1}G'\left\{p_{-1}-p_{-1}G\left(p_{-1}\right)\right\}\right]^{2}(y-x),\nonumber
\end{multline}
and this same upper bound also applies to the difference $\left|b\left(a^{(2)}\left(b(x)\right)\right)-b\left(a^{(2)}\left(b(y)\right)\right)\right|$. Therefore, as long as we have
\begin{align}\label{Lipschitz_cond}
p_{-1}\left(1-p_{-1}\right)G'\left\{1-\left(1-p_{-1}\right)G\left(p_{-1}\right)\right\}G'\left\{p_{-1}-p_{-1}G\left(p_{-1}\right)\right\} < 1,
\end{align}
each of the functions $b^{(2)}\circ a^{(2)}$ and $b\circ a^{(2)} \circ h$ is Lipschitz on $\left[G\left(p_{-1}\right),1\right]$, with a Lipschitz constant that is strictly less than $1$, implying that each of them has a unique fixed point in $\left[G\left(p_{-1}\right),1\right]$ and making each of $d_{1,1}$ and $d_{1,2}$ equal $0$. This brings us to the end of the proof of part \eqref{kappa=3_part_2} of Theorem~\ref{thm:kappa=3}.

\subsection{Proof of Corollary~\ref{cor_1}}
\begin{proof}
	First consider the case when $\chi$ is the Dirac measure at 2 implying $G(x)=x^2$. We complete the proof by showing that
	\begin{equation}\label{eq_cor_kappa=3_part_2}
		p_{-1}\left(1-p_{-1}\right) G'\left[1-\left(1-p_{-1}\right)G\left(p_{-1}\right)\right]G'\left[p_{-1}\left\{1-G\left(p_{-1}\right)\right\}\right]<1
	\end{equation}
	 for all $p_{-1}\in [0,1]$ as stated in Theorem \ref{thm:kappa=3_special} (\ref{kappa=3_part_2}). If $p_{-1}\in \{0,1\}$, (\ref{eq_cor_kappa=3_part_2}) holds. Suppose $p_{-1}\in (0,1)$. Then, 
	\begin{align*}
	&	p_{-1}\left(1-p_{-1}\right) G'\left[1-\left(1-p_{-1}\right)G\left(p_{-1}\right)\right]G'\left[p_{-1}\left\{1-G\left(p_{-1}\right)\right\}\right] \\
	=\hspace{2mm}& 4p_{-1}\left(1-p_{-1}\right) \left[1-\left(1-p_{-1}\right)p_{-1}^2\right]\left[p_{-1}\left\{1-p_{-1}^2\right\}\right] \\
	<\hspace{2mm}& 4p_{-1}\left(1-p_{-1}\right) \leqslant 1. 
	\end{align*}
	This concludes that all the draw probabilities are $0$.
	
	Now, when $\chi$ is Poisson$(2)$, $G(x)=e^{2(x-1)}$. Again, if  $p_{-1}\in \{0,1\}$, (\ref{eq_cor_kappa=3_part_2}) holds. If $p_{-1}\in (0,1)$.  Then,
	\begin{align*}
		&	p_{-1}\left(1-p_{-1}\right) G'\left[1-\left(1-p_{-1}\right)G\left(p_{-1}\right)\right]G'\left[p_{-1}\left\{1-G\left(p_{-1}\right)\right\}\right] \\
		=\hspace{2mm}& 4p_{-1}\left(1-p_{-1}\right) e^{2\left[1-\left(1-p_{-1}\right)e^{2(p_{-1}-1)}-1\right]}e^{2\left[p_{-1}\left\{1-e^{2(p_{-1}-1)}\right\}-1\right]} \\
		=\hspace{2mm}& 4p_{-1}\left(1-p_{-1}\right) e^{-2\left[\left(1-p_{-1}\right)e^{-2(1-p_{-1})}\right]}e^{-2\left[1-p_{-1}\left\{1-e^{-2(1-p_{-1})}\right\}\right]} \\
		<\hspace{2mm}& 4p_{-1}\left(1-p_{-1}\right)\leqslant 1
	\end{align*}
	where the strict inequality holds as both $e^{-2\left[\left(1-p_{-1}\right)e^{-2(1-p_{-1})}\right]}$ and $e^{-2\left[1-p_{-1}\left\{1-e^{-2(1-p_{-1})}\right\}\right]}$ are less than $1$.
\end{proof}


\noindent \textbf{Data availability statement:} We hereby confirm that our manuscript has no associated data. 

\bibliography{Toll_tax_bib}

\begin{thebibliography}{35}
\providecommand{\natexlab}[1]{#1}
\providecommand{\url}[1]{\texttt{#1}}
\expandafter\ifx\csname urlstyle\endcsname\relax
  \providecommand{\doi}[1]{doi: #1}\else
  \providecommand{\doi}{doi: \begingroup \urlstyle{rm}\Url}\fi

\bibitem[Angel et~al.(2008)Angel, Goodman, Den~Hollander, and
  Slade]{angel2008invasion}
Omer Angel, Jesse Goodman, Frank Den~Hollander, and Gordon Slade.
\newblock Invasion percolation on regular trees.
\newblock 2008.

\bibitem[Athreya and Ney(2012)]{athreya2012branching}
Krishna~B Athreya and Peter~E Ney.
\newblock \emph{Branching processes}, volume 196.
\newblock Springer Science \& Business Media, 2012.

\bibitem[Basu et~al.(2016)Basu, Holroyd, Martin, and
  W{\"a}stlund]{basu2016trapping}
Riddhipratim Basu, Alexander~E Holroyd, James~B Martin, and Johan W{\"a}stlund.
\newblock Trapping games on random boards.
\newblock 2016.

\bibitem[Baur(2016)]{baur2016percolation}
Erich Baur.
\newblock Percolation on random recursive trees.
\newblock \emph{Random Structures \& Algorithms}, 48\penalty0 (4):\penalty0
  655--680, 2016.

\bibitem[Bhasin et~al.(2022)Bhasin, Karmakar, Podder, and
  Roy]{bhasin2022ergodicity}
Dhruv Bhasin, Sayar Karmakar, Moumanti Podder, and Souvik Roy.
\newblock Ergodicity of a generalized probabilistic cellular automaton with
  parity-based neighbourhoods.
\newblock \emph{arXiv preprint arXiv:2212.01753}, 2022.

\bibitem[Bhasin et~al.(2023)Bhasin, Karmakar, Podder, and Roy]{bhasin2022class}
Dhruv Bhasin, Sayar Karmakar, Moumanti Podder, and Souvik Roy.
\newblock On a class of pca with size-3 neighborhood and their applications in
  percolation games.
\newblock \emph{Electronic Journal of Probability}, 28:\penalty0 1--60, 2023.

\bibitem[Bhasin et~al.(2024)Bhasin, Karmakar, Podder, and Roy]{bhasin2024bond}
Dhruv Bhasin, Sayar Karmakar, Moumanti Podder, and Souvik Roy.
\newblock Bond percolation games on the $2 $-dimensional square lattice, and
  ergodicity of associated probabilistic cellular automata.
\newblock \emph{arXiv preprint arXiv:2405.12199}, 2024.

\bibitem[Broadbent and Hammersley(1957)]{broadbent1957percolation}
Simon~R Broadbent and John~M Hammersley.
\newblock Percolation processes: I. crystals and mazes.
\newblock In \emph{Mathematical proceedings of the Cambridge philosophical
  society}, volume~53, pages 629--641. Cambridge University Press, 1957.

\bibitem[Brouwer(1911)]{brouwer1911abbildung}
Luitzen Egbertus~Jan Brouwer.
\newblock {\"U}ber abbildung von mannigfaltigkeiten.
\newblock \emph{Mathematische annalen}, 71\penalty0 (1):\penalty0 97--115,
  1911.

\bibitem[Day and Falgas-Ravry(2021{\natexlab{a}})]{day2021maker}
A~Nicholas Day and Victor Falgas-Ravry.
\newblock Maker--breaker percolation games i: crossing grids.
\newblock \emph{Combinatorics, Probability and Computing}, 30\penalty0
  (2):\penalty0 200--227, 2021{\natexlab{a}}.

\bibitem[Day and Falgas-Ravry(2021{\natexlab{b}})]{day2021makerescaping}
A~Nicholas Day and Victor Falgas-Ravry.
\newblock Maker-breaker percolation games ii: Escaping to infinity.
\newblock \emph{Journal of Combinatorial Theory, Series B}, 151:\penalty0
  482--508, 2021{\natexlab{b}}.

\bibitem[Dvo{\v{r}}{\'a}k et~al.(2021)Dvo{\v{r}}{\'a}k, Mond, and
  Souza]{dvovrak2021maker}
Vojt{\v{e}}ch Dvo{\v{r}}{\'a}k, Adva Mond, and Victor Souza.
\newblock The maker-breaker percolation game on the square lattice.
\newblock \emph{arXiv preprint arXiv:2105.12864}, 2021.

\bibitem[Grimmett(1999)]{grimmett1999percolation}
Geoffrey Grimmett.
\newblock \emph{Percolation?}
\newblock Springer, 1999.

\bibitem[Haggstrom(1997)]{haggstrom1997infinite}
Olle Haggstrom.
\newblock Infinite clusters in dependent automorphism invariant percolation on
  trees.
\newblock \emph{The Annals of Probability}, pages 1423--1436, 1997.

\bibitem[Holroyd and Martin(2021)]{holroyd2021galton}
Alexander~E Holroyd and James~B Martin.
\newblock Galton--watson games.
\newblock \emph{Random Structures \& Algorithms}, 59\penalty0 (4):\penalty0
  495--521, 2021.

\bibitem[Holroyd et~al.(2019)Holroyd, Marcovici, and
  Martin]{holroyd2019percolation}
Alexander~E Holroyd, Ir{\`e}ne Marcovici, and James~B Martin.
\newblock Percolation games, probabilistic cellular automata, and the hard-core
  model.
\newblock \emph{Probability Theory and Related Fields}, 174:\penalty0
  1187--1217, 2019.

\bibitem[Kesten(1982)]{kesten1982percolation}
Harry Kesten.
\newblock \emph{Percolation theory for mathematicians}, volume~2.
\newblock Springer, 1982.

\bibitem[Khoshnevisan(2008)]{khoshnevisan2008dynamical}
Davar Khoshnevisan.
\newblock Dynamical percolation on general trees.
\newblock \emph{Probability theory and related fields}, 140\penalty0
  (1-2):\penalty0 169--193, 2008.

\bibitem[Lyons(1989)]{lyons1989ising}
Russell Lyons.
\newblock The ising model and percolation on trees and tree-like graphs.
\newblock \emph{Communications in Mathematical Physics}, 125:\penalty0
  337--353, 1989.

\bibitem[Lyons(1990)]{lyons1990random}
Russell Lyons.
\newblock Random walks and percolation on trees.
\newblock \emph{The annals of Probability}, 18\penalty0 (3):\penalty0 931--958,
  1990.

\bibitem[Lyons(1992)]{lyons1992random}
Russell Lyons.
\newblock Random walks, capacity and percolation on trees.
\newblock \emph{The Annals of Probability}, pages 2043--2088, 1992.

\bibitem[Lyons and Pemantle(1992)]{pemantle1992random}
Russell Lyons and Robin Pemantle.
\newblock Random walk in a random environment and first-passage percolation on
  trees.
\newblock \emph{The Annals of Probability}, pages 125--136, 1992.

\bibitem[M{\"u}ller and Stojakovi{\'c}(2014)]{muller2014threshold}
Tobias M{\"u}ller and Milo{\v{s}} Stojakovi{\'c}.
\newblock A threshold for the maker-breaker clique game.
\newblock \emph{Random structures \& algorithms}, 45\penalty0 (2):\penalty0
  318--341, 2014.

\bibitem[Nenadov et~al.(2016)Nenadov, Steger, and
  Stojakovi{\'c}]{nenadov2016threshold}
Rajko Nenadov, Angelika Steger, and Milo{\v{s}} Stojakovi{\'c}.
\newblock On the threshold for the maker-breaker h-game.
\newblock \emph{Random Structures \& Algorithms}, 49\penalty0 (3):\penalty0
  558--578, 2016.

\bibitem[Nowak and Krug(2013)]{nowak2013accessibility}
Stefan Nowak and Joachim Krug.
\newblock Accessibility percolation on n-trees.
\newblock \emph{Europhysics Letters}, 101\penalty0 (6):\penalty0 66004, 2013.

\bibitem[Olle et~al.(1997)Olle, Yuval, and Jeffrey]{olle1997dynamical}
H{\"a}ggstr{\"o}m Olle, Peres Yuval, and E~Steif Jeffrey.
\newblock Dynamical percolation.
\newblock In \emph{Annales de l'Institut Henri Poincare (B) Probability and
  Statistics}, volume~33, pages 497--528. Elsevier, 1997.

\bibitem[Peres and Steif(1998)]{peres1998number}
Yuval Peres and Jeffrey~E Steif.
\newblock The number of infinite clusters in dynamical percolation.
\newblock \emph{Probability theory and related fields}, 111:\penalty0 141--165,
  1998.

\bibitem[Perron(1907)]{perron1907theorie}
Oskar Perron.
\newblock Zur theorie der matrices.
\newblock \emph{Mathematische Annalen}, 64\penalty0 (2):\penalty0 248--263,
  1907.

\bibitem[Shante and Kirkpatrick(1971)]{shante1971introduction}
Vinod~KS Shante and Scott Kirkpatrick.
\newblock An introduction to percolation theory.
\newblock \emph{Advances in Physics}, 20\penalty0 (85):\penalty0 325--357,
  1971.

\bibitem[Stauffer and Aharony(2018)]{stauffer2018introduction}
Dietrich Stauffer and Ammon Aharony.
\newblock \emph{Introduction to percolation theory}.
\newblock CRC press, 2018.

\bibitem[Steif(2009)]{steif2009survey}
Jeffrey~E Steif.
\newblock A survey of dynamical percolation.
\newblock In \emph{Fractal geometry and stochastics IV}, pages 145--174.
  Springer, 2009.

\bibitem[Stojakovi{\'c} and Szab{\'o}(2005)]{stojakovic2005positional}
Milo{\v{s}} Stojakovi{\'c} and Tibor Szab{\'o}.
\newblock Positional games on random graphs.
\newblock \emph{Random Structures \& Algorithms}, 26\penalty0 (1-2):\penalty0
  204--223, 2005.

\bibitem[Toom(2001)]{toom2001contours}
Andr{\'e} Toom.
\newblock Contours, convex sets, and cellular automata.
\newblock \emph{Publica{\c{c}}oes Matem{\'a}ticas do IMPA.[IMPA Mathematical
  Publications]. Instituto de Matem{\'a}tica Pura e Aplicada (IMPA), Rio de
  Janeiro}, page 23o, 2001.

\bibitem[Wallwork(2022)]{wallwork2022maker}
Freddie Wallwork.
\newblock Maker-breaker-crossing-game on the triangular grid-graph.
\newblock \emph{arXiv preprint arXiv:2201.01348}, 2022.

\bibitem[Watson and Galton(1875)]{watson1875probability}
Henry~William Watson and Francis Galton.
\newblock On the probability of the extinction of families.
\newblock \emph{The Journal of the Anthropological Institute of Great Britain
  and Ireland}, 4:\penalty0 138--144, 1875.

\end{thebibliography}
\end{document}